\documentclass[11pt, a4paper]{article}
\usepackage{amsmath,amsthm,stackengine}
\usepackage{amsfonts}
\usepackage{mathrsfs}
\usepackage{mathtools}
\usepackage{hyperref}
\usepackage{amssymb}
\usepackage{epsfig}
\usepackage{dsfont}
\usepackage{stmaryrd}
\usepackage{mathtools}
\usepackage[cal=cm]{mathalfa}
\usepackage{float}
\restylefloat{table}
\restylefloat{figure}
\usepackage{graphicx}
\usepackage{bm}
\usepackage[nottoc]{tocbibind}
\usepackage[top=1in,bottom=1in,left=1.25in,right=1.25in]{geometry}
\usepackage{enumitem}
\usepackage{pgfplots}
\usepackage{pgfplotstable}
\usepackage{lipsum}
\usepackage{siunitx}
\usepackage[belowskip=-20pt]{caption} 
\usepackage{tikz}
\usetikzlibrary{arrows,shapes.geometric}
\usetikzlibrary{patterns}
\usepackage{pst-eucl}
\usepackage{multido,fp}
\pgfplotsset{width=7cm,compat=1.8}
\usepackage{caption}
\usepackage{subcaption}
\usepackage{fancyhdr}
\usepackage{bbold}
\hypersetup{hidelinks}
\usepackage[numbers]{natbib} 

\def\Xint#1{\mathchoice
	{\XXint\displaystyle\textstyle{#1}}%
	{\XXint\textstyle\scriptstyle{#1}}%
	{\XXint\scriptstyle\scriptscriptstyle{#1}}%
	{\XXint\scriptscriptstyle\scriptscriptstyle{#1}}%
	\!\int}
\def\XXint#1#2#3{{\setbox0=\hbox{$#1{#2#3}{\int}$ }
		\vcenter{\hbox{$#2#3$ }}\kern-.6\wd0}}

\def\dashint{\Xint-}
\newtheorem{proposition}{Proposition}[section]
\newtheorem{theorem}[proposition]{Theorem}

\newtheorem{example}[proposition]{Example}
\newtheorem*{lemma*}{Lemma}

\newtheorem*{theorem*}{Theorem}
\theoremstyle{definition}
\newtheorem{definition}[proposition]{Definition}

\newtheorem{remark}{Remark}
\numberwithin{equation}{section}

\newcommand{\T}{\mathcal{T}}
\newcommand{\e}{\mathcal{E}}
\newcommand{\tT}{\widehat{\mathcal{T}}}
\newcommand{\eE}{\widehat{\mathcal{E}}}
\newcommand{\p}{\mathcal{P}}
\newcommand{\M}{\mathcal{M}}
\newcommand{\A}{\mathcal{A}}
\newcommand{\cB}{\mathcal{B}}
\newcommand{\cL}{\mathcal{L}}
\newcommand{\pid}{\Pi^\nabla}
\newcommand{\cK}{\mathcal{K}}
\newcommand{\pw}{\mathrm{pw}}
\newcommand{\ba}{\textbf{A}}
\newcommand{\bb}{\textbf{b}}
\newcommand{\bs}{\boldsymbol{\sigma}}
\newcommand{\dv}{\mathrm{div}}
\newcommand{\B}{\mathrm{B}}
\newcommand{\J}{\mathrm{J}}
\newcommand{\s}{\mathrm{s}}
\newcommand{\cS}{\mathrm{S}}
\newcommand{\apx}{\mathrm{apx}}
\newcommand{\pf}{\mathrm{PF}}
\newcommand{\cb}{\mathrm{b}}
\newcommand{\I}{\mathrm{I}}
\newcommand{\F}{\mathrm{F}}
\newcommand{\h}{\mathrm{H}}
\newcommand{\stab}{\mathrm{stab}}
\title{Conforming VEM for general second-order elliptic problems  with rough data on polygonal meshes and its application to a Poisson inverse source problem}
\author{Rekha Khot\thanks{School of Mathematics, Monash University, Clayton, VIC, 3800. Email: Rekha.Khot@monash.edu}\footnotemark[1], \;  Neela Nataraj\thanks{Department of Mathematics, Indian Institute of Technology, Powai, Mumbai, 400076 Email: neela@math.iitb.ac.in, niteshnirania@gmail.com}\footnotemark[2],\; and Nitesh Verma\footnotemark[2]}
\date{}	
\counterwithin{figure}{section}
\counterwithin{table}{section}

\begin{document}
	\maketitle
	\begin{abstract}
This paper focuses  on the analysis of conforming virtual element methods  for general second-order linear elliptic problems with rough source terms and applies it to a Poisson inverse source problem with rough measurements. For the forward problem, when the source term belongs to $H^{-1}(\Omega)$, the  right-hand side for the discrete approximation defined through polynomial projections is not meaningful even for standard conforming virtual element method. The modified discrete scheme in this paper  introduces a novel companion operator in the context of conforming virtual element method and allows data in $H^{-1}(\Omega)$.
This paper has {\it three} main contributions.  The {\it first} contribution is the  design of a conforming companion operator $J$ from the {\it conforming virtual element space}  to the Sobolev space $V:=H^1_0(\Omega)$, a modified virtual element scheme, and  the \textit{a priori} error estimate for the Poisson problem  in the best-approximation form without data oscillations. 
The {\it second} contribution is the extension of the \textit{a priori} analysis to general second-order elliptic problems with source term in $V^*$. The {\it third} contribution is an application of the companion operator in a Poisson inverse source problem when the measurements belong to $V^*$. The Tikhonov's regularization technique regularizes the ill-posed inverse problem, and the conforming virtual element method approximates the regularized problem given a finite measurement data. The inverse problem is also discretised using the conforming virtual element method and error estimates are established. Numerical tests on different polygonal meshes  for  general second-order problems, and for a Poisson inverse source problem with finite measurement data verify the theoretical results.  
\end{abstract}

	\section{Introduction}
	 The finite element method (FEM) is  the most widely used numerical method to solve boundary value problems governed by partial differential equations in applied science and engineering models. Although the standard definition of \textit{finite element in the sense of Ciarlet} \cite{ciarlet1978finite} allows elements having more general shapes, triangles/quadrilaterals in 2D and the higher-dimensional equivalents  are more popular in the literature.   The last decade has witnessed a significant advancement in the development of  discretisation methods that allow polygonal meshes. There are various  such methods like  
polygonal finite element methods (PFEM) \cite{talischi2010polygonal},  virtual element methods (VEM) \cite{beirao2013basic}, hybrid high-order methods (HHO) (\cite{di2019hybrid}, \cite[Chapter~39]{ guermond2021finite}, and references therein) discontinuous Galerkin method (DGFEM) \cite{cangiani2014hp},  hybridisable discontinuous Galerkin method (HDG) \cite{cockburn2009unified}, and so on. 
\par
	 Virtual element method (VEM)  can be regarded as a generalization of the finite element method (FEM) to arbitrary element-geometry. It is one of the well-received polygonal methods, and the main advantages are the mesh flexibility and a common framework for higher approximation orders. Like FEM is a variational formulation of finite difference method, VEM is a variational analogue of the mimetic finite difference method \cite{da2014mimetic}.  The virtual element space is locally a set of solutions to some problem dependent partial differential equation, and  contains polynomials as
well as possibly non-polynomial functions. The key idea of VEM lies in the fact that it
does not demand the explicit construction of complicated shape functions (hence the name virtual) and
the knowledge of degrees of freedom along with suitable projections of virtual element functions onto polynomial
subspace is sufficient to analyse and implement the method.  Since the explicit computation of virtual element functions is not feasible,  the  bilinear forms  comprise of    consistency  and stabilization terms. The consistency
term is built on approximating  discrete functions by computable projections  on the polynomial subspace and the stabilization term is to ensure the
stability of the discrete bilinear form.   VEM has been applied to a wide range of model problems (\cite{MR3507277, beirao2013basic, MR3671497} to name a few) in the last decade.
\par  The conforming VEM for a Poisson model  problem $-\Delta u = f$ that seeks $u\in H^1_0(\Omega)$ is investigated in the very first paper on VEM \cite{beirao2013basic}, and applied to general second-order linear elliptic problems in \cite{MR3460621,MR3671497} for $f\in L^2(\Omega)$. The  companion operators also referred to as {\it smoothers} are introduced in the literature in the context of nonconforming FEM and VEM approximations, and/or when the source is rough \cite{MR3194809,MR2142191,MR4235819,MR3907927}. Ern \emph{et al.} apply HHO method   for the Poisson equation with loads in $H^{-1}(\Omega)$ \cite{MR4167044} and the analysis is based on  a smoother  designed using averaging and bubble functions.  An enrichment  operator from nonconforming to conforming virtual element spaces is constructed in \cite{huang2021medius} for Poisson and biharmonic problems, and a slight variation  of this operator is provided in \cite{adak2022morley} for fourth-order problems. Since conforming VEM functions are also \emph{not computable} through degrees of freedom, such operators  can only be used for the purpose of analysis and not to approximate the rough data.   Carstensen \emph{et al.} design a computable companion operator  in  \cite{carstensen2022nonconforming, MR4444835}  for   the nonconforming VEM that applies to second and fourth-order problems. 
 
\par Even in the conforming case, {\it the virtual element} spaces involve locally polynomials as well as non-polynomial functions and  the discrete problem is defined through computable polynomial projections. Though the discrete space is contained in the continuous space, and the method is referred to as \textit{conforming}  from this perspective,  the discrete problem in VEM involves discontinuous polynomial projections. Hence conforming VEMs are more challenging than conforming FEMs, especially when the source term belongs to $V^*$.  
This motivates us to design a companion operator for the conforming VEM, for the  first time in the literature to the best of our knowledge, to handle  {\it rough source terms  in the discrete scheme}. This paper first discusses conforming VEM for the Poisson problem and then extends it for  general second-order elliptic problems. The difficulties arising from non-symmetric and non-coercive bilinear forms  are addressed successfully.  This analysis is of independent interest to other research problems. 

\par

The inverse source problems  are majorly found in the practical examples like electromagnetic theory and crack determination. There are different types of  inverse problems \cite{MR3524926} like inverse source problems \cite{MR3245124}, parameter identification  problems \cite{MR4399830,nair2017linear} etc., and in this paper we aim to determine the source function  from a finite density field of measurements following the model problem in \cite{MR3245124}.  In general, the main drawback is the ill-posed behavior of these problems, e.g.,  for displacement $u_n = 1/n \sin(n \pi x) \sin(n \pi y)$ on $[0,1]^2$ and for the force  $f_n=2n \pi^2 \sin(n \pi x)\sin(n \pi y)$, $u_n \in H^1(\Omega)$ converges to zero whereas $f_n \notin L^2(\Omega)$ diverges as $n \to \infty$. It is well-known that this can be overcome with the regularization techniques.
\par Huhtala \emph{et al.} \cite{MR3245124} analyse  conforming FEM for a Poisson inverse source problem for measurements in $V^*$, {but the numerical experiments consider only $L^2$ measurements}.  Nair \emph{et al.} \cite{nair2021conforming} investigate conforming and nonconforming FEM for the biharmonic problem with   $L^2$ measurement functionals. The analysis therein for nonconforming FEM doesn't cover the important situation when the source term or the measurements belong to $V^*$. 
\par The analysis of the inverse problem heavily relies on that of the corresponding forward problem and to the best of our knowledge, {\it VEM for inverse problems} has not been studied in the literature.   
Moreover, the measurement functionals belong to $V^*$ and the techniques to analyze VEM heavily depend on a novel conforming companion operator. This motivates us to discuss this problem as an application of the  analysis of  the forward problem with rough right-hand sides. {We numerically investigate  the point measurement example, which is also novel in the  conforming FEM   for a Poisson inverse source problem \cite{MR3245124}.}   The ideas developed in this paper can be extended to the nonconforming FEM \cite{nair2021conforming} for the case of rough measurement functionals. The main contributions of this paper are stated below.
\begin{itemize}
    \item The first part of the paper
    \begin{itemize}
    \item constructs a novel computable companion  from the conforming virtual element space $V_h^k$ to $V=H^1_0(\Omega)$ for the general degree $k\in \mathbb{N}$, 
    \item offers a modified VEM  that introduces companion in the discrete source approximation in comparison to the standard VEM that uses polynomial projection,
    \item analyses the conforming VEM for Poisson problem with rough data in $V^*$ and extends it to general second-order linear elliptic problems,
    \item proves energy and $L^2$ error estimates in the best-approximation form without data oscillations for the choice of a smoother.
    \end{itemize}
    \item The second part of the paper
    
    \begin{itemize}
   \item deals with a Poisson inverse source problem that seeks  $f$ 
   given rough measurements $h_i\in V^*$ of the solution $u$ of the forward problem,
    \item   approximates the regularized solution utilizing the conforming VEM and proves  error estimates.
    \end{itemize}
\end{itemize}
Illustrative  numerical experiments  confirm the theoretical convergence rates for both forward and inverse problems. 

\par The paper is organized as follows: Section~2  introduces  the Poisson problem with a  rough source term. This section presents the admissible meshes and a construction of the companion operator for the conforming VEM, introduces  the standard and modified  discrete schemes, and discusses the error analysis for the forward problem. Section~3 analyses the conforming VEM for general second-order linear elliptic problems with rough data. The error estimates in the energy and $L^2$ norms are presented. Section~4 deals with a Poisson inverse source problem and the regularized solution is approximated using the conforming VEM. The numerical experiments  in Section~5 for both forward and inverse problems show  empirical convergence  rates. 
\par Standard notation on Lebesgue and Sobolev spaces  and norms applies throughout this paper, e.g.,  $\|
	\cdot\|_{s,\cal{D}}$ (resp. seminorm $|\cdot|_{s,\cal{D}}$) for $s\geq0$ denotes norm on the Sobolev space  $H^s(\mathcal{D}):=H^s(\text{int}(\cal{D}))$  of order $s\in\mathbb{R}$ defined in the interior $\text{int}(\mathcal{D})$ of a  domain $\mathcal{D}$, while $(\cdot,\cdot)_{L^2({\cal D})}$ and $\|\cdot\|_{L^2({\cal D})}$  denote the $L^2$ scalar product and $L^2$ norm in  ${\cal D}$. The Euclidean norm of a vector in $\mathbb{R}^N$ for $N\in\mathbb{N}$ is denoted by $\|\cdot\|$. Given a barycenter $x_{\mathcal{D}}$ and diameter $h_{\mathcal{D}}$ of a  domain $\mathcal{D}$, define the set of scaled monomials $\M_k(\mathcal{D})$ of degree less than equal to $k$ and $\M_k^*(\mathcal{D})$ of degree equal to $k$ by
$$\M_k(\mathcal{D})= \Big\lbrace \Big(\frac{x-x_{\mathcal{D}}}{h_{\mathcal{D}}} \Big)^\beta: |\beta| \le k \Big\rbrace, \text{ and } \mathcal{M}^*_k(\mathcal{D})= \Big\lbrace \Big(\frac{x-x_{\mathcal{D}}}{h_{\mathcal{D}}} \Big)^\beta: |\beta| = k \Big\rbrace.$$ Let $\p_k({\cal D})$ denote  the set of polynomials of degree at most $k\in\mathbb{N}_0$ defined on a domain ${\cal D }$ and  $\p_k({\T_h})$ denote the set of piecewise polynomials on an admissible partition $\T_h\in\mathbb{T}$ (defined in Subsection~2.2). The piecewise seminorm and norm  in $H^s(\T_h)$ for $s\in\mathbb{R}$   read
	$|\cdot|_{s,\text{pw}}:=\big(\sum_{P\in\T_h}| \cdot|_{s,P}^2\big)^{1/2}$ and $\|\cdot\|_{s,\text{pw}}:=\big(\sum_{P\in\T_h}\| \cdot\|_{s,P}^2\big)^{1/2}$. The generic constants are denoted in the sequence $C_1,C_2,\dots$ and the constants that depend on standard inequalities are specifically  defined, e.g., the constant $C_\pf$ comes from  Poincar\'e-Friedrichs inequality. 
	\section{Virtual element method for Poisson problem with rough source}
 Let $u\in V:=H^1_0(\Omega)$ solve the Poisson equation $-\Delta u =f$  for a given source field $f\in V^*$ on a polygonal subdomain $\Omega\subset\mathbb{R}^2$ with a boundary $\partial \Omega$. The solution operator $\cK:V^*\to V$ defines the Riesz representation $\cK f\in V$ for a given $f\in V^*$ with
	\begin{align}
	    a(\cK f,v) = f(v)\quad\text{for all}\;v\in V\label{eq:weak-fwd}
	\end{align}
	for an inner product $a(\cdot,\cdot):=(\nabla\cdot,\nabla\cdot)_{L^2(\Omega)}$ on $V\times V$ (with piecewise version denoted by $a_\pw$ throughout the paper).
	\par This section has five subsections.  Subsection~2.1 states  two conditions \ref{M1}-\ref{M2} for  admissible polygonal meshes. Subsection~2.2  introduces the virtual element spaces $V_h^k$ and Subsection~2.3 designs the companion operator $J:V_h^k\to V$ and establishes its properties.  The discrete problem is presented in Subsection~2.4 and the error estimate in energy norm is proved in the best-approximation form in Subsection~2.5.

	\subsection{Polygonal meshes}
	 The virtual element method allows fairly general polygonal meshes. Let $\mathbb{T}$ be a family of decomposition of $\overline{\Omega}$ into polygonal subdomains  satisfying the two mesh conditions \ref{M1}-\ref{M2} with a universal positive constant $\rho$ \cite{beirao2013basic}.
\begin{enumerate}[label={(\bfseries M\arabic*)}] \item\label{M1} \textbf{Admissibility}.  Any two distinct polygonal subdomains $P$ and $P'$ in $\T_h\in\mathbb{T}$ are disjoint or share   a  finite number of edges  and vertices. 
\begin{figure}[H]
		\centering
		\includegraphics[width=0.3\linewidth]{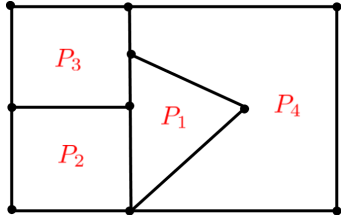}               
	\caption{Decomposition of a rectangular domain into polygonal subdomains $P_1,\dots,P_4$.}
	\label{fig:P}
	\end{figure}
Note that   one of the advantages of polygonal meshes is that hanging nodes are seamlessly incorporated in the mesh and treated as just another vertex of a polygonal subdomain.  Figure~\ref{fig:P} displays an example of a rectangular domain divided into  non-uniform polygonal subdomains and  observe that $P_1$ looks like a triangle but is considered as a quadrilateral, $P_2$ is a rectangle, $P_3$ is a pentagon (looks like a rectangle), and $P_4$ is a hexagon (nonconvex element).
\item \label{M2} \textbf{Mesh regularity}.   Every polygonal subdomain $P$ of diameter $h_P$ is star-shaped with respect to every point of a ball of radius greater than equal to $\rho h_P$ and every edge $E$ of $P$ has a length $h_E$ greater than equal to $\rho h_P$.

\end{enumerate}
Here and throughout this paper,  $h_{\T_h}|_P:=h_P:=\text{diam}(P)$ denotes the piecewise constant mesh-size  $h_{\T_h}\in\p_0(\T_h)$  and $h:=\max_{P\in\T}h_P$ denotes the maximum diameter over all $P\in\T_h\in\mathbb{T}$. Let $\mathcal{V}(P)$ (resp. $\mathcal{V}$) denote the set of vertices  of $P$ (resp. of $\T_h$) and let $\e(P)$ (resp. $\e$) denote the set of edges  of $P$ (resp. of $\T_h$). Denote the interior and boundary edges of $\T_h$ by $\e(\Omega)$ and $\e(\partial\Omega)$.  Let $|\mathcal{V}|$ (resp. $|\e|$) denote the number of vertices (resp. edges) of $\T_h$, and $N_P$ denote the number of vertices of a polygonal subdomain $P$.  
\subsection{Virtual element spaces}\label{sec:vspace}
For any $v\in H^1(P)$, its $H^1$ elliptic projection operator $\pid_k:H^1(P)\to\p_k(P)$ is denoted as $\pid_k v\in\p_k(P)$ and is defined by
\begin{align}
    (\nabla \pid_k v,\nabla \chi_k)_{L^2(P)} = (\nabla v,\nabla\chi_k)_{L^2(P)}\quad\text{for all}\;\chi_k\in\p_k(P)\label{pid}
\end{align}
with an additional condition (to fix the constant) 
\begin{align}
    \frac{1}{N_P}\sum_{j=1}^{N_P}\pid_k v(z_j)& = \frac{1}{N_P}\sum_{j=1}^{N_P} v(z_j) \quad\text{for}\; k=1\;\text{and}\;z_j\in\mathcal{V}(P),\\\dashint_P \pid_k v\,dx &= \dashint_P v\,dx\quad\text{for}\; k\geq 2.
\end{align}
It follows easily from  \eqref{pid} that $\pid_k$ is stable with respect to $H^1$ norm. That is, 
\begin{align}|\pid_k v|_{1,\pw}\leq | v|_{1,\Omega}. \label{pid_stab}\end{align}
Let $\Pi_k$  denote the $L^2$ projection on $\p_k(\T_h)$ for $k\in\mathbb{N}_0$. In other words, for any $v\in L^2(\Omega)$,
\begin{align}(\Pi_kv,\chi_k)_{L^2(\Omega)}=(v,\chi_k)_{L^2(\Omega)}\quad\text{ for all}\quad \chi_k\in \p_k(\T_h).\label{def:L2}\end{align}  An immediate consequence is $L^2$ stability of $\Pi_k$, that is,
\begin{align}
    \|\Pi_k v\|_{L^2(\Omega)}\leq \|v\|_{L^2(\Omega)}.\label{est:L2}
\end{align}
\begin{proposition}[polynomial projection \cite{brenner2008mathematical}] \label{prop:est-polyPi}
	For a sufficiently smooth function $v\in  H^{s}(P)$, for  $1 \le s \le k+1$, and $P \in \T_h$, there exists a positive constant $C_{\mathrm{apx}}$ (that depends exclusively  on $\rho$ from \ref{M2}) such that
	\begin{align*}
	\|v - \pid_k v \|_{d,P} + \|v - \Pi_k v \|_{d,P} \leq C_{\mathrm{apx}} h_P^{s-d}|v|_{s,P}\quad\text{for}\quad 0\leq d\leq s.
	\end{align*} 
\end{proposition}
 The local  conforming virtual element space is defined as a set of solutions to  the Poisson equation with Dirichlet boundary condition, and an enhanced space is constructed with an additional orthogonality condition in the original conforming virtual element space so that the $L^2$ projection $\Pi_k$ is computable cf. \cite{MR3073346}. In particular, the local enhanced virtual element space $V_h^k(P)$ is
\begin{align}
    V_h^k(P):=\begin{rcases}
    \begin{dcases}
    &v_h\in H^1(P): \Delta v_h\in\p_{k}(P), \; v_h|_E\in\p_k(E)\quad\text{for all}\; E\in\e(P),\\& v_h\in C^0(\partial P),\; (v_h-\pid_kv_h,\chi)_{L^2(P)}=0\;\text{for all}\;\chi\in\mathcal{M}^*_{k-1}(P)\cup\mathcal{M}^*_{k}(P)\label{local_sp}
    \end{dcases}
    \end{rcases}.
\end{align}
\begin{remark}[lowest-order case $k=1$]
The local virtual element space \eqref{local_sp} for $k=1$ reduces to
\begin{align*}
    V_h^1(P):=\begin{rcases}
    \begin{dcases}
    &v_h\in H^1(P): \Delta v_h\in\p_{1}(P), \; v_h|_E\in\p_1(E)\quad\text{for all}\; E\in\e(P),\\& v_h\in C^0(\partial P),\;(v_h-\pid_1v_h,\chi)_{L^2(P)}=0\quad\text{for all}\;\chi\in\mathcal{M}_{1}(P)
    \end{dcases}
    \end{rcases}.
\end{align*}
The orthogonality condition in $V_h^1(P)$ shows that $\pid_1=\Pi_1$. The freedom of $\Delta v_h$ to be any linear polynomial in $P$ is suppressed by the orthogonality condition, and hence the dimension of the space is $N_P$. If $P$ is a triangle, then it coincides with the Lagrange finite element space.\end{remark}
\noindent The local degrees of freedom (dofs) for $v_h\in V_h^k(P)$ are
\begin{itemize}
	\item values of $v_h$ at the vertices of $P$,
	\vspace{-0.25cm}
	\item values of $v_h$ at the $k-1$ interior Gauss-Lobatto points on each edge $E \in \partial P$,
	\vspace{-0.25cm}
	\item  
	$\dashint_{P} v_h \;\chi_{k-2}\; dx \quad \text{for all} \; \chi_{k-2} \in \M_{k-2}(P).$
\end{itemize} 
\begin{figure}[H]
	\centering
	\begin{subfigure}{.33\textwidth}
		\centering
		\includegraphics[width=0.7\linewidth]{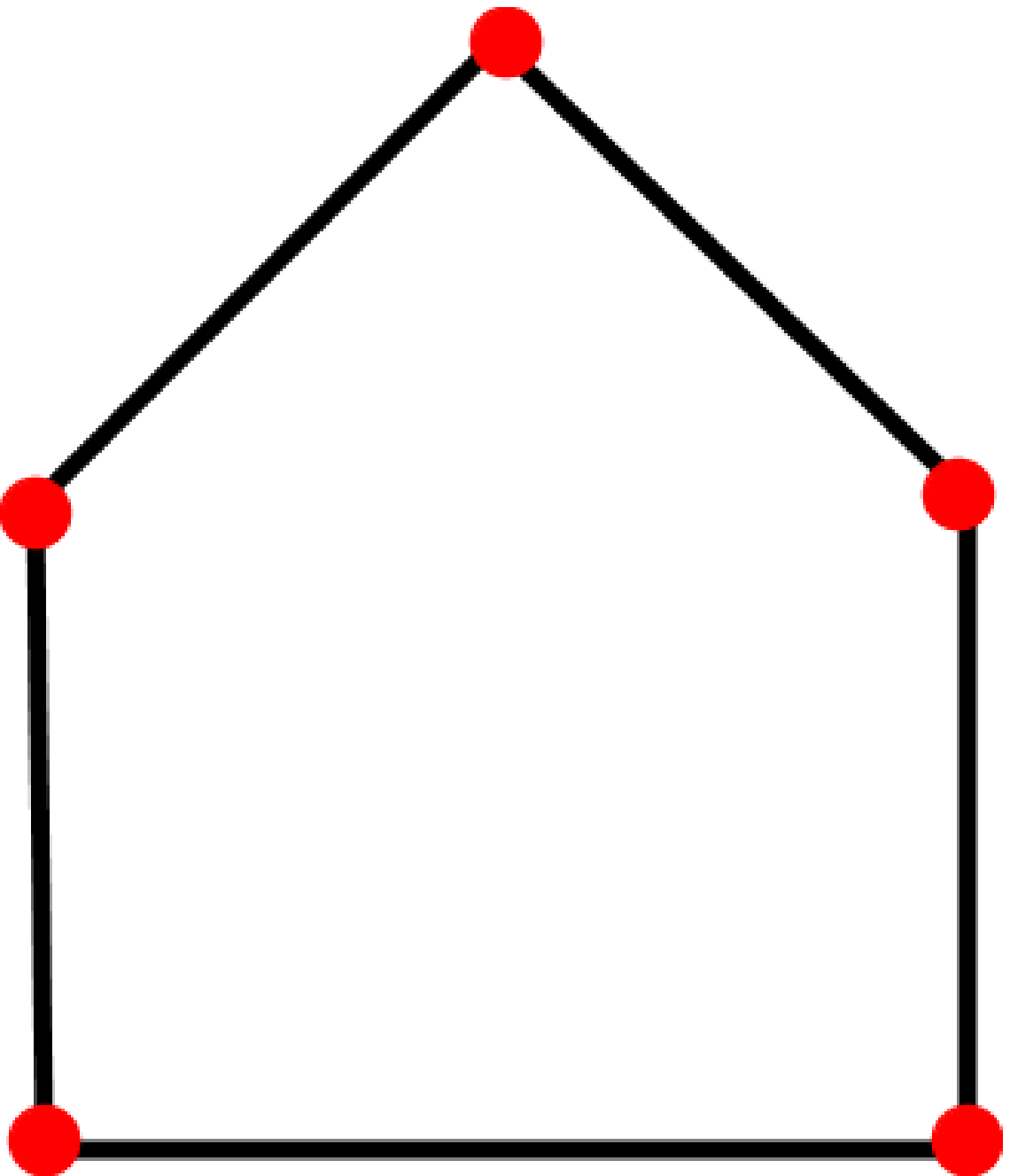}
	\end{subfigure}%
	\begin{subfigure}{.33\textwidth}
		\centering
		\includegraphics[width=0.7\linewidth]{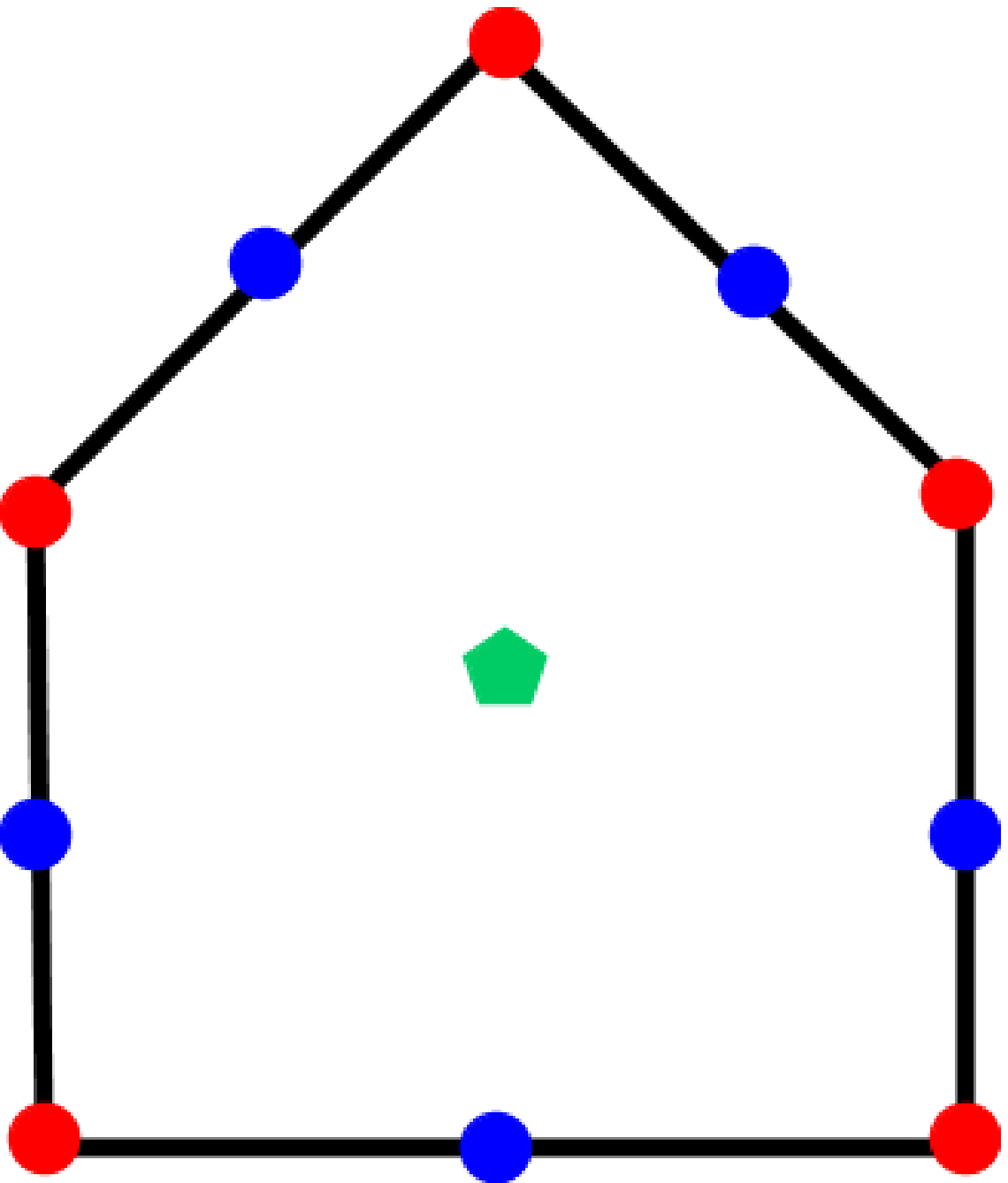}
	\end{subfigure}%
	\begin{subfigure}{.33\textwidth}
		\centering
		\includegraphics[width=0.7\linewidth]{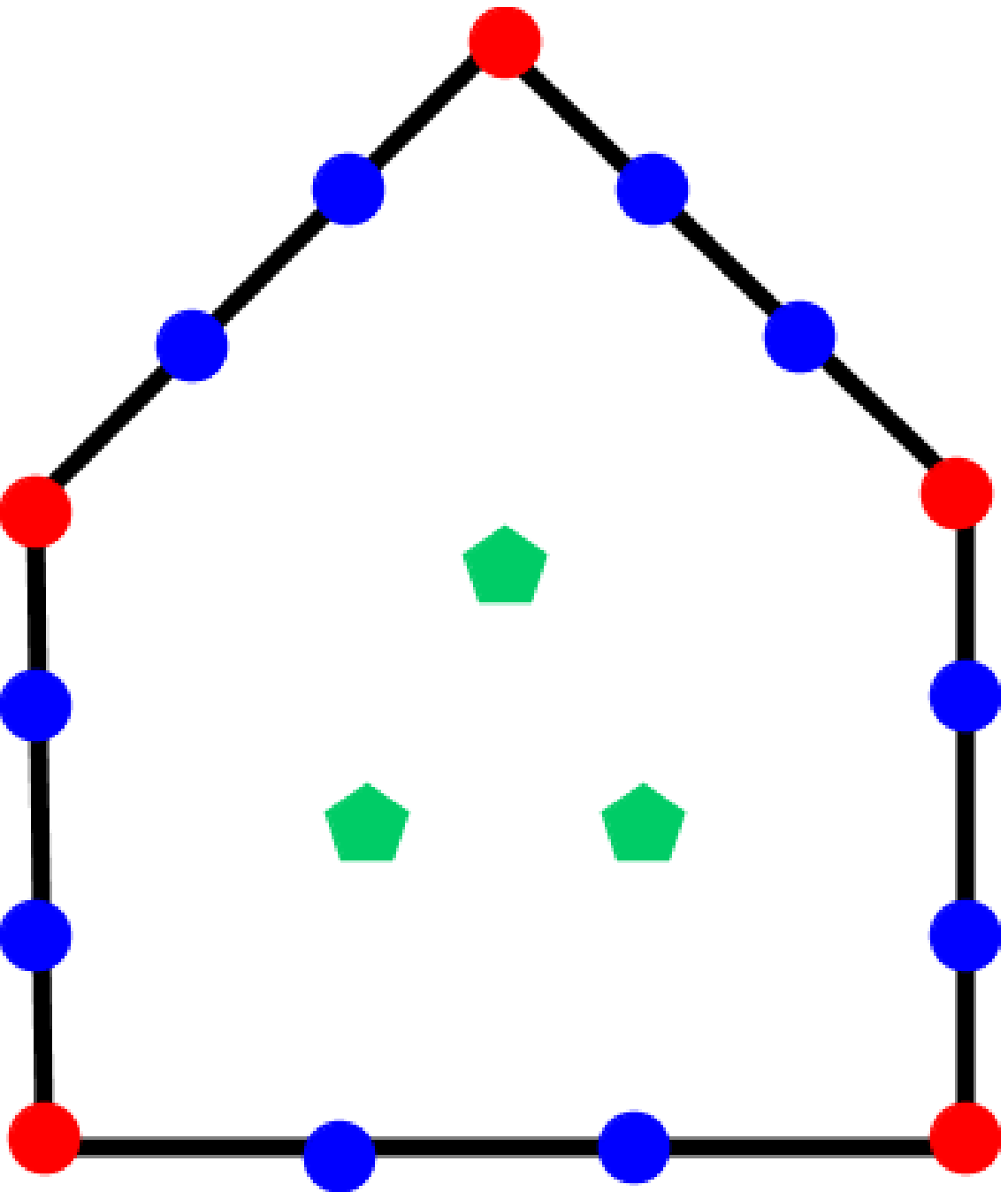}
	\end{subfigure}
	\caption{The degrees of freedom for $k=1, 2,$ and 3.}
	\label{fig1}
\end{figure}
We enumerate the above mentioned dofs that are linear functionals as $\text{dof}_1,\dots,\text{dof}_{N_{\text{dof}}}$ for $N_{\text{dof}}:=kN_P+\frac{k(k-1)}{2}$ and  the triplet $(P, V_h^k(P), (\text{dof}_1,\dots,\text{dof}_{N_{\text{dof}}}))$ forms a finite element in the sense of Ciarlet \cite{MR3073346}.   The global virtual element space is defined as
\begin{align*}
    V_h^k :=\{v_h\in H^1_0(\Omega): \forall P\in\T_h \quad v_h\in V_h^k(P)\}.
\end{align*}
\begin{definition}[Interpolation operator]
    Let $\varphi_1,\dots,\varphi_{N_{\text{dof}}}$ be the nodal basis functions of $V_h^k$.  Given $v\in V$, define its interpolation $I_hv\in V_h^k$ by
    \[I_hv:=\sum_{j=1}^{N_{\text{dof}}}\text{dof}_j(v)\varphi_j.\]
\end{definition}
\noindent For any $\chi_k\in\p_k(P)$ and any $P\in\T_h$, this and  an integration by parts show
\begin{align*}
    (\nabla(v-I_hv),\nabla\chi_k)_{L^2(P)}= (v-I_hv,\Delta\chi_k)_{L^2(P)} +(v-I_hv,\partial_\textbf{n}\chi_k)_{L^2(\partial P)}=0.
\end{align*}
In other words,  
\begin{align}
\nabla(v-I_hv)\perp(\p_{k-1}(\T_h))^2\quad\text{in}\;(L^2(\Omega))^2.\label{int:ortho}
\end{align}
\begin{proposition}[Interpolation estimates \cite{MR3671497}] \label{prop:est-inter-forward}
	For every $v\in V \cap H^{s}(\Omega)$, the interpolant $I_hv \in V_h^k$ of $v$ satisfies
	\begin{align*}
	\| v - I_hv\|_{d,\Omega} \le C_\I h^{s-d} |v|_{s,\Omega} \qquad \text{for } 0\leq d\leq s\;\text{and}\; 1\leq s \le k+1.
	\end{align*}
\end{proposition}
\begin{proposition}[Poincar\'e-Friedrichs inequality \cite{brenner2008mathematical}]\label{PF}
  Let $P\in\T_h$ be a polygonal domain with vertices $z_1,\dots,z_{N_P}$. For any $v\in H^1(P)$ with $\frac{1}{N_P}\sum_{j=1}^{N_P}v(z_j)=0$ or $\int_P v\,dx=0$, there exists a positive constant $C_\pf$ that exclusively depend on $\rho$ from \ref{M2} with   
  \[\|v\|_{L^2(P)}\leq C_\pf h_P|v|_{1,P}.\]
\end{proposition}
\subsection{Companion operator}
 It is clear from the definition of $V_h^k(P)$ in \eqref{local_sp} that determining an explicit expression for the discrete functions $v_h\in V_h^k(P)$ is not feasible and hence we need  computable quantities to define the discrete problem.  Hence the standard VEM utilizes $\pid_kv_h$, but $f(\pid_kv_h)$ is well-defined for $f\in L^2(\Omega)$ and  not in general for $f\in V^*=H^{-1}(\Omega)$. This section presents a computable companion operator $J:V_h^k\to V$ and enables  to  introduce $J$ in the discrete problem when $f\in V^*$. 
 \begin{figure}[H]
	\centering
	\begin{subfigure}{.5\textwidth}
		\centering
		\includegraphics[width=0.7\linewidth]{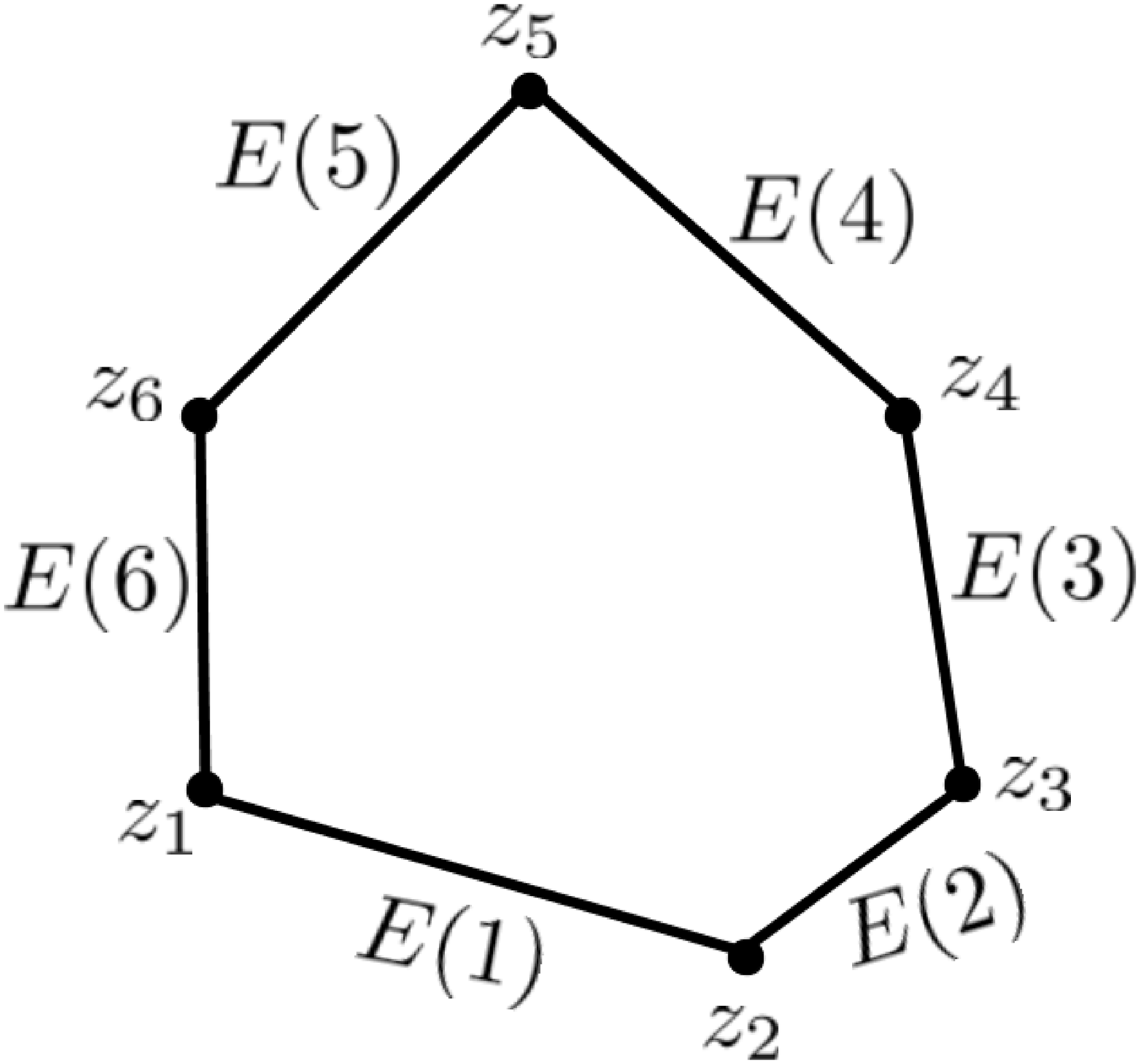}
	\end{subfigure}%
	\begin{subfigure}{.5\textwidth}
		\centering
		\includegraphics[width=0.7\linewidth]{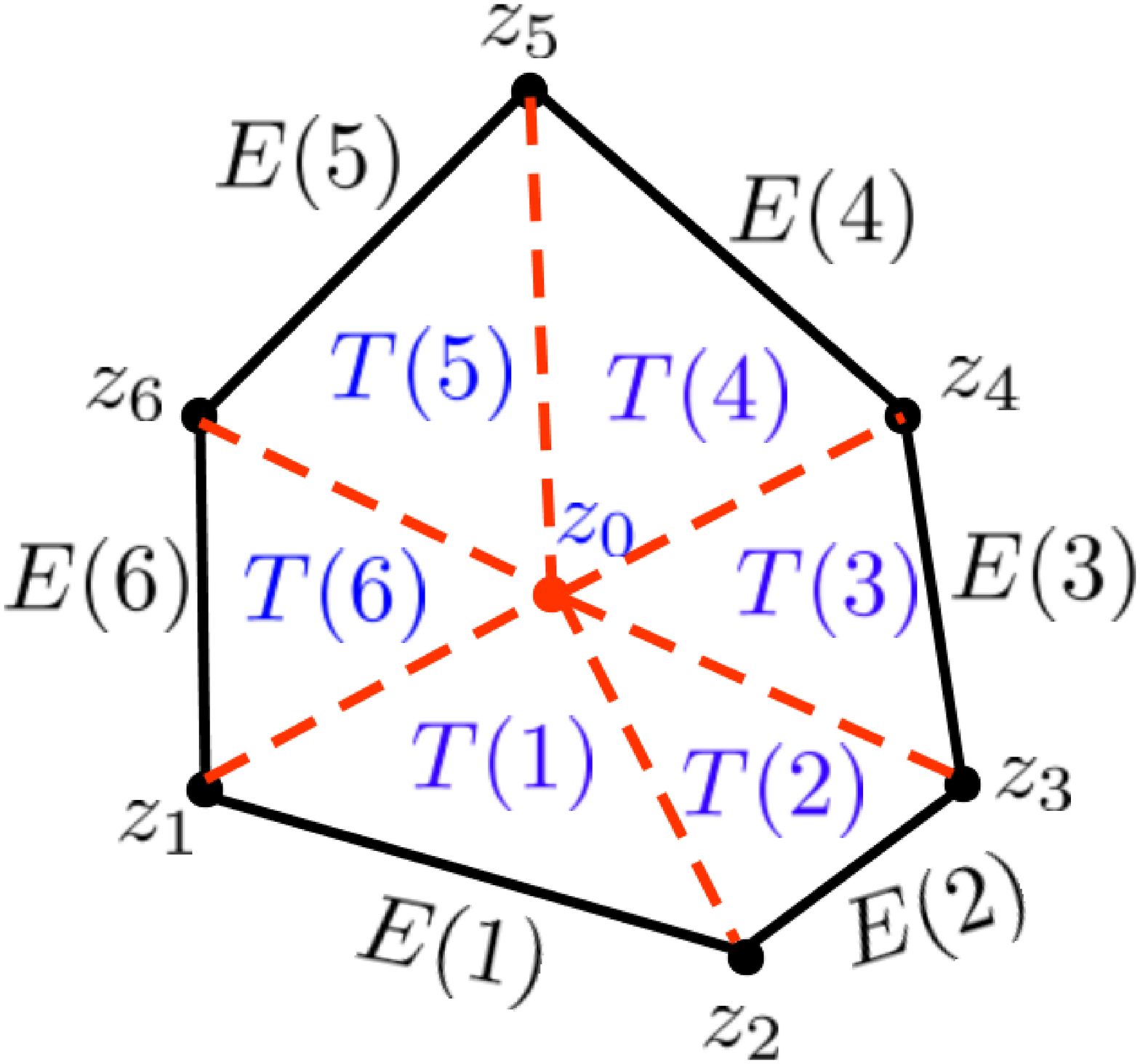}
	\end{subfigure}
		\caption{Polygon $P$ (left) and its sub-triangulation $\tT_h(P):=T(1)\cup\dots T(6)$ (right).}
	\label{fig:TP}
\end{figure}
 \par Enumerate the vertices $\mathcal{V}(P):=\{z_1,\dots,z_{N_P}\}$ and edges $\e(P):=\{E(1),\dots,E(N_P)\}$ consecutively, that is, $E(j)=\text{conv}\{z_j,z_{j+1}\}$ for $j=1,\dots,N_P$ with $z_{N_P+1}:=z_1$ and enumerate $z_1,\dots,z_{N_P}$ counterclockwise along the boundary $\partial P$.  \ref{M2} implies that  each polygonal subdomain $P\in\T_h$  can be divided into triangles $T(j):=\text{conv}\{z_0,z_j,z_{j+1}\}$  for all $j=1,\dots,N_P$ and for the midpoint $z_0$ of the ball from \ref{M2} in Figure~\ref{fig:TP}. It is known \cite{brenner2017some} that the resulting sub-triangulation $\tT_h|_P:=\tT_h(P):=\cup_{j=1}^{N_P}T(j)$ of $P\in\T_h$ is uniformly shape-regular; that is, the minimum angle in each triangle $T\in\tT_h(P), P\in\T_h\in\mathbb{T}$, is bounded below by some positive constant $w_0>0$ that exclusively depends on $\rho$ from \ref{M2}. Let  $\widehat{\mathcal{V}}$ (resp. $\widehat{\mathcal{V}}(P)$) denote the set of vertices and   $\eE$ (resp. $\eE(P)$) denote the set of edges  in  $\tT_h$ (resp. $\tT_h(P)$). 

	Let $\widehat{V}_h^k:=\{v_h\in H^1_0(\Omega) : \forall T\in\tT_h\quad v_h|_T\in\p_k(T)\}$ denote the Lagrange finite element space for general $k$ on the sub-triangulation $\tT_h$, and  the local degrees of freedom  for $v_h\in \widehat{V}_h^k(T)$ and $T\in\tT_h$ are
\begin{itemize}
	\item values of $v_h$ at the vertices of $T$,
	\vspace{-0.25cm}
	\item values of $v_h$ at the $k-1$ interior Gauss-Lobatto points on each edge $E \in \partial T$,
	\vspace{-0.25cm}
	\item  
	$\dashint_{T} v_h \;\chi_{k-3}\; dx \quad \text{for all} \; \chi_{k-3} \in \M_{k-3}(T).$
\end{itemize}   
\begin{theorem}[computable companion in conforming VEM]\label{thm:J}
Given  the virtual element space $V_h^k$ of the degree $k\geq 1$, there exist a linear operator $J : V_h^k\to V$  and a positive constant $C_\J$  that depends exclusively on $\rho$ from \ref{M2} such that
\begin{enumerate}[label=$(\alph*)$]
    \item\label{a} $Jv_h = v_h \quad\text{on}\;\partial P$ for all $P\in\T_h$,
    \item\label{b} $v_h-Jv_h\perp \p_k(\T_h)\quad\text{in}\;L^2(\Omega)$,
    \item\label{c} $\nabla(v_h-Jv_h) \perp  \p_{k-1}(\T_h)^2\quad\text{in}\;(L^2(\Omega))^2$,
    \item\label{d} $\|h_{\T_h}^{-1}(v_h-Jv_h)\|_{L^2(\Omega)}+|v_h-Jv_h|_{1,\Omega}\leq C_\J |v_h-\pid_kv_h|_{1,{\mathrm{pw}}}$.
\end{enumerate}
\end{theorem}
\textit{Design of $J$}. Define $J_1:V_h^k\to \widehat{V}_h^k $, for $v_h\in V_h^k$, by
\begin{align*}
    J_1v_h(z) =\begin{dcases} v_h(z)\quad\text{for all nodes on the boundary of}\; P,
    \\
    \pid_kv_h(z)\quad\text{for all nodes in the interior of}\;P,\end{dcases}
\end{align*}
and,  for all $\chi_{k-3}\in\mathcal{M}_{k-3}(T),\; T\in\tT_h(P)$  and $P\in\T_h$, by
\begin{align*}
    \dashint_T J_1v_h\;\chi_{k-3}\,dx = \dashint_T \pid_kv_h\;\chi_{k-3}\,dx.
\end{align*}
For each $T\in\tT_h(P)$, let $b_T\in W_0^{1,\infty}(T)$ denote the cubic bubble-function $27\lambda_1\lambda_2\lambda_3$ for the barycentric co-ordinates $\lambda_1, \lambda_2, \lambda_3\in\p_1(T)$ of $T$ with $\dashint_Tb_T\,dx=9/20$. Let $b_T$ be extended by zero outside $T$ and, for $P\in\T_h$, define 
	\begin{align}
	b_P:=\frac{20}{9}\sum_{T\in\tT_h(P)}b_T\in W_0^{1,\infty}(P)\subset W_0^{1,\infty}(\Omega)\label{bubble}
	\end{align}
	with 
	$
	\dashint_{P}b_P\,dx=1$. Let $v_P\in\p_k(\T_h)$ be the Riesz representation of the linear functional $\p_k(\T_h)\to\mathbb{R}$ defined by $w_k\mapsto(v_h-J_1v_h,w_k)_{L^2(\Omega)}$ for $w_k\in\p_k(\T_h)$ in the Hilbert space $\p_k(\T_h)$ endowed with the weighted $L^2$ scalar product $(b_P\bullet,\bullet)_{L^2(P)}$. Hence $v_P$ exists uniquely and satisfies $\Pi_k(v_h-J_1v_h) = \Pi_k(b_Pv_P)$. In other words,
	\begin{align}
	    (v_h-J_1v_h,w_k)_{L^2(P)}=(b_Pv_P,w_k)_{L^2(P)}\quad\text{for all}\;w_k\in\p_k(P).\label{2.10}
	\end{align}
	Given the bubble-functions $(b_P:P\in\T_h)$ from \eqref{bubble} and the above functions $(v_P:P\in\T_h)$ for $v_h\in V_h$, define
	\begin{align}
	Jv_h:=J_1v_h+\sum_{P\in\T_h}b_Pv_P\in V.\label{J}
	\end{align} 

\begin{proof}[Proof of \ref{a}]
Since $b_P$ vanishes on $\partial P$, there holds $Jv_h=J_1v_h$ along $\partial P$ for all $P\in\T_h$. The definition of $J_1$ implies that $J_1v_h = v_h\in\p_k(\e(P))$ for all $P\in\T_h$. This concludes the proof of \ref{a}.
\end{proof}
\begin{proof}[Proof of \ref{b}]
For all $w_k\in \p_k(P)$ and any $P\in\T_h$, the definition of $J$ in \eqref{J} and \eqref{2.10}  prove 
\[(v_h-Jv_h,w_k)_{L^2(P)}=(v_h-J_1v_h,w_k)_{L^2(P)}-(b_Pv_P,w_k)_{L^2(P)}=0. \]
This concludes the proof of \ref{b}.
\end{proof}
\begin{proof}[Proof of \ref{c}]
For $\chi_{k-1}\in(\p_{k-1}(P))^2$ and  any $P\in\T_h$, an integration by parts  shows  
\begin{align*}
    (\nabla(v_h-Jv_h),\chi_{k-1})_{L^2(P)} =(v_h-Jv_h,\dv{\chi_{k-1}})_{L^2(P)}+ (v_h-Jv_h,\chi_{k-1}\cdot{\textbf{n}})_{L^2(\partial P)}=0
\end{align*}
with \ref{b} resp. \ref{a} for the first term resp. second term in the last step. This concludes the proof of \ref{c}.
\end{proof}
\begin{proof}[Proof of \ref{d}]
Since the edge and cell moments of $J_1v_h$ are defined in terms of $\pid_kv_h$ inside the interior of $P$,
 the expression of $(\pid_kv_h-J_1v_h)|_P$ in terms of the Lagrange basis functions     leads to
 \begin{align}\label{2.12}
     (\pid_kv_h-J_1v_h)|_P = \sum_{z}(\pid_kv_h-v_h)(z)\phi_z,
 \end{align}
 where  $\phi_z$ is a canonical basis function with respect to the node $z$ along the boundary of $P$.
 The Sobolev inequality \cite{brenner2017some} (with a positive constant $C_\cS$) implies for $v_h-\pid_kv_h\in H^1(P)$ that 
 \begin{align}
     |(v_h-\pid_kv_h)(z)|\leq \|v_h-\pid_kv_h\|_{L^\infty(P)} &\leq C_\cS (h_P^{-1}\|v_h-\pid_kv_h\|_{L^2(P)}+|v_h-\pid_kv_h|_{1,P})\nonumber\\&\leq C_\cS(1+C_{\text{PF}}) |v_h-\pid_kv_h|_{1,P}\label{4.3}
 \end{align}
 with the last step from Proposition~\ref{PF}. The standard scaling arguments \cite{brenner2008mathematical} show that $|\phi_z|_{1,P}$ is bounded. The combination of this and \eqref{4.3} in \eqref{2.12} results in $|\pid_kv_h-J_1v_h|_{1,P}\leq C_\B C_\cS(1+C_{\text{PF}}) |v_h-\pid_kv_h|_{1,P}$. Note that though \eqref{2.12} depends on number of degrees of freedom, it can be uniformly bounded by a constant $C_\B$ independent of $P$.  This and the triangle inequality $|v_h-J_1v_h|_{1,\Omega}\leq |v_h-\pid_kv_h|_{1,\text{pw}}+|\pid_kv_h-J_1v_h|_{1,\text{pw}}$ result in 
 \begin{align}
     |v_h-J_1v_h|_{1,\Omega}\leq (1+C_\B C_\cS(1+C_{\text{PF}})) |v_h-\pid_kv_h|_{1,\text{pw}}.\label{4.5}
 \end{align}
 The definition of $J$ leads to $v_h-Jv_h=(v_h-J_1v_h)-\sum_Pb_Pv_P$ and so it remains to handle the term $|b_Pv_P|_{1,P}$. For any $ \chi\in \p_k(P)$, there exists  a positive constant $C_\cb$  in the inverse estimates \cite{MR3671497,MR4444835} 
	 \begin{align}
	 C_\cb^{-1}\|\chi\|^2_{L^2(P)}\leq &(b_P,\chi^2)_{L^2(P)}\leq C_\cb\|\chi\|^2_{L^2(P)},\label{b1}\\
	 C_\cb^{-1}\|\chi\|_{L^2(P)}\leq&\sum_{m=0}^1 h_P^m|b_P\chi|_{m,P}\leq C_\cb\|\chi\|_{L^2(P)}.\label{b2}
	 \end{align} The inverse inequality \eqref{b2} proves $ |b_Pv_P|_{1,P}	\leq  C_\cb h_P^{-1}\|v_P\|_{L^2(P)}.$ The first inequality in \eqref{b1} and \eqref{2.10} for $w_k=v_P\in\p_k(P)$  lead to $C_\cb^{-1}\|v_P\|^2_{L^2(P)}\leq (b_Pv_P,v_P)_{L^2(P)}=(v_h-J_1 v_h,v_P)_{L^2(P)}\leq\|v_h-J_1 v_h\|_{L^2(P)}\|v_P\|_{L^2(P)}$ with  a Cauchy-Schwarz inequality in the last step. Hence $\|v_P\|_{L^2(P)}\leq C_\cb\|v_h-J_1v_h\|_{L^2(P)}$ and a combination  with the aforementioned  estimates verifies \begin{align}|b_Pv_P|_{1,P}\leq C_\cb^2h_P^{-1}\|v_h-J_1v_h\|_{L^2(P)}\leq C_\cb^2C_{\text{PF}}|v_h-J_1v_h|_{1,P}\label{4.8}\end{align}
with Proposition~\ref{PF} in the last step. 
  A triangle inequality, and the combination of \eqref{4.5} and \eqref{4.8} result in  \[|v_h-Jv_h|_{1,\Omega}\leq(1+C_\cb^2C_{\text{PF}})(1+C_\B C_\cS(1+C_{\text{PF}}))|v_h-\pid_kv_h|_{1,\text{pw}}.\] 
	 Since $v_h-Jv_h$ vanishes on $\partial P$, the Poincar\'e-Friedrichs inequality from Proposition~\ref{PF} leads to $$\|v_h-Jv_h\|_{L^2(P)}\leq C_{\text{PF}}h_P|v_h-Jv_h|_{1,P}.$$ The two last displayed estimates conclude the proof of \ref{d} with $C_\J:=(1+C_{\text{PF}})(1+C_\cb^2C_{\text{PF}})(1+C_\B C_\cS(1+C_{\text{PF}}))$. 
\end{proof}
\begin{remark}[comparison with nonconforming VEM]
Carstensen \emph{et al}. \cite{carstensen2022nonconforming,MR4444835} discuss a companion  for the lowest-order nonconforming VEM, and the construction is done in two steps: the companion operator is first defined from the nonconforming VEM space to 
the nonconforming FEM space and then lifted to the 
the Sobolev space $V$. The extension to general degree $k$ is complex in this case and some hints for the construction are provided in \cite{carstensen2022nonconforming}. On the other hand, the construction of a companion presented in this paper for the conforming VEM  is for general degree $k$ and also is quite simple compared to the nonconforming VEM. This is useful in the context of second-order problems.
\end{remark}
\subsection{Discrete problem}
The restriction of the bilinear form  $a(\cdot,\cdot)$ to a polygonal subdomain $P\in\T_h$ is  denoted by $a^P(\cdot,\cdot)$ and its piecewise version by $a_\pw(\cdot,\cdot)$.  Define the corresponding local discrete bilinear form $a_h$ on $V_h^k\times V_h^k$ by
\begin{align}
    a_h(v_h,w_h):=a_\pw(\pid_k v_h,\pid_kw_h)+s_h((1-\pid_k) v_h,(1-\pid_k)w_h)\label{2.18}
\end{align}
with a symmetric positive definite stability term $s_h$ that satisfies 
\begin{align}
    C_\s^{-1}a_\pw(v_h,v_h)\leq s_h(v_h,v_h)\leq C_\s a_\pw(v_h,v_h)\label{stab}
\end{align}
for all $v_h\in \text{Ker}(\pid_k)$ and for a positive constant $C_\s$ that depends exclusively on $\rho$ from \ref{M2}. Let the restriction of $s_h$ to a polygonal domain be denoted by $s_h^P$ and recall the notation $N_{\text{dof}}$ from Subsection~\ref{sec:vspace} for the total number of local dofs. A standard example of $s_h^P(\cdot,\cdot)$ \cite{beirao2013basic} satisfying \eqref{stab} is
\begin{align*}
    s_h^P(v_h,w_h)=\sum_{j=1}^{N_{\text{dof}}}\text{dof}_j(v_h)\text{dof}_j(w_h)\quad\text{for all}\; v_h,w_h\in V_h^k(P).
\end{align*}
The discrete problem seeks $u_h\in V_h^k$ such that
\begin{align}
  a_h(u_h,v_h) = f_h(v_h):=f(Qv_h)\quad \text{for all}\; v_h\in V_h^k \label{eq:fweak-VE} 
\end{align}
 with $Q=\Pi_k$ when $f\in L^2(\Omega)$ for the {\it standard VEM} scheme and $Q=J$ when $f\in V^*$ for the {\it modified VEM} scheme. It is easy to observe from \eqref{2.18} that the bilinear form $a_h$ is polynomial consistent, that is, for all $v_h\in V_h^k$ and $\chi_k\in\p_k(\T_h)$,
 \begin{align}
     a_h(v_h,\chi_k) = a_\pw(v_h,\chi_k)\quad\text{and}\quad a_h(\chi_k,v_h) = a_\pw(\chi_k,v_h).\label{poly_cons}
 \end{align}
 From \eqref{2.18}-\eqref{stab}, $a_h$ is bounded and coercive (see \cite{beirao2013basic} for a proof). That is, 
\begin{align}
    a_h(v_h,w_h)&\leq (1+C_\s) |v_h|_{1,\pw}|w_h|_{1,\pw}\quad\text{for all}\;v_h,w_h\in V_h^k+\p_k(\T_h)\label{3.8.a} \\
    a_h(v_h,v_h)&\geq C_\s^{-1}|v_h|_{1,\Omega}^2\hspace{2.5cm}\text{for all}\;v_h\in V_h^k.\label{3.8}
\end{align}
This implies that the discrete solution operator $\cK_h:V^*\to V_h^k$ with $\cK_h f:=u_h$ is well-defined. 
\subsection{Best approximation}
The error analysis relies on the two choices $Q=\Pi_k$ and $Q=J$ in the discrete problem \eqref{eq:fweak-VE}. Recall that the term $(f,\Pi_kv_h)_{L^2(\Omega)}$ is well-defined for $f\in L^2(\Omega)$ and $f(Jv_h)$ for $f\in V^*$. The oscillation of $f\in L^2(\Omega)$   reads  
	\begin{align*}
	\mathrm{osc}_1(f,{\T_h}):=\Big(\sum_{P\in{\T_h}}\mathrm{osc}_1^2(f,P)\Big)^{1/2}\quad\text{and}\quad\mathrm{osc}^2_1(f,P):= \|h_P(1-\Pi_k)f\|_{L^2(P)}^2.
	\end{align*}
\begin{theorem}[error estimate] \label{thm:fwd-best}
	Let $\cK$ and $\cK_h$ be the solution operators for the continuous problem \eqref{eq:weak-fwd} and discrete problem \eqref{eq:fweak-VE}.  There exist   positive constants $C_1$  and $C_2$ that depend  exclusively  on $\rho$ from \ref{M2} such that,
	\begin{enumerate}[label=(\alph*)]
	\item\label{thm:2.3.a} for $f\in L^2(\Omega)$ and $Q=\Pi_k$,
	\begin{align*}
	C_1^{-1}|(\cK- \cK_h)f|_{1,\Omega}  \le&  \inf_{v_h \in V_h^k} |\cK f-v_h|_{1,\Omega} +  \inf_{p_k \in \p_k(\T_h)}|\cK f-p_k|_{1,\pw} + \mathrm{osc}_1(f,\T_h).
 \end{align*}
 \item for $f\in V^*$ and $Q=J$,
 \begin{align*}
 C_2^{-1}|(\cK- \cK_h)f|_{1,\Omega}\le\inf_{v_h \in V_h^k} |\cK f-v_h|_{1,\Omega} +  \inf_{p_k \in \p_k(\T_h)}|\cK f-p_k|_{1,\pw}.
 \end{align*}
	\end{enumerate}
\end{theorem}

\begin{proof}[(a) Error analysis for standard VEM ($Q=\Pi_k$)]
	The proof of \ref{thm:2.3.a} can be referred from \cite{beirao2013basic} and is presented here to maintain the continuity of reading.
	
	By the definition of operators $\cK$ and $\cK_h$, we have $(\cK- \cK_h)f = u -u_h$. Let $v_h \in V_h^k$ be any arbitrary function and $e_h := v_h-u_h$. The coercivity in \eqref{3.8} of the discrete bilinear form $a_h(\cdot, \cdot)$ on $V_h^k$ and   \eqref{eq:weak-fwd} with the test function $v=e_h\in V$ lead to
	\begin{align*}
	C_\s^{-1} |e_h|_{1,\Omega}^2 & \le a_h(v_h, e_h) - a_h(u_h, e_h) - a(u,e_h) + (f,e_h)_{L^2(\Omega)}\\
	& =a_h(v_h, e_h)- a(u,e_h) +(f,e_h)_{L^2(\Omega)}-  (f,\Pi_k e_h)_{L^2(\Omega)}
	\end{align*}
	 with  \eqref{eq:fweak-VE} in the last step. The polynomial consistency  $a_h(v_h, p_k)=a_\pw(v_h,p_k)$ in \eqref{poly_cons} for $v_h\in V_h^k$ and  $p_k \in \p_k(\T_h)$ shows
	\begin{align*}
	C_\s^{-1} | e_h |_{1,\Omega}^2 \le  (a_h(v_h-p_k, e_h)-a_\pw(u-p_k,e_h))+ (h_{\T_h}(f-\Pi_kf),h_{\T_h}^{-1}(e_h-\pid_ke_h))_{L^2(\Omega)}
	\end{align*}
	with \eqref{def:L2} in the last step. 
 For any $v\in H^1(\T_h)$, the definition of $\pid_k$ in \eqref{pid} implies that  
 \begin{align}|v-\pid_k v|_{1,\pw}\leq \inf_{\chi_k\in\p_k(\T_h)}|v-\chi_k|_{1,\pw}.\label{est:pid}
 \end{align}
 Proposition~\ref{PF} and \eqref{est:pid} for $\chi_k=0$ imply $\|h_{\T_h}^{-1}(e_h-\pid_ke_h)\|_{L^2(\Omega)}\leq C_{\text{PF}}|e_h-\pid_ke_h|_{1,\pw}\leq C_\pf|e_h|_{1,\Omega}$. This, a Cauchy-Schwarz inequality for $a_\pw$, \eqref{3.8.a} for $a_h$,  and the triangle inequality $|v_h-p_k|_{1,\pw}\leq |u-v_h|_{1,\Omega}+|u-p_k|_{1,\pw}$ provide {	
	\begin{align*}
	C_\s^{-1}| e_h |_{1,\Omega} &\le (1+C_\s)|u-v_h|_{1,\Omega}+(2+C_\s)|u-p_k|_{1,\pw}+C_\pf\|h_{\T_h}(f-\Pi_k f)\|_{L^2(\Omega)}\\&\le (2+C_\s+C_{\text{PF}}) \Big(  |u-p_k|_{1,\pw} + |u-v_h|_{1,\Omega} +   \|h_{\T_h}(f-\Pi_k f)\|_{L^2(\Omega)}\Big).
	\end{align*}}
	An infimum over  $v_h \in V_h^k$ and $p_k \in \p_k(\T_h)$ leads to
	\begin{align*}
	| e_h |_{1,\Omega} \le C_\s(2+C_\s+C_{\text{PF}}) \Big( \inf_{v_h \in V_h^k} |u-v_h|_{1,\Omega} + \inf_{p_k \in \p_k(\T_h)}|u-p_k|_{1,\pw} + \mathrm{osc}_1(f,\T_h) \Big).
	\end{align*}
	The triangle inequality $|u-u_h|_{1,\Omega}\leq |u-v_h|_{1,\Omega}+|e_h|_{1,\Omega}$ and the above estimate prove   \ref{thm:2.3.a} with  $C_1:=1+C_\s(2+C_\s+C_{\text{PF}})$.
\end{proof}
\begin{proof}[(b) Error analysis for the modified scheme ($Q=J$)]
Let $e:=u-u_h$ and $e_h:=v_h-u_h\in V_h^k$ for any $v_h\in V_h^k$. The coercivity of $a_h$ from \eqref{3.8} and the definition of $a_h$ from \eqref{2.18} show
\begin{align}
    C_\s^{-1}|e_h|_{1,\Omega}^2&\leq a_h(e_h,e_h)=a_\pw(\pid_k v_h,\pid_k e_h)+s_h((1-\pid_k)v_h,(1-\pid_k)e_h)-a_h(u_h,e_h)\nonumber\\&=a_\pw(v_h,\pid_k e_h)+s_h((1-\pid_k)v_h,(1-\pid_k)e_h)-a(u,Je_h).\label{2.20}
\end{align}
  The second equality above follows from   \eqref{pid} and the key identities $a_h(u_h,e_h) = F(Je_h) = a(u,Je_h)$ from the continuous and discrete problems \eqref{eq:weak-fwd} and \eqref{eq:fweak-VE}. The orthogonality results $a_\pw(e_h-\pid_k e_h,p_k)=0=a_\pw(e_h-Je_h,p_k)$ for any $p_k\in\p_k(\T_h)$ from \eqref{pid} and Theorem~\ref{thm:J}.c in \eqref{2.20} show

\begin{align}
   C_\s^{-1}|e_h|_{1,\Omega}^2&\leq a_\pw(v_h-u,\pid_k e_h)+s_h((1-\pid_k)v_h,(1-\pid_k)e_h)\nonumber\\&\quad+a_\pw(u-p_k,\pid_ke_h-Je_h).
\end{align}
A Cauchy-Schwarz inequality and \eqref{pid_stab}  prove that 
\begin{align}
   a_\pw(v_h-u,\pid_k e_h)\leq |u-v_h|_{1,\Omega}|e_h|_{1,\Omega}. 
\end{align} Again a Cauchy-Schwarz inequality for the inner product $s_h$, \eqref{stab}, and the continuity of $a_\pw$ lead to
\begin{align}
 s_h((1-\pid_k)v_h,(1-\pid_k)e_h)\leq C_\s|   (1-\pid_k)v_h|_{1,\pw}|(1-\pid_k)e_h|_{1,\pw}.\label{2.26}
\end{align}
 The estimate \eqref{est:pid} with $v=v_h$ and $\chi_k=\pid_ku$ for the first term, and $v=e_h$ and $\chi_k=0$ for the second term in \eqref{2.26} show
\begin{align}
   &s_h((1-\pid_k)v_h,(1-\pid_k)e_h) \leq C_\s |v_h-\pid_k u|_{1,\pw}|e_h|_{1,\Omega}\nonumber\\&\leq C_\s(|u-v_h|_{1,\Omega}+|u-\pid_ku|_{1,\pw})|e_h|_{1,\Omega}\leq C_\s(|u-v_h|_{1,\Omega}+|u-p_k|_{1,\pw})|e_h|_{1,\Omega}\label{eqn:2.24}
\end{align}
with a triangle inequality in the second step and \eqref{est:pid} for any $p_k\in\p_k(\T_h)$ in the last step. A Cauchy-Schwarz inequality and  a triangle inequality  imply
\begin{align}
    a_\pw(u-p_k,\pid_ke_h-Je_h)&\leq |u-p_k|_{1,\pw}(|e_h-\pid_ke_h|_{1,\pw}+|e_h-Je_h|_{1,\Omega})\nonumber\\&\leq (1+C_\J)|u-p_k|_{1,\pw}|e_h|_{1,\Omega}.\label{2.23}
\end{align}
The last step results from  Theorem~\ref{thm:J}.d,  and \eqref{est:pid} with $v=e_h$ and $\chi_k=0$. The combination  \eqref{2.20}-\eqref{2.23} results in {
\begin{align*}
C_s^{-1}|e_h|_{1,\Omega}&\leq (1+C_\s)|u-v_h|_{1,\Omega}+(1+C_\s+C_\J)|u-p_k|_{1,\pw}\\&\leq (1+C_\s+C_\J)(|u-v_h|_{1,\Omega}+|u-p_k|_{1,\pw}).
\end{align*}}
This and the triangle inequality $|e|_{1,\Omega}\leq|u-v_h|_{1,\Omega}+|e_h|_{1,\Omega}$ conclude the proof with $C_2:= 1+C_\s(1+C_\s+C_\J)$.
\end{proof}
\begin{remark}[$L^2$ error estimate]
    The error estimate in $L^2$ norm is based on Aubin-Nitsche duality argument. The techniques follow analogously as in the proof of Theorem~\ref{thm:err} for general second-order elliptic problems discussed in the next section and hence we omit the details here.
\end{remark}
\section{Generalization to second-order linear indefinite elliptic problems with rough source}
This section presents   the modified VEM scheme for  general second-order linear elliptic problems with rough source terms. Subsection~3.1 discusses the model problem, and the corresponding weak and discrete formulations. Subsection~3.2 proves energy and $L^2$ norm error estimates in the best approximation form.  
\subsection{Weak and discrete problem}\label{subsec:3.1}
The conforming VEM approximates the weak solution $u\in H^1_0(\Omega)$ to 
\begin{align}
    -\dv(\ba\nabla u+\bb u)+\gamma u = f\quad\text{in}\;\Omega\label{cp}
\end{align}
for $f\in V^*=H^{-1}(\Omega)$.  We assume that the coefficients $\ba_{jk}, \bb_j,$ and $\gamma$  are smooth  functions with $\ba = [\ba_{jk}]$ and $\bb=[\bb_j]$ for $j,k=1,2$ and there exists a unique solution to \eqref{cp} (see \cite{MR4444835} for more details). Define, for all $v,w\in V$, 
\begin{align*}
    B(v,w) = (\ba \nabla v,\nabla w)_{L^2(\Omega)}+  (v,\bb\cdot\nabla w)_{L^2(\Omega)}+ (\gamma v,w)_{L^2(\Omega)}.
\end{align*}
The weak formulation  seeks $u\in V$ with
\begin{align}
    B(u,v)  = f(v) \quad\text{for all}\; v\in V.\label{3.2}
\end{align}
The bilinear form $B(\cdot,\cdot)$ is continuous and satisfies an inf-sup condition \cite{braess2007finite}, that is,
\begin{align}
    B(v,w)\leq M_\cb|v|_{1,\Omega}|w|_{1,\Omega} \text{ and } 0<\beta_0:=\inf_{0\neq v\in V}\sup_{0\neq w\in V} \frac{B(v,w)}{|v|_{1,\Omega}|w|_{1,\Omega}}.\label{3.3}
\end{align}
Recall the definition of  global virtual element space  $V_h^k$ from Subsection~\ref{sec:vspace}, and define the  discrete counterpart, for all $v_h,w_h\in V_h^k$, by
\begin{align*}
    B_h(v_h,w_h) &:= (\ba \Pi_{k-1}\nabla v_h, \Pi_{k-1}\nabla w_h)_{L^2(\Omega)} +s_h((1-\pid_k)v_h,(1-\pid_k)w_h)\\&\quad+ (\Pi_k v_h,\bb\cdot\Pi_{k-1}\nabla w_h)_{L^2(\Omega)}+ (\gamma\Pi_k v_h,\Pi_k w_h)_{L^2(\Omega)}.
\end{align*}
The discrete problem seeks $u_h\in V_h^k$ such that
\begin{align}
    B_h(u_h,v_h)=  f_h(v_h)\quad\text{for all}\; v_h\in V_h^k\label{dp}
\end{align}
with $f_h(v_h) = f(Qv_h)$ for $Q=\Pi_k$ when $f\in L^2(\Omega)$ and for $Q=J$ when $f\in V^*$. There exists a unique weak solution $u\in V$ to \eqref{cp} and discrete solution $u_h\in V_h^k$ to \eqref{dp}. Refer to \cite{MR3460621} for a proof. 
\begin{remark}[comparison of bilinear forms $a_h(\cdot,\cdot)$ and $B_h(\cdot,\cdot)$]
    Note that $\Pi_{k-1}\nabla v$ and $\nabla\pid_k v$ are same for $k=1$ and $v\in H^1(\Omega)$, but not for higher values of $k\geq 2$. The choice of $\nabla\pid_k$  for $k\geq 3$ shows heavy loss of convergence rates for general second-order problems  (see \cite[Remark~4.3]{MR3460621} for more details) and hence we utilize $\Pi_{k-1}\nabla$  instead of $\nabla\pid_k$ in $B_h$.
\end{remark}
\subsection{Error estimates}
This subsection proves  the energy norm error estimate and the $L^2$ error estimate  through Aubin-Nitche duality arguments. 
\par For any $g\in L^2(\Omega)$, the Fredholm
theory  entails the existence of a unique solution to the adjoint problem that corresponds to \eqref{cp} given by \begin{align}-\dv(\ba\nabla\Phi)+\bb\cdot\nabla\Phi+\gamma\Phi=g.\label{adjoint}\end{align} That is, there exists a $\widetilde{s}>1$ such that the dual  solution $\Phi\in H^{\widetilde{s}}(\Omega)$ and satisfies the regularity estimate $\|\Phi\|_{\widetilde{s},\Omega}\leq C^*_{\text{reg}}\|g\|_{L^2(\Omega)}$.
\begin{theorem}[error estimates]\label{thm:err}
Let $u\in V\cap H^s(\Omega)$ solve \eqref{3.2} for $s\geq 1$ and $u_h\in V_h^k$ solve \eqref{dp} for the polynomial degree $k\geq \max\{1,s-1\}$. Set $\bs:= \ba \nabla u +\bb u$. For $f\in L^2(\Omega)$ and $Q=\Pi_k$, there exists a positive constant $C_3$ that depends on $\rho$ from \ref{M2} and coefficients $\ba, \bb, \gamma$ such that for sufficiently small $h$, the estimate given below holds,
\begin{align*}
    &C_3^{-1}(h^{-\min\{\widetilde{s}-1,1\}}\|u-u_h\|_{L^2(\Omega)}+|u-u_h|_{1,\Omega}) \leq |u-I_hu|_{1,\Omega}+|u-\pid_k u|_{1,\pw}\\&\hspace{6cm}+\|\bs - \Pi_{k-1}\bs\|_{L^2(\Omega)}+ \mathrm{osc}_1(\gamma u-f,\T_h).
\end{align*}
 The term $\mathrm{osc}_1(f,\T_h)$ in the above estimate vanishes for $f\in V^*$ and $Q=J$.
\end{theorem}
\begin{proof}[Proof of energy error estimate]
Since the discrete problem \eqref{dp} is well-posed, the bilinear form $B_h$ satisfies the discrete inf-sup condition for sufficiently small $h$ with a positive constant $\beta$ \cite[Lemma~5.7]{MR3245124} and implies the existence of a $v_h\in V_h^k$ for $I_hu-u_h$ with 
\begin{align}
  &\beta^{-1}|I_hu-u_h|_{1,\Omega}|v_h|_{1,\Omega} \le B_h(I_hu-u_h,v_h)= B_h(I_hu,v_h)-B(u,Jv_h) +f(Jv_h)-f_h(v_h)\nonumber\\&=\big((\ba (\Pi_{k-1}\nabla I_hu-\nabla u)+\bb (\Pi_{k}I_hu-u),\Pi_{k-1}\nabla v_h)_{L^2(\Omega)}+(\gamma (\Pi_{k}I_hu-u),\Pi_k v_h)_{L^2(\Omega)}\big)\nonumber\\&\quad+(\bs,\Pi_{k-1}\nabla v_h- \nabla J v_h)_{L^2(\Omega)}+ \big((\gamma u,\Pi_kv_h-Jv_h)_{L^2(\Omega)}+f(Jv_h)-f_h(v_h)\big)\nonumber\\&\quad+s_h((1-\pid_k)I_hu,(1-\pid_k)v_h)=:T_1+T_2+T_3 +T_4.\label{3.4}
\end{align}
The second step above follows from the continuous problem \eqref{cp} and the discrete problem \eqref{dp}, and the third step from the definitions of $B_h$ and $B$ from Subsection~\ref{subsec:3.1} and  elementary algebra.
\par For the first term $T_1$, a Cauchy-Schwarz inequality and the $L^2$ stability of $\Pi_{k-1}$ and $\Pi_k$ from \eqref{est:L2} show
\begin{align}
  T_1&\leq C_{\ba,\cb,\gamma}\big(\|\Pi_{k-1}\nabla I_hu-\nabla u\|_{L^2(\Omega)}|v_h|_{1,\Omega}+\|\Pi_{k}I_hu-u\|_{L^2(\Omega)}(|v_h|_{1,\Omega}+\|v_h\|_{L^2(\Omega)})\big)\nonumber\\& \leq C_{\ba,\bb,\gamma}(1+C_{\text{F}}) \Big(|u-I_hu|_{1,\Omega}+\|(1- \Pi_{k-1})\nabla I_hu\|_{L^2(\Omega)}\nonumber\\&\hspace{3cm}+\|u-I_hu\|_{L^2(\Omega)}+\|(1-\Pi_k)I_hu\|_{L^2(\Omega)}\Big)|v_h|_{1,\Omega}\label{3.6}
\end{align}
with triangle inequalities and the Friedrichs inequality $\|v_h\|_{L^2(\Omega)}\leq C_{\text{F}}|v_h|_{1,\Omega}$ \cite{brenner2008mathematical} in the last step. The $L^2$ orthogonalities of $\Pi_{k-1}$ resp. $\Pi_k$ from \eqref{def:L2} implies \begin{align}\| (1- \Pi_{k-1})\nabla I_hu\|_{L^2(\Omega)}&\leq \inf_{\chi_k\in\p_k(\T_h)}\|\nabla I_hu -\nabla \chi_k\|_{L^2(\Omega)}\leq |I_hu -\pid_ku|_{1,\pw},\label{est:3.8} \\ \text{resp.}\quad\|(1-\Pi_k)I_hu\|_{L^2(\Omega)}&\leq \inf_{\chi_k\in\p_k(\T_h)}\|I_hu-\chi_k\|_{L^2(\Omega)}\leq \|I_hu-\pid_ku\|_{L^2(\Omega)}.\label{3.9}\end{align} The above displayed estimates and a triangle inequality in \eqref{3.6} prove
\begin{align}
   T_1&\leq C_{\ba,\bb,\gamma}(1+C_{\text{F}})(|u-I_hu|_{1,\Omega}+|u-\pid_ku|_{1,\pw}\nonumber\\&\quad+\|u-I_hu\|_{L^2(\Omega)}+\|u-\pid_ku\|_{L^2(\Omega)})|v_h|_{1,\Omega}\nonumber\\&\leq C_{\ba,\bb,\gamma}(1+C_{\text{F}})(1+C_\pf)(|u-I_hu|_{1,\Omega}+|u-\pid_ku|_{1,\pw})|v_h|_{1,\Omega}
\end{align}
with Proposition~\ref{PF} in the last step. For the second term $T_2$, the $L^2$ orthogonality of $\Pi_{k-1}$ and Theorem~\ref{thm:J}.c lead to
\begin{align}
    T_2&=((1-\Pi_{k-1})\bs,\Pi_{k-1}\nabla v_h-\nabla Jv_h)_{L^2(\Omega)}.\label{eq:3.8}
\end{align}
  The triangle inequality $\|\Pi_{k-1}\nabla v_h-\nabla Jv_h\|_{L^2(\Omega)}\leq \|(1-\Pi_{k-1})\nabla v_h\|_{L^2(\Omega)}+|v_h-Jv_h|_{1,\Omega}$, \eqref{est:3.8} with $I_hu=v_h$ and $\chi_k=0$, and Theorem~\ref{thm:J}.d 
  {plus \eqref{est:pid}} show
\begin{align}
    T_2\leq (1+C_\J)\|(1-\Pi_{k-1})\bs\|_{L^2(\Omega)}|v_h|_{1,\Omega}.
\end{align}
   For the third term $T_3$ with $Q=\Pi_k$, the $L^2$ orthogonality \eqref{def:L2} of $\Pi_k$ and Theorem~\ref{thm:J}.c show
\begin{align}
  T_3 &= ((1-\Pi_k) (\gamma u-f),\Pi_kv_h-Jv_h)_{L^2(\Omega)}=((1-\Pi_k) (\gamma u-f),\pid_kv_h-Jv_h)_{L^2(\Omega)}\label{eqn:3.11}
\end{align}
with the last equality again from \eqref{def:L2}. Proposition~\ref{PF}  implies $\|\pid_kv_h-Jv_h\|_{L^2(\Omega)}\leq C_\pf|h_{\T_h}(\pid_kv_h-Jv_h)|_{1,\pw}$. A triangle inequality,   Theorem~\ref{thm:J}.d, and \eqref{est:pid} show
\begin{align*}|\pid_kv_h-Jv_h|_{1,\pw}&\leq |v_h-\pid_kv_h|_{1,\pw}+|v_h-Jv_h|_{1,\Omega}\\&\leq (1+C_\J) |v_h-\pid_kv_h|_{1,\pw}\leq  (1+C_\J) |v_h|_{1,\Omega}.\end{align*}
This and \eqref{eqn:3.11} result in
\begin{align}
 T_3 \leq C_\pf   (1+C_\J) \mathrm{osc}_1(\gamma u-f,\T_h)|v_h|_{1,\Omega}.
\end{align}
For $Q=J$ in $T_3$, the term $f(Jv_h)-f_h(v_h)$ in \eqref{3.4} vanishes and in that case the oscillation $\text{osc}_1(f,\T_h)$ vanishes from the above bound for $T_3$.   The term $T_4$ is bounded as in \eqref{eqn:2.24} with $v_h=I_hu$ and $e_h=v_h$. This and the combination of the aforementioned estimates for $|I_hu-u_h|_{1,\Omega}$, and a triangle inequality conclude the proof for $|u-u_h|_{1,\Omega}$ with {$C_3:= 1+\beta((1+C_\pf)(1+C_\J+C_{\ba,\bb,\gamma}(1+C_{\text{F}}))+C_\s)$}.
\end{proof}
\begin{proof}[Proof of $L^2$ error estimate]
The proof is based on  Aubin-Nitsche duality arguments. For $g=e_h:=I_hu-u_h\in L^2(\Omega)$, there exists a unique dual solution $\Phi\in V\cap H^{\widetilde{s}}(\Omega)$ to \eqref{adjoint}
with the regularity estimate $\|\Phi\|_{\widetilde{s},\Omega}\leq C^*_{\text{reg}}\|e_h\|_{L^2(\Omega)}$. Test \eqref{adjoint} with $I_hu-u_h$ and apply an integration by parts to obtain
\begin{align}
   \|e_h\|^2_{L^2(\Omega)}&= B(e_h,\Phi)\nonumber\\&=B(e_h,\Phi-I_h\Phi)+(B(e_h,I_h\Phi)-B_h(e_h,I_h\Phi))+B_h(e_h,I_h\Phi)\label{3.15}
\end{align}
with an algebraic manipulation in the last step. The boundedness of $B$ from \eqref{3.3} and Proposition~\ref{prop:est-inter-forward} show
\begin{align}
   B(e_h,\Phi-I_h\Phi)&\leq M_\cb C_\I h^{\widetilde{s}-1}|e_h|_{1,\Omega} |\Phi|_{\widetilde{s},\Omega}.
\end{align}
The definitions of $B$ and $B_h$ simplify to
\begin{align}
    &B(e_h,I_h\Phi)-B_h(e_h,I_h\Phi)= \big((\ba \nabla e_h,\nabla I_h\Phi)_{L^2(\Omega)}-(\ba \Pi_{k-1}\nabla e_h,\Pi_{k-1}\nabla I_h\Phi)_{L^2(\Omega)}\big)\nonumber\\&\quad+\big((e_h,\bb\cdot\nabla I_h\Phi)_{L^2(\Omega)}-(\Pi_ke_h,\bb\cdot\Pi_{k-1}\nabla I_h\Phi)_{L^2(\Omega)}\big)+\big((\gamma e_h,I_h\Phi)_{L^2(\Omega)}\nonumber\\&\quad-(\gamma\Pi_ke_h,\Pi_k I_h\Phi)_{L^2(\Omega)}\big)-s_h((1-\pid_k)e_h,(1-\pid_k)I_h\Phi)=:T_5+T_6+T_7+T_8.
\end{align}
For the term $T_5$,  elementary algebra and the $L^2$ orthogonality \eqref{def:L2} of  $\Pi_{k-1}$ show 
\begin{align}
  T_5&=(\ba\nabla e_h,(1-\Pi_{k-1})\nabla I_h\Phi)_{L^2(\Omega)}+\big((1-\Pi_{k-1})\nabla e_h,(\ba-\Pi_0\ba)\Pi_{k-1}\nabla I_h\Phi\big)_{L^2(\Omega)} \nonumber\\&\leq C_\ba|e_h|_{1,\Omega}(|\Phi-I_h\Phi|_{1,\Omega}+|\Phi-\pid_k\Phi|_{1,\pw})+|e_h|_{1,\Omega}\|\ba-\Pi_0\ba\|_{L^\infty(\Omega)}|I_h\Phi|_{1,\Omega}\label{3.18.1}
\end{align}
with a Cauchy-Schwarz inequality, \eqref{est:3.8} for  $u=\Phi$ followed by a triangle inequality and for $I_hu =e_h$ and $\chi_k=0$ in the last estimate. The Bramble-Hilbert lemma \cite[Theorem~4.3.8]{brenner2008mathematical} implies \begin{align}\|\ba-\Pi_0\ba\|_{L^\infty(\Omega)}\leq C_{\text{BH}}h|\ba|_{1,\infty}.\label{bh}\end{align} A triangle inequality and Proposition~\ref{prop:est-inter-forward} show $|I_h\Phi|_{1,\Omega}\leq |I_h\Phi-\Phi|_{1,\Omega}+|\Phi|_{1,\Omega}\leq (1+C_\I)|\Phi|_{1,\Omega}$. This, \eqref{bh}, and Propositions~\ref{prop:est-polyPi} and \ref{prop:est-inter-forward} in \eqref{3.18.1} prove
\begin{align}
    T_5\leq (C_\ba (C_\I+C_\apx)+C_\text{BH}(1+C_\I))h^{\min\{\widetilde{s}-1,1\}}|e_h|_{1,\Omega}|\Phi|_{\widetilde{s},\Omega}.
\end{align}
 Analogous arguments  bound $T_6$ and $T_7$ by
\begin{align}
    T_6&\leq C_\bb(C_\F(C_\I+C_\apx)+C_\pf(1+C_\I))h^{\min\{\widetilde{s}-1,1\}}|e_h|_{1,\Omega}|\Phi|_{\widetilde{s},\Omega}\\
    T_7&\leq C_\gamma C_\F(C_\I+C_\apx+C_\pf(1+C_\I)) h|e_h|_{1,\Omega}|\Phi|_{\widetilde{s},\Omega}.
\end{align}
For the stability term $T_8$, the bound \eqref{stab}, and \eqref{est:pid} with $v=e_h$ and $\chi_k=0$ for the first term and with $v=I_h\Phi$ and $\chi_k=\pid_k\Phi$ for the second term  lead to 
\begin{align}
    T_8\leq C_\s |(1-\pid_k)e_h|_{1,\pw}|(1-\pid_k)I_h\Phi|_{1,\pw}&\leq C_\s |e_h|_{1,\Omega}|I_h\Phi-\pid_k\Phi|_{1,\pw}\nonumber\\&\leq C_\s(C_\I+C_\apx)h^{\widetilde{s}-1}|e_h|_{1,\Omega}|\Phi|_{\widetilde{s},\Omega}.
\end{align}
The last step results from Propositions~\ref{prop:est-polyPi} and \ref{prop:est-inter-forward}. It remains to estimate the term $B_h(e_h,I_h\Phi)$ in \eqref{3.15}. The continuous problem \eqref{3.2} and the discrete problem \eqref{dp} imply
\begin{align}
    &B_h(e_h,I_h\Phi) = (B_h(I_hu,I_h\Phi)-B(u,JI_h\Phi))+(f(JI_h\Phi)-f_h(I_h\Phi))\nonumber\\&=(\Pi_{k-1}\nabla I_hu-\nabla u,(\ba-\Pi_0\ba)\Pi_{k-1}\nabla I_h\Phi)_{L^2(\Omega)} \nonumber\\&\quad+(\Pi_kI_hu-u,\bb\cdot\Pi_{k-1}\nabla I_h\Phi+\gamma\Pi_kI_h\Phi)_{L^2(\Omega)}+(\bs,\Pi_{k-1}\nabla I_h\Phi-\nabla JI_h\Phi)_{L^2(\Omega)}\nonumber\\&\quad+\Big((\gamma u,\Pi_kI_h\Phi-JI_h\Phi)_{L^2(\Omega)}+f(JI_h\Phi)-f_h(I_h\Phi)\Big)+s_h((1-\pid_k)I_hu,(1-\pid_k)I_h\Phi)\label{3.23}
\end{align}
with algebraic manipulations and the orthogonality $(\Pi_{k-1}\nabla I_hu-\nabla u,\chi_{k-1})_{L^2(\Omega)} = ((\Pi_{k-1}-1)\nabla I_hu,\chi_{k-1})_{L^2(\Omega)} + (\nabla (I_hu- u),\chi_{k-1})_{L^2(\Omega)}=0$ for any $\chi_{k-1}\in(\p_{k-1}(\T_h))^2$ from \eqref{def:L2} and \eqref{int:ortho} for the first term in the last step. The equality is analogous to \eqref{3.4} for $v_h=I_h\Phi$  with the only difference  the first term  is handled as in $T_1$ and leads to an additional power of $h$. Hence for the first term in \eqref{3.23}, this and \eqref{bh} prove
\begin{align}
    (\Pi_{k-1}\nabla I_hu-\nabla u,(\ba-\Pi_0\ba)\Pi_{k-1}\nabla I_h\Phi)_{L^2(\Omega)}&\nonumber\\&\hspace{-5cm}\leq 2C_\ba C_\text{BH}(1+C_\I)h(|u-I_hu|_{1,\Omega}+|u-\pid_ku|_{1,\pw})|\Phi|_{1,\Omega}. \label{3.25}
\end{align}
For the extra power of $h$ in the second term from \eqref{3.23}, a triangle inequality, and \eqref{3.9} show $$\|\Pi_kI_hu-u\|_{L^2(\Omega)}\leq 2\|u-I_hu\|_{L^2(\Omega)}+\|u-\pid_ku\|_{L^2(\Omega)}\leq 2C_\pf h(|u-I_hu|_{1,\Omega}+|u-\pid_ku|_{1,\pw})$$ with Proposition~\ref{PF} in the end. The last three terms in \eqref{3.23} are bounded analogously as  $T_2-T_4$ from \eqref{3.4} and the extra power of $h$ arises from the interpolation and polynomial projection estimates of $\Phi$. This, the combination  \eqref{3.15}-\eqref{3.25}, and the regularity estimate $|\Phi|_{\widetilde{s},\Omega}\leq C^*_{\text{reg}}\|e_h\|_{L^2(\Omega)}$ show that there exists a positive constant $C_4$ with
\begin{align}
    \|e_h\|_{L^2(\Omega)}\leq C_4h^{\min\{\widetilde{s}-1,1\}}(|e_h|_{1,\Omega}+|u-I_hu|_{1,\Omega}+|u-\pid_ku|_{1,\pw}+\text{osc}_1(\gamma u -f,\T_h)).\label{3.26}
\end{align}
Proposition~\ref{PF} implies that $\|u-I_hu\|_{L^2(\Omega)}\leq C_\pf h |u-I_hu|_{1,\Omega}$. This, the bound for $|e_h|_{1,\Omega}$ from the proof of energy error estimate in \eqref{3.26}, and the triangle inequality 
$\|u-u_h\|_{L^2(\Omega)}\leq \|u-I_hu\|_{L^2(\Omega)}+\|I_hu-u_h\|_{L^2(\Omega)}$
conclude the proof of the $L^2$ error estimate {with a re-labelled constant $C_3$}.
\end{proof}


	\section{Inverse problem}
Given a measurement $m$, this section deals with the inverse problem to reconstruct the source field of the Poisson problem. 	
 
 The problem is defined as follows. For given measurements of $u$  on {\it finite locations} in the polygonal domain $\Omega \subset {\mathbb R}^2$ with boundary $\partial \Omega$, determine $f$ such that 
 \begin{align} \label{inv}
 - \Delta u = f \text{ in } \Omega.
 \end{align}
  Since \eqref{inv} is ill-posed \cite{MR3245124}, Subsection~4.1 introduces a regularized problem following Tikhonov regularization technique. This is followed by a set of assumptions that are crucial in the rest of the paper for obtaining error estimates. Subsection~4.2 analyses the given measurement data and its discrete VEM approximation through two computable operators $\Pi_k$ and $J$. The last subsection presents discrete VEM spaces for the inverse problem, the discrete inverse problem, and the corresponding  error estimates.
	\subsection{Regularized problem and assumptions}
	Given a source $f$, the Poisson forward 
  problem in \eqref{eq:weak-fwd}  seeks the density field $u$. The inverse problem \eqref{inv} approximates  the source $f\in F:=H^p(\Omega)$ for $p\geq 0$ from the finite measurement data of $u$  and we assume that  this data is given in terms of the linear functionals. For example, the measurement functionals could be average of $u$ in different pockets of the domain denoted as $ \omega_1, \omega_2,$ and $ \omega$ in Figure~\ref{fig:m-domains}. 
  \begin{figure}[H]
	\begin{subfigure}{.5\textwidth}
		\centering
		\includegraphics[width=0.9\linewidth]{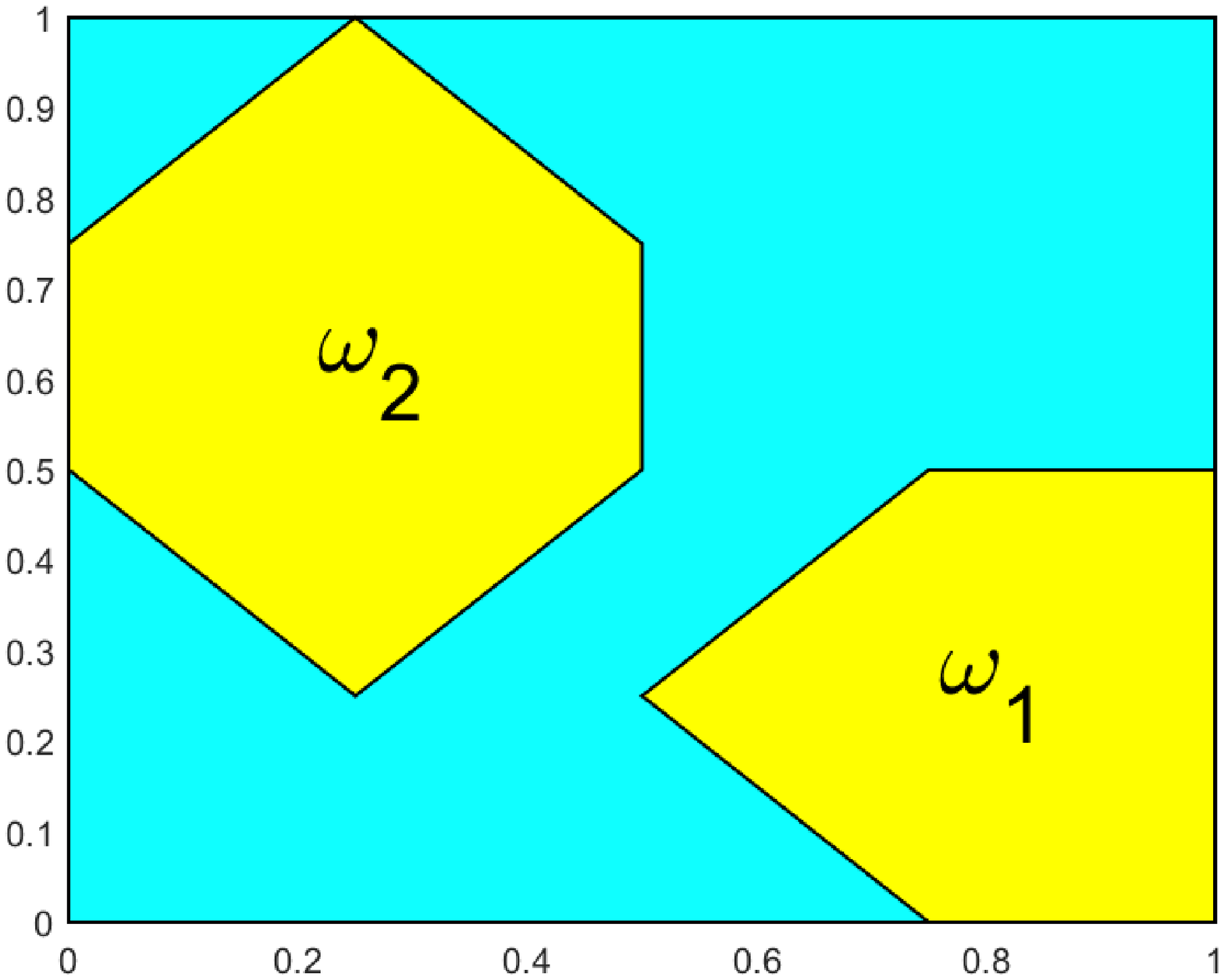}
	\end{subfigure}
\begin{subfigure}{.5\textwidth}
	\centering
	\includegraphics[width=0.9\linewidth]{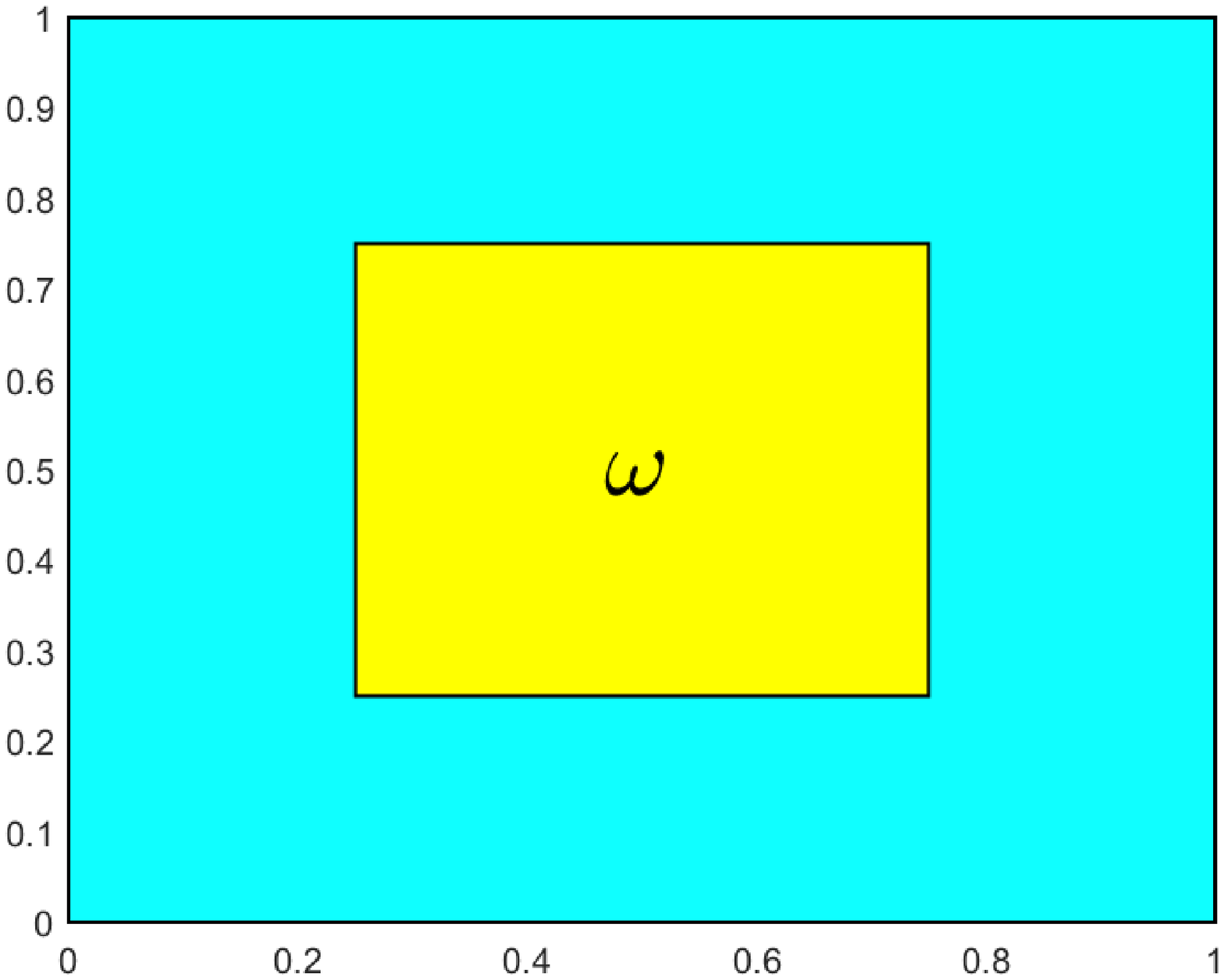}
\end{subfigure}
	\caption{Examples of measurement domains.}
	\label{fig:m-domains}
\end{figure}
  Given the measurement functionals $\mathbb{h}_1,\dots,\mathbb{h}_N\in V^*$, define $H:V^*\to\mathbb{R}^N$ by \[Hv:=(\mathbb{h}_i(v))_{i=1,\dots,N}\quad\text{for}\;v\in V.\]
	Assume that $\mathbb{h}_i$'s are uniformly bounded and this implies that $H$ is a bounded operator, that is, there exists a positive constant $C_\h$ (independent of $v$) such that \begin{align}\|Hv\|\leq C_\h\|v\|_{1,\Omega}\quad\text{ for} \; v\in V.\label{bd_meas}\end{align}
	\textbf{Auxiliary problems}. Given  $\mathbb{h}_1,\dots,\mathbb{h}_N\in V^*$, the auxiliary problems seek the measurement functions $\xi_1,\dots,$$\xi_N\in V$ such that
	\begin{align}
	    a(\xi_i,v) = \mathbb{h}_i(v)\quad\text{for all}\;v\in V\;\text{and}\;i=1,\dots,N.\label{ap}
	\end{align}
 The ill-posedness of  \eqref{inv} is overcome by a Tikhonov regularization   \cite{MR3245124}. In this paper, we define the regularized problem as 
\begin{align} \label{eq:regularized}
f_r= \arg \min_{f \in F} \left\{ \| m-H \cK f\|^2 + \alpha \, \cB(f,f) \right\}
\end{align}
 with a regularization parameter $\alpha>0$ and a  bilinear form $\cB(\cdot, \cdot): F\times F \rightarrow {\mathbb R}$ satisfying \ref{A0} below.
 \begin{enumerate}[label={(\bfseries A0)}]
	\item\label{A0} \textbf{Continuity and coercivity}.
 There exist  positive constants $C^*$ and $C_*$ that depend exclusively on $\rho$ from \ref{M2} such that
\begin{align}
    \cB(f,g)\leq C^*\|f\|_{p,\Omega}\|g\|_{p,\Omega}\;\text{and}\; \cB(f,f)\geq C_*\|f\|_{p,\Omega}^2\quad\text{for all}\;f,g\in F.\label{b_bounds}
\end{align}
\end{enumerate}
 
 \medskip
 \noindent\textbf{Regularized problem}. Denote the actual source field by $f_{\text{true}}\in F$. 
 Given the measurement $m:=HKf_{\text{true}}$ of $u$, we  reconstruct the source field of the Poisson problem in a closed subspace of $F$.   For a regularization parameter $\alpha$, the regularized problem seeks $f_r\in F$ such that
	\begin{align}
	    \A_r(f_r,g) =\cL_r(g) \qquad\text{for all}\;g\in F\label{eq:cts_IP}
	\end{align}
	with $\A_r(f,g) := (HKf)^T(HKg)+\alpha \cB(f,g)$, and the continuous right-hand side $\cL_r(g):=m^T(HKg)$ for  $f,g\in F$.
  The bilinear form $\A_r(\cdot,\cdot)$ is continuous and coercive on $F\times F$ (see \cite[Lemma~3.1]{MR3245124} for a proof), and consequently Lax-Milgram lemma implies the existence of a unique solution $f_r \in F$ to \eqref{eq:cts_IP}.
  \par The rest of the paper assumes the regularity results \ref{A1}-\ref{A3} stated below.
	\begin{enumerate}[label={(\bfseries A\arabic*)}]
	\item\label{A1} \textbf{Regularity of the forward solution}. For a given $p\geq 0$, there exists an $s\geq 1$, such that for every $f\in H^p(\Omega)$, the solution $u:=Kf$ of the problem \eqref{eq:weak-fwd} belongs to $H^s(\Omega)$ with a positive constant $C_{\mathrm{reg1}}$ (independent of $f$) and 
	\begin{align}
	    \|u\|_{s,\Omega}\leq C_{\mathrm{reg1}}\|f\|_{p,\Omega}.
	\end{align}
	\item\label{A2} \textbf{Regularity of measurement functions}. Given  the measurement functionals $h_i$, the functions $\xi_i$ belong to $H^r(\Omega)$ for some $r\geq 1$ and there exists a positive constant $C_{\mathrm{reg2}}$ with 
	\begin{align}
	    \|\xi_i\|_{r,\Omega}\leq C_{\mathrm{reg2}} \qquad \text{for}\;i=1,\dots,N. 
	\end{align}
	\item \label{A3} \textbf{Regularity of the reconstructed source}. There exist a $q\geq p$  and  a positive constant $C_{\mathrm{reg3}}$ such that the regularized solution $f_r$ belongs to $H^{q}(\Omega)$ with 
	\begin{align}
	    \|f_r\|_{q,\Omega}\leq C_{\mathrm{reg3}}\|m\|.
	\end{align}
	
	\end{enumerate}
 The notation in the regularity assumptions \ref{A1}-\ref{A3} follows from 
 \cite{MR3245124}. 

\subsection{Discrete measurement functionals}	
For any $f\in F$, recall the inner product $a(\bullet,\bullet)=(\nabla\bullet,\nabla\bullet)_{L^2(\Omega)}$ on $V\times V$ and the solution operator $\cK$ with $a(\cK f,\bullet)=f(\bullet)$ on $V$ from \eqref{eq:weak-fwd}. Also recall the discrete bilinear form $a_h(\bullet,\bullet)=a_\pw(\pid_k\bullet,\pid_k\bullet)+s_h((1-\pid_k)\bullet,(1-\pid_k)\bullet)$ on $V_h^k\times V_h^k$ from \eqref{eq:fweak-VE}  and the discrete solution operator $\cK_h$ with $a_h(\cK_h f,\bullet)=f_h(\bullet)=f(Q\bullet)$ on $V_h^k$.

Define the discrete counterpart $H_h:V_h^k\to\mathbb{R}^N$ of $H$, for $v_h\in V_h^k$, by 
\begin{align}
    H_hv_h:=HQv_h:=(\mathbb{h}_i(Qv_h))_{i=1,\dots,N} \label{dm}
\end{align}
with the two choices  $Q=\Pi_k$ when $\mathbb{h}_i\in L^2(\Omega)$ and $Q=J$ when $\mathbb{h}_i\in V^*$ for $i=1,\dots,N$. 
This leads to a discrete approximation $m_h:=H_h\cK_hf$ of a measurement $m= H\cK f$ for any $f\in F$.

\begin{theorem}[measurement approximation]\label{thm:meas}
Let $u:=\cK f$ and $u_h:=\cK_hf$. The two choices of $Q$ in the discrete problem $a_h(u_h,v_h)=f(Qv_h)$ from \eqref{eq:fweak-VE} and in the discrete measurements functionals $HQv_h$ from \eqref{dm} show
\begin{enumerate}[label=$(\alph*)$]
\item \label{thm:4.1.a}for $f,\mathbb{h}_i\in V^*$ and $Q=J$
\begin{align*}
    C_5^{-1}\|m-m_h\|&\leq (|u-u_h|_{1,\Omega}+|u-\pid_ku|_{1,\mathrm{pw}})\nonumber\\&\quad\times\sum_{i=1}^{N}\inf_{v_h\in V_h^k}(|\xi_i-v_h|_{1,\Omega}+|v_h-\pid_kv_h|_{1,\mathrm{pw}}),
\end{align*}
\item\label{thm:4.1.b} for $f,\mathbb{h}_i\in L^2(\Omega)$ and $Q=\Pi_k$  
\begin{align*}
    C_6^{-1}\|m-m_h\|&\leq (|u-u_h|_{1,\Omega}+|u-\pid_ku|_{1,\mathrm{pw}}+\mathrm{osc}_1(f,\T_h))\nonumber\\&\quad\times\sum_{i=1}^{N}\Big(\inf_{v_h\in V_h^k}(|\xi_i-v_h|_{1,\Omega}+|v_h-\pid_kv_h|_{1,\mathrm{pw}})+\mathrm{osc}_1(\mathbb{h}_i,\T_h)\Big)
\end{align*}
\end{enumerate}
with positive constants $C_5$ and $C_6$ that depend exclusively on $\rho$ from \ref{M2}. For $Q=J$ in \eqref{eq:fweak-VE} resp. \eqref{dm}  and $Q=\Pi_k$ in \eqref{dm} resp. \eqref{eq:fweak-VE}, the oscillation term $\mathrm{osc}_1(f,\T_h)$ resp. $\mathrm{osc}_1(\mathbb{h}_i,\T_h)$ vanishes. 
\end{theorem}
\begin{proof}[Proof of \ref{thm:4.1.a} ]
  Note that the bilinear form $a(\cdot,\cdot)$ on $V\times V$ is symmetric. This and  the auxiliary problem \eqref{ap} for $i=1,\dots,N$ with the test function $v=u-Ju_h\in V$ imply
\begin{align}
    \mathbb{h}_i(u-Ju_h)=a(u-Ju_h,\xi_i) &= a(u-Ju_h,\xi_i-Jv_h)-(a(Ju_h,Jv_h)-a_h(u_h,v_h))\nonumber\\&\quad-(f_h(v_h)-f(Jv_h))\label{4.10}
\end{align}
with  \eqref{eq:weak-fwd} and \eqref{eq:fweak-VE} in the last equality. 
A triangle inequality, Theorem~\ref{thm:J}.d, and \eqref{est:pid} with $v=u_h$ and $\chi_k=\pid_ku$  show
 \begin{align}
 &|u-Ju_h|_{1,\Omega}\leq |u-u_h|_{1,\Omega}+C_\J |u_h-\pid_k u_h|_{1,\pw}\nonumber\\& \leq |u-u_h|_{1,\Omega}+C_\J |u_h-\pid_ku|_{1,\pw} \leq (1+ C_\J) (|u-u_h|_{1,\Omega}+ |u-\pid_ku|_{1,\pw}) 
 \label{4.13}
 \end{align}
 with a triangle inequality applied again in the last step.
 Similarly, a triangle inequality plus Theorem~\ref{thm:J}.d yield
\begin{align}
|\xi_i-Jv_h|_{1,\Omega} &\le |\xi_i-v_h|_{1,\Omega}+C_\J|v_h-\pid_kv_h|_{1,\pw} \nonumber \\ & \le (1+C_\J)(|\xi_i-v_h|_{1,\Omega}+|v_h-\pid_kv_h|_{1,\pw} ).
\end{align}
For the first term in \eqref{4.10}, a Cauchy-Schwarz inequality and the last two displayed  estimates show 
 \begin{align}
   a(u-Ju_h,\xi_i-Jv_h) &\leq (1+C_\J)^2(|u-u_h|_{1,\Omega}+|u-\pid_ku|_{1,\text{pw}})
   \nonumber\\&\qquad\times(|\xi_i-v_h|_{1,\Omega}+|v_h-\pid_kv_h|_{1,\pw}).\label{4.11}
\end{align}
  For the second term in \eqref{4.10}, the polynomial-consistency from \eqref{poly_cons} for $v_h\in V_h^k$ and $p_k\in\p_k(\T_h)$, $a_\pw(\pid_k u,v_h)=a_\pw(\pid_k u,Jv_h)$ and $a_\pw(u_h-\pid_hu,p_k) = a_\pw(Ju_h-\pid_ku,p_k)$ in the first and  second steps below from Theorem~\ref{thm:J}.c lead   to
\begin{align*}
 a(Ju_h,Jv_h)- a_h(u_h,v_h)&=a_{\text{pw}}(Ju_h-\pid_ku,Jv_h)+a_h(\pid_ku-u_h,v_h)\\ &=  a_{\text{pw}}(Ju_h-\pid_ku,Jv_h-p_k)+ a_h(\pid_ku-u_h,v_h-p_k).
\end{align*}
A Cauchy-Schwarz inequality for $a_\pw$ and the continuity of $a_h$ from \eqref{3.8.a}   imply
$$a(Ju_h,Jv_h)-a_h(u_h,v_h)\leq |Ju_h-\pid_ku|_{1,\pw}|Jv_h-p_k|_{1,\pw}+(1+C_\s)|\pid_ku-u_h|_{1,\pw}|v_h-p_k|_{1,\pw}.$$
Triangle inequalities and \eqref{4.13} show
\begin{align}
  &a(Ju_h,Jv_h)-a_h(u_h,v_h)\leq (3+C_\s+C_\J)(|u-u_h|_{1,\Omega}+|u-\pid_ku|_{1,\pw})\nonumber\\&\hspace{5cm}\times(|v_h-p_k|_{1,\pw}+|v_h-Jv_h|_{1,\Omega})\nonumber\\&\leq (3+C_\s+C_\J)(1+C_\J)(|u-u_h|_{1,\Omega}+|u-\pid_ku|_{1,\pw})|v_h-\pid_kv_h|_{1,\pw}.\label{4.15}
\end{align}
The last step above results from the choice $p_k=\pid_kv_h$ and from Theorem~\ref{thm:J}.d. The last term $f(Jv_h)-f_h(v_h)$ in \eqref{4.10} vanishes for the choice $Q=J$. The definitions of $m$ and $m_h$ for $Q=J$ imply $m-m_h=H\cK f-HJ\cK_hf=((\mathbb{h}_i,u-Ju_h))_{i=1,\dots,N}$. Hence the combination of \eqref{4.11} and \eqref{4.15} in \eqref{4.10} conclude the proof of \ref{thm:4.1.a} with $C_5 := (1+C_\J)^2+(3+C_\s+C_\J)(1+C_\J)$.
\end{proof}
\begin{proof}[Proof of \ref{thm:4.1.b}]
The definitions of $m=H\cK f= Hu$ and $m_h=H_h\cK_hf=HQu_h$ for $Q=\Pi_k$ lead to
\begin{align}
m-m_h = (\mathbb{h}_i(u-\Pi_ku_h))_{i=1,\dots,N} = (\mathbb{h}_i(u-Ju_h)+\mathbb{h}_i(Ju_h-\Pi_ku_h))_{i=1,\dots,N}.\label{4.9}
\end{align} 
For the  first term in the last displayed identity, we utilize  \eqref{4.10} and estimate the first two terms   as  in the proof of \ref{thm:4.1.a}. It remains to bound the last term in \eqref{4.10}  for the choice  $Q=\Pi_k$. The $L^2$ orthogonality of $\Pi_k$ and of $J$ from Theorem~\ref{thm:J}.b show $(J-\Pi_k)u_h\perp \p_k(\T_h)$ in $L^2(\Omega)$. This and again the $L^2$ orthogonality of $\Pi_k$  imply
\begin{align}
    f(Jv_h)-f_h(v_h) =  (f,(J-\Pi_k)v_h)_{L^2(\Omega)} =(h_{\T_h}(f-\Pi_kf),h_{\T_h}^{-1}(J-\pid_k)v_h))_{L^2(\Omega)}.\label{4.16}
\end{align}
The Poincar\'e-Friedrichs inequality  $\|h_{\T_h}^{-1}(J-\pid_k)v_h\|_{L^2(\Omega)}\leq C_{\text{PF}}|(J-\pid_k)v_h|_{1,\text{pw}}$ from Proposition~\ref{PF}, a triangle inequality, and Theorem~\ref{thm:J}.d show
\begin{align}
 f(Jv_h)-f_h(v_h)\leq  C_\pf (1+C_\J)\mathrm{osc}_1(f,\T_h)|v_h-\pid_kv_h|_{1,\pw}.\label{4.17}
\end{align}
For the second term in \eqref{4.9}, analogous arguments in \eqref{4.16}-\eqref{4.17} lead to
\begin{align}
    \mathbb{h}_i(Ju_h-\Pi_ku_h) &\leq C_\pf(1+C_\J)\mathrm{osc}_1(\mathbb{h}_i,\T_h)|u_h-\pid_ku_h|_{1,\pw}\nonumber\\&\leq C_{\text{PF}}(1+C_\J)\mathrm{osc}_1(\mathbb{h}_i,\T_h)(|u-u_h|_{1,\Omega}+|u-\pid_ku|_{1,\text{pw}})\label{4.14}
\end{align}
 with \eqref{est:pid} and a triangle inequality in the last step. The substitution  of  \eqref{4.11}-\eqref{4.15} and \eqref{4.17} in \eqref{4.10} for the first term, and  \eqref{4.14} for the second term in \eqref{4.9}  conclude the proof of \ref{thm:4.1.b} with $C_6:=C_5+2C_\pf(1+C_\J)$.
\end{proof}
\begin{theorem}[convergence rates]
Let $\cK$ and $\cK_h$ be the solution operators   for the continuous problem \eqref{eq:weak-fwd} and discrete problem \eqref{eq:fweak-VE}, respectively.   Let $u\in H^s(\Omega)$ for $s\geq 1$, $f\in H^p(\Omega)\cap H^{s-2}(\Omega)$ for $p\geq 0$,   $k\geq s-1$ for the polynomial degree $k$ of the virtual element space $V_h^k$, and $\xi_i\in H^r(\Omega)$ for $r\geq 1$ and for the solutions $\xi_i$ of auxiliary problems \eqref{ap}.   Then for the choices $Q=\Pi_k$ and $J$, under the assumptions \ref{A1}-\ref{A3} the estimates below hold.
\begin{align}
    |u-u_h|_{1,\Omega}:=|\cK f-\cK_h f|_{1,\Omega}\leq C_7 h^{s-1}\|f\|_{p,\Omega}.\label{3.18}
\end{align}
In addition, if we  assume that $\mathbb{h}_i\in V^*\cap H^{r-2}(\Omega)$ and $k\geq \max\{s,r\}-1$, then 
\begin{align}
    \|m-m_h\|:=\|H\cK f- H_h\cK_hf\| \leq C_8 h^{s+r-2}\|f\|_{p,\Omega}.\label{3.19}
\end{align}
\end{theorem}
\begin{proof}[Proof of \eqref{3.18}]
 For $Q=\Pi_k$, choose $v_h=I_h\cK f$ and $p_k=\pid_k\cK f$ in Theorem~\ref{thm:fwd-best}.a. Proposition~\ref{prop:est-polyPi}-\ref{prop:est-inter-forward}   lead  to
\begin{align*}
C_1^{-1}|\cK f- \cK_h f|_{1,\Omega}&\leq  |\cK f- I_h \cK f|_{1,\Omega} + |\cK f- \pid_k \cK f|_{1,\pw}+ h\|f-\Pi_kf\|_{L^2(\Omega)}\\&\leq (C_\I+C_\apx)h^{s-1}(|u|_{s,\Omega}+|f|_{s-2,\Omega}) \leq C_{\text{reg1}}(C_\I+C_\apx)h^{s-1}\|f\|_{p,\Omega}
\end{align*}
with $p\geq s-2$ and \ref{A1} in the last  step. \par For $Q=J$, the estimate in Theorem~\ref{thm:fwd-best}.b is same as  the first two terms in Theorem~\ref{thm:fwd-best}.a, and hence bounded by the previous displayed estimate.  This concludes the proof of \eqref{3.18} with $C_7:=C_1C_{\text{reg1}}(C_\I+C_\apx)$.
\end{proof}
\begin{proof}[Proof of \eqref{3.19}]
For $i=1,\dots,N$, the choice  $v_h= I_h\xi_i\in V_h^k$ in Theorem~\ref{thm:meas}.a and Proposition~\ref{prop:est-inter-forward} imply   \begin{align}|\xi_i-I_h\xi_i|_{1,\Omega}\leq C_\I h^{r-1}|\xi_i|_{r,\Omega}.\label{4.22}\end{align} This, \eqref{3.18}, and \ref{A2}  conclude the proof for $Q=J$ with $C_8:=C_6C_7C_\I C_{\text{reg2}} N$.
\par For $Q=\Pi_k$ in Theorem~\ref{thm:meas}.b, a consequence \eqref{est:pid} of the definition of $\pid_k$ and a triangle inequality show \begin{align}|I_h\xi_i - \pid_kI_h\xi_i|_{1,\pw}\leq |I_h\xi_i-\pid_k \xi_i|_{1,\pw}&\leq|\xi_i-I_h\xi_i|_{1,\Omega}+|\xi_i-\pid_k\xi_i|_{1,\pw}\nonumber\\&\leq (C_\I+C_\apx) h^{r-1}|\xi_i|_{r,\Omega}\label{4.23}\end{align}  with Proposition~\ref{prop:est-polyPi}-\ref{prop:est-inter-forward}  in the end. For the oscillation of $\mathbb{h}_i$, Proposition~\ref{prop:est-polyPi} implies 
\begin{align}
\text{osc}_1(\mathbb{h}_i,\T_h)\leq  h\|\mathbb{h}_i-\Pi_k\mathbb{h}_i\|_{L^2(\Omega)}\leq C_\apx h^{r-1}|\mathbb{h}_i|_{r-2,\Omega}.\label{4.24}
\end{align}
 The estimate \eqref{3.18} and  the combination \eqref{4.22}-\eqref{4.24} in Theorem~\ref{thm:meas}.b  provide
\begin{align*}
   C_6^{-1} \|m-m_h\|&\leq {2C_7(C_\I+C_\apx)}h^{s+r-2}\|f\|_{p,\Omega}\sum_{i=1}^{N}(|\xi_i|_{r,\Omega}+|\mathbb{h}_i|_{r-2,\Omega})\\&\leq 2C_7(C_\I+C_\apx)( C_{\text{reg2}}+C_\h)Nh^{s+r-2} \|f\|_{p,\Omega}
\end{align*}
with \ref{A2} and \eqref{bd_meas} in the last step. This  concludes the proof of \eqref{3.19} with a re-labelled constant $C_8:=2C_6C_7(C_\I+C_\apx)( C_{\text{reg2}}+C_\h)N$.
\end{proof}
\subsection{Virtual element method for the inverse problem}
Recall $h_P$ denote the diameter of a polygonal domain $P$. Let $\T_\tau$ be an admissible polygonal decomposition satisfying \ref{M1}-\ref{M2} for the discretisation parameter $\tau:=\max_{P\in\T_\tau}h_P$  in the inverse problem.   
Given $ F=H^p(\Omega)$ for $p\geq 1$, construct the discrete space $F_{\
\tau,p}^{\ell}:=\{g_\tau\in F :g_\tau|_P\in F_{\tau,p}^{\ell}(P)\quad\text{for all}\;P\in\T_\tau\}\subset F$ on polygonal meshes $\T_\tau$ with the local conforming virtual element space $F_{\tau,p}^{\ell}(P)$ of order $p$ and of degree $\ell$. Let $\Pi_\ell^*:F_{
\tau,p}^{\ell}(P)\to \p_\ell(P)$ be  bounded in $H^p$ seminorm and  computable projection operator for all $P\in\T_\tau$. The definition of the space $F_{
\tau,p}^{\ell}(P)$ for $p\geq 1$ and $\ell\geq 2p-1$ from \cite{MR4070337} reads
\begin{align}F_{\tau,p}^{\ell}(P):= \begin{rcases}\begin{dcases}&g_\tau\in H^p(P) : \Delta^p g_\tau\in\p_{\ell}(P),\quad \forall E\in\e(P)\quad g_\tau|_E\in\p_\ell(E)\\&\quad\text{and}\quad\gamma_j(g_\tau)|_E\in\p_{\ell-j}(E)\quad\text{for}\;j=1,\dots,p-1,\\&\quad (g_\tau-\Pi_\ell^*g_\tau,\chi_\ell)_{L^2(P)}=0\quad\text{for all}\;\chi_\ell\in\mathcal{M}_{\ell}(P)\setminus\mathcal{M}_{\ell-p-1}(P)\end{dcases}\end{rcases}\label{ds_inverse}\end{align}
for the $j^{\text{th}}$ trace $\gamma_j$ on the boundary $\partial P$ of the polygonal subdomain $P\in\T_\tau$. For $p=0$, one can simply choose the discrete space as piecewise polynomials, that is, \[F_{\tau,0}^\ell:=\{g_\tau\in L^2(\Omega):g_\tau|_P\in\p_\ell(P)\quad\text{for all}\; P\in\T_\tau\}.\]
The functions $g_\tau$ in \eqref{ds_inverse} can be characterized through following degrees of freedom:
\begin{itemize}
    \item $\tau_z^{|a|}D^{a}g_\tau(z)$ for $|a|\leq p-1$ and for any vertex $z\in\mathcal{V}(P)$ with associated characteristic length $\tau_z^{|a|}$,
    \item $\dashint_Eg_\tau p_{\ell-2p}\,ds$ for  $p_\ell\in\mathcal{M}_{\ell-2p}(E)$ and for any edge $E\in\e(P)$,
    \item $\tau_E^{-1+j}\int_E(g_\tau)_\textbf{n}p_{\ell-2p}\,ds$ for $p_{\ell-2p}\in\mathcal{M}_{\ell-2p-j}(E)$, $j=1,\dots,p-1$ and  for any edge $E\in\e(P)$,
    \item $\dashint_Pg_\tau p_{\ell-2p}\,dx$ for $p_{\ell-2p}\in\mathcal{M}_{\ell-2p}(P)$.
\end{itemize}
\begin{remark}[comparison of virtual elements spaces in forward and inverse problems]
  Note that the virtual element space $V_h^k$ is a subset of  $V=H^1_0(\Omega)$ in the forward problem, whereas the discrete space $F_{\tau,p}^\ell\subset H^p(\Omega)$ changes with the order $p$ in the inverse problem. For $p=1$ and $h=\tau$, we can choose $\Pi_\ell^*=\Pi_\ell^\nabla$ and then the definition \eqref{ds_inverse} coincides with the local enhanced virtual element space \eqref{local_sp} in the forward problem. For $p=2$, \eqref{ds_inverse} coincides with the conforming virtual element space for the biharmonic problem  \cite{MR3002804}.
\end{remark}
Let $\cB_\pw$ denote the piecewise version of $\cB$. We make an additional assumption \ref{A4}
 on the discrete bilinear form $\cB_\tau$. 
\begin{enumerate}[label={(\bfseries A4)}]
\item\label{A4} The  discrete bilinear form $\cB_\tau:F_{\tau,p}^{\ell}\times F_{\tau,p}^{\ell}\to \mathbb{R}$  satisfies   the  two  properties  below:
\begin{itemize}
\item Polynomial consistency: 
\begin{align}\cB_\tau(g_\tau,\chi_\ell) = \cB_\pw(g_\tau,\chi_\ell)\quad\text{ for all}\; g_\tau\in F_{\tau,p}^{\ell}\quad\text{and}\quad \chi_\ell\in\p_\ell(\T_\tau).\label{b:cons}\end{align}
\item  Stability   with respect the norm $ \cB(\cdot,\cdot)^{1/2}$ on $F_{\tau,p}^{\ell}$ : There exists a positive constant $C_\stab$ (depending exclusively on $\rho$ from \ref{M2}) with
\begin{align}C_\stab^{-1}\cB(g_\tau,g_\tau)\le \cB_\tau(g_\tau,g_\tau)\le C_\stab \cB(g_\tau,g_\tau)\quad\text{ for all}\; g_\tau\in F_{\tau,p}^{\ell}.\label{b:stab}\end{align}
\end{itemize}
\end{enumerate}
\textbf{Discrete inverse problem}. The discrete version of \eqref{eq:cts_IP} seeks $f_h^\tau\in F_{\tau,p}^{\ell}$ such that
\begin{align} \label{eq:weak-inv-h}
\A^\tau_h(f_h^\tau,g_\tau)=\cL^\tau_h(g_\tau) \qquad \text{for all}\; g_\tau \in F_{\tau,p}^\ell
\end{align}
with $\A_h^\tau(f,g) := (H_hK_hf)^T(H_hK_hg)+\alpha \cB_\tau(f,g)$ and $\cL^\tau_h(g):=m^T(H_hK_hg)$ for  $f,g\in F_{\tau,p}^\ell$. 
\begin{theorem}[well-posedness of discrete inverse problem]\label{wp:dis}
 For all $f_\tau,g_\tau\in F_{\tau,p}^\ell$,  there exist positive constants $C_9(\alpha)$ and $C_{10}(\alpha)$ such that
\begin{align*}
 \A_h^\tau(f_\tau,g_\tau) \leq C_9(\alpha)\|f_\tau\|_{p,\Omega}\|g_\tau\|_{p,\Omega} \quad\text{and}\quad\A_h^\tau(g_\tau,g_\tau) \geq C_{10}(\alpha)\|g_\tau\|_{p,\Omega}^2.
\end{align*}
Moreover, there exists a unique discrete solution $f_h^\tau\in F_{\tau,p}^\ell$ to \eqref{eq:weak-inv-h}.
\end{theorem}
\begin{proof}
 For $f_\tau,g_\tau\in F_{\tau,p}^{\ell}$, the definition of $H_h$  and the continuity of $\cB$ from \eqref{b_bounds} show 
\begin{align}
  \A_h^\tau(f_\tau,g_\tau) &\leq \|HQ\cK_hf_\tau\|  \|HQK_hg_\tau\| + \alpha C^*\|f_\tau\|_{p,\Omega}\|g_\tau\|_{p,\Omega}.\label{3.10}
\end{align}
A triangle inequality, Proposition~\ref{prop:est-polyPi} for $Q=\Pi_k$, Theorem~\ref{thm:J}.d, {and the inequality \eqref{est:pid} (with $\chi_k=0$)} for $Q=J$ imply $\|Q\cK_hg_\tau\|_{1,\Omega}\leq \|(Q-1)\cK_hg_\tau\|_{1,\Omega}+\|\cK_hg_\tau\|_{1,\Omega}\leq (2+C_\apx+C_\J)\|\cK_hg_\tau\|_{1,\Omega}$. This, the boundedness of $H$ from \eqref{bd_meas}, and a triangle inequality     result in \begin{align*}\|HQ\cK_hg_\tau\|&\leq C_\h(2+C_\apx+C_\J)(\|\cK_hg_\tau-\cK g_\tau\|_{1,\Omega}+\|\cK g_\tau\|_{1,\Omega})\nonumber\\&\leq C_\h(2+C_\apx+C_\J)(C_7+C_{\text{reg1}})\|g_\tau\|_{p,\Omega}\end{align*}
with  \eqref{3.18} and \ref{A1} in the last step.  Then $C_{11}:= C_\h(C_7+C_{\text{reg1}})(2+C_\apx+C_\J)$ implies that 
\begin{align}
\|H_h\cK_hg_\tau\|\leq C_{11}\|g_\tau\|_{p,\Omega}\quad\text{for any}\;g_\tau\in F_{\tau,p}^{\ell}.\label{3.11}
\end{align}
Hence the   estimates \eqref{3.10}-\eqref{3.11} prove that the bilinear form $\A_h^\tau$ is bounded with $C_9(\alpha):=C_{11}^2+\alpha C^*$.
\par For $g_\tau\in F_{\tau,p}^\ell$, the stability of $\cB_\tau$ in \eqref{b:stab} and $\|H_hK_hg_\tau\|^2\geq 0$  lead  to 
\begin{align}
   \A_h^\tau(g_\tau,g_\tau) \geq \|H_h\cK_hg_\tau\|^2 + \alpha C_\stab^{-1}\cB(g_\tau,g_\tau) \geq \alpha C_\stab^{-1}C_* \|g_\tau\|_{p,\Omega}^2
\end{align}
with the coercivity of $\cB$  from \eqref{b_bounds}  in the last inequality. This proves that $\A_h^\tau$ is coercive with $C_{10}(\alpha):=\alpha C_*C_\stab^{-1}$.
\par The bound \eqref{3.11} shows $|\cL_h^\tau(g_\tau)|\leq C_{11}\|m\|\|g_\tau\|_{p,\Omega}$ for any $g_\tau\in F_{\tau,p}^\ell$ and proves the continuity of a linear functional $\cL_h^\tau$. Hence the Lax-Milgram lemma concludes the proof.
\end{proof}
\begin{remark}
Note that the  solution $f_r$ to the regularized problem \eqref{eq:cts_IP} and the  solution $f_h^\tau$ to the discrete inverse solution \eqref{eq:weak-inv-h}  depend on $\alpha$. Suppresion of $\alpha$  in $f_r$ and $f_h^\tau$  is just for the notational convenience, but we track dependency of constants on  $\alpha$ in the error estimates.
\end{remark}
\begin{proposition}[Interpolation estimates for inverse problem \cite{MR4070337}] \label{prop:est-inter}
	For every $f\in H^{q}(\Omega)$ and $q\geq p$, there exists an interpolant $I_\tau f\in F_{\tau,p}^{\ell}$ of $f$ with 
	\begin{align*}
	\| f - I_\tau f\|_{p,\Omega} \leq C_\I^*\tau^{q-p}\|f\|_{q,\Omega} \qquad \text{for }\;0\leq p\leq \ell+1.
	\end{align*}
\end{proposition}
Recall the true source field $f_{\text{true}}$, the  solution $f_r$ to the regularized problem, and the solution $f_h^\tau$ to the discrete inverse problem. Our aim is to estimate $f_{\text{true}}-f_h^\tau=(f_{\text{true}}-f_r)+(f_r-f_h^\tau)$. The error $f_{\text{true}}-f_r$ is estimated in \cite{MR3245124,nair2021conforming}. Hence we focus on the discretisation error $f_r-f_h^\tau$ in Theorem~\ref{thm:dis-error}.
\begin{remark}[noisy measurement]
   Given the true source field $f_{\text{true}}$, the measurement $m=H\cK f_{\text{true}}$ can be noisy. This noisy measurement, denoted by $m^\delta$ for $\delta>0$, can be obtained as $m^\delta=m+n$  with the additive noise $n$ and $\|n\|\leq \delta$. Let $f_r^\delta$ solve  \eqref{eq:cts_IP} for the noisy measurement $m^\delta$. Then $\|f_r-f_r^\delta\|_{p,\Omega}\leq \frac{1}{2\sqrt{C_*}}\frac {\delta}{\sqrt{\alpha}}$ (see \cite[Theorem~3.4]{nair2021conforming} for a proof). Hence  an optimal choice of $\alpha$ depending on $\delta$ for a fixed $N$ is $\alpha\approx \delta^{2/3}$, and consequently $\|f_r-f_r^\delta\|_{p,\Omega}\lesssim \delta^{2/3}$. 
\end{remark}
\begin{theorem}[discretisation error]\label{thm:dis-error}
Let $f_r$ be the solution to the regularized problem \eqref{eq:cts_IP} and $f_h^\tau$ be the solution to the discrete problem \eqref{eq:weak-inv-h}. Let $u\in H^s(\Omega)$ for $s\geq 1$, $f\in H^p(\Omega)\cap H^{s-2}(\Omega)$ for $p\geq 0$, $f_r\in H^q(\Omega)$ for $q\geq p$,  $\ell\geq q-1$ for the polynomial degree $\ell$ of the virtual element space $F_{\tau,p}^\ell$, and $\xi_i\in H^r(\Omega)$ for $r\geq 1$ and for the solutions $\xi_i$ of auxiliary problems \eqref{ap}. Then under the assumptions \ref{A1}-\ref{A3}, there exists a positive constant $C_{12}(\alpha)$ such that \begin{align*}
    \|f_r-f_h^\tau\|_{p,\Omega}\leq C_{12}(\alpha) (h^{r+s-2}+\tau^{q-p})\|m\|.
\end{align*}
\end{theorem}

\begin{proof}[Proof of Theorem~\ref{thm:dis-error}]
Recall the interpolation $I_\tau f_r\in F_{\tau,p}^{\ell}$ from Proposition~\ref{prop:est-inter} and let $e_\tau:=I_\tau f_r-f_h^\tau\in F_{\tau,p}^{\ell}$. The coercivity of $\A_h^\tau$ from Theorem~\ref{wp:dis} and the discrete problem \eqref{eq:weak-inv-h} lead to
 \begin{align}
    C_{10}(\alpha) \|e_\tau\|_{p,\Omega}^2 &\leq \A_h^\tau(I_\tau f_r,e_\tau)- \cL_h^{\tau}(e_\tau)\nonumber\\&=\A_h^\tau(I_\tau f_r - \Pi_\ell f_r ,e_\tau)+(\A_h^\tau(\Pi_\ell f_r,e_\tau)-\A_r( \Pi_\ell f_r,e_\tau))\nonumber\\&\quad+\A_r(\Pi_\ell f_r-f_r,e_\tau) +(\cL_r(e_\tau)-\cL_h^\tau(e_\tau))\label{3.21}
 \end{align}
 with the regularized problem \eqref{eq:cts_IP} in the last step. The continuity of $\A_h^\tau$ from Theorem~\ref{wp:dis} for the first step and a triangle inequality for the second step show
 \begin{align}
    \A_h^\tau(I_\tau f_r- \Pi_\ell f_r,e_\tau)&\leq C_9(\alpha) \| I_\tau f_r- \Pi_\ell f_r\|_{p,\pw}\|e_\tau\|_{p,\Omega}\nonumber\\&\leq C_9(\alpha) (\|I_\tau f_r-f_r\|_{p,\Omega}+\|f_r- \Pi_\ell f_r\|_{p,\pw})\|e_\tau\|_{p,\Omega}\nonumber\\&\leq C_9(\alpha)(C_\I^*+C_\apx)  \tau^{q-p}\|f_r\|_{q,\Omega} \|e_\tau\|_{p,\Omega}
 \end{align}
 with   Proposition~\ref{prop:est-inter} and \ref{prop:est-polyPi} in the last step. The polynomial consistency in \eqref{b:cons} implies $\cB_\tau(\Pi_\ell f_r,e_\tau)=\cB_\pw(\Pi_\ell f_r,e_\tau)$. This and an elementary algebra lead to
 \begin{align}
  &\A_h^\tau(\Pi_\ell f_r,e_\tau)-\A_r( \Pi_\ell f_r,e_\tau) = (H_h\cK_h(\Pi_\ell f_r-f_r)-H\cK(\Pi_\ell f_r-f_r))^TH_h\cK_he_\tau \nonumber\\&\quad+ (H_h\cK_hf_r-H\cK f_r)^T H_h\cK_h e_\tau +(HK\Pi_\ell f_r)^T (H_h\cK_he_\tau-HKe_\tau).\label{4.28}
 \end{align}
 The bound for $\|H\|$ from \eqref{bd_meas} and the assumption \ref{A1} show  \begin{align}\|H\cK f\|\leq C_\h\|\cK f\|_{1,\Omega}\leq C_{13} \|f\|_{p,\Omega}\quad\text{for any $f\in F$ with $C_{13}:=C_\h C_{\text{reg1}}$}.\label{3.12}\end{align}
  A triangle inequality and Proposition~\ref{prop:est-polyPi} show $\|\Pi_\ell f_r\|_{p,\Omega} \leq \|\Pi_\ell f_r-f_r\|_{p,\Omega}+\|f_r\|_{p,\Omega}\leq (1+C_\apx)\|f_r\|_{p,\Omega}$. This, and  the estimates \eqref{3.11} and \eqref{3.12} in \eqref{4.28} prove 
 \begin{align*}
 &\A_h^\tau(\Pi_\ell f_r,e_\tau)-\A_r( \Pi_\ell f_r,e_\tau)\leq C_{11}(1+C_{11}+C_{13})\|e_\tau\|_{p,\Omega}(\|f_r-\Pi_\ell f_r\|_{p,\pw} \\&\hspace{3cm}+ \|H\cK f_r-H_h\cK_h f_r\|) + C_{13}(1+C_\apx)\|f_r\|_{p,\Omega} \|H\cK e_\tau - H_h \cK_h e_\tau\|\nonumber\\&\leq C_{14}(\tau^{q-p}\|f_r\|_{q,\Omega}+h^{s+r-2}\|f_r\|_{p,\Omega})\|e_\tau\|_{p,\Omega}
 \end{align*}
with Proposition~\ref{prop:est-polyPi},  \eqref{3.19} for $f_r$ and $e_\tau$, and $C_{14}:=C_{11}(1+C_{11}+C_{13})(C_\apx+C_8)+C_8C_{13}(1+C_\apx)$ in the last step. The definition of $\A_r$, \eqref{3.12}, and \ref{A0} result in 
\begin{align}
\A_r(\Pi_\ell f_r-f_r,e_\tau) &\leq (C_{13}^2+ \alpha C^*)\|\Pi_\ell f_r-f_r\|_{p,\Omega}\|e_\tau\|_{p,\Omega}\nonumber\\&\leq C_\apx(C_{13}^2+\alpha C^*)\tau^{q-p}\|f_r\|_{q,\Omega}\|e_\tau\|_{p,\Omega}
\end{align}
with Proposition~\ref{prop:est-polyPi} in the last step.
The definitions of $\cL_r$ and $\cL_h^\tau$, and \eqref{3.19} for $e_\tau$ provide
 \begin{align}
  |\cL_r(e_\tau)-\cL_h^\tau(e_\tau)|\leq \|m\|\|H\cK e_\tau -H_h\cK_h e_\tau \|\leq C_8 h^{s+r-2}\|m\|\|e_\tau\|_{p,\Omega}.\label{3.24}
 \end{align}
 The combination \eqref{3.21}-\eqref{3.24} and  \ref{A3} result in $\|e_\tau\|_{p,\Omega}\leq C_{10}^{-1}(\alpha)C_{15}(\alpha) (h^{s+r-2}+\tau^{q-p})\|m\|$ with $C_{15}(\alpha):=C_{\text{reg3}}(C_9(\alpha)(C_\I^*+C_\apx)+C_{14}+C_\apx(C_{13}^2+\alpha C^*)+C_8$. Note that $C_{10}^{-1}(\alpha)\propto \alpha^{-1}$, which comes from the coercivity of $\A_h^\tau$. This and Proposition~\ref{prop:est-inter} in the triangle inequality
 \[\|f_r-f_h^\tau\|_{p,\Omega}\leq \|f_r-I_\tau f_r\|_{p,\Omega}+\|I_\tau f_r-f_h^\tau\|_{p,\Omega}\leq (C^*_\I+C_{10}^{-1}(\alpha)C_{15}(\alpha))(h^{s+r-2}+\tau^{q-p})\|m\|\]
 conclude the proof with $C_{12}(\alpha):=C^*_\I+C_{10}^{-1}(\alpha)C_{15}(\alpha)$.
 \end{proof}

 \begin{remark}[comparison with \cite{MR3245124,nair2021conforming}]
 An intermediate problem is introduced in the conforming FEM for the Poisson inverse source problem \cite[Theorem~3.8]{MR3245124} and the Galerkin orthogonality provides a simpler proof therein. The analysis for the inverse biharmonic problem in  \cite{nair2021conforming} also considers intermediate problem and  is based on a companion operator. In this VEM analysis,  we avoid both  intermediate problem and companion for the inverse problem. 
 \end{remark}

\section{Numerical results} \label{sec:numer}
This section demonstrates numerical examples for general second-order linear elliptic problems  and  Poisson inverse  source problems in two subsections.
\par Since an explicit structure of 
 the discrete solution $u_h$ is not feasible, we compare $u$ with the projection $\Pi_ku_h$ of the discrete solution $u_h$. Also if the exact solution is  not known,  we compare the discrete solution 
 $\Pi_k u_h^{\text{final}}$ at the finest level  to the solution $\Pi_ku_h^j$ at each level $j$. Note that $\Pi_k$ also depends on each refinement level $j$. In all the experiments below, we assume $k=1$, and the relative $H^1$ and $L^2$ errors are computed using  
 \[\text{err}_d(u):=\frac{|u-\Pi_1u_h|_{d,\pw}}{|u|_{d,\Omega}}\;\text{and}\;\text{err}_d(u_h):=\frac{|\Pi_1u_h^{\text{final}}-\Pi_1u_h|_{d,\pw}}{|\Pi_1u_h^{\text{final}}|_{d,\pw}}\quad\text{for}\;d=0,1.\] 
 
\subsection{General second-order problems with modified scheme}
The conforming VEM for general second-order problems is discussed in \cite{MR3460621,MR3671497} with the various benchmark examples for $f\in L^2(\Omega)$ and $Q=\Pi_k$.  Refer to \cite{beirao2014hitchhiker} for the details on the implementation of VEM applied to the Poisson problem.
\subsubsection{Academic example}\label{sec:acad_ex}
The exact solution $u(x,y) = \sin(\pi x)\sin(\pi y)$ solves the general second-order indefinite (non-coercive) problem \eqref{cp} with the coefficients \[\ba = \begin{pmatrix} y^2+1& -xy \\ -xy&x^2+1\end{pmatrix},\quad \bb = (x,y), \quad\text{and}\quad\gamma= x^2+y^3.\]
We perform numerical tests on a sequence of $25,100,400,1600$, and $6400$ nonconvex polygonal subdomains and observe that the errors compare for the two choices of $Q=\Pi_1$ and $Q=J$ for this example.
\begin{figure}[H]
	\centering
	\begin{subfigure}{.5\textwidth}
		\centering
		\includegraphics[width=1\linewidth]{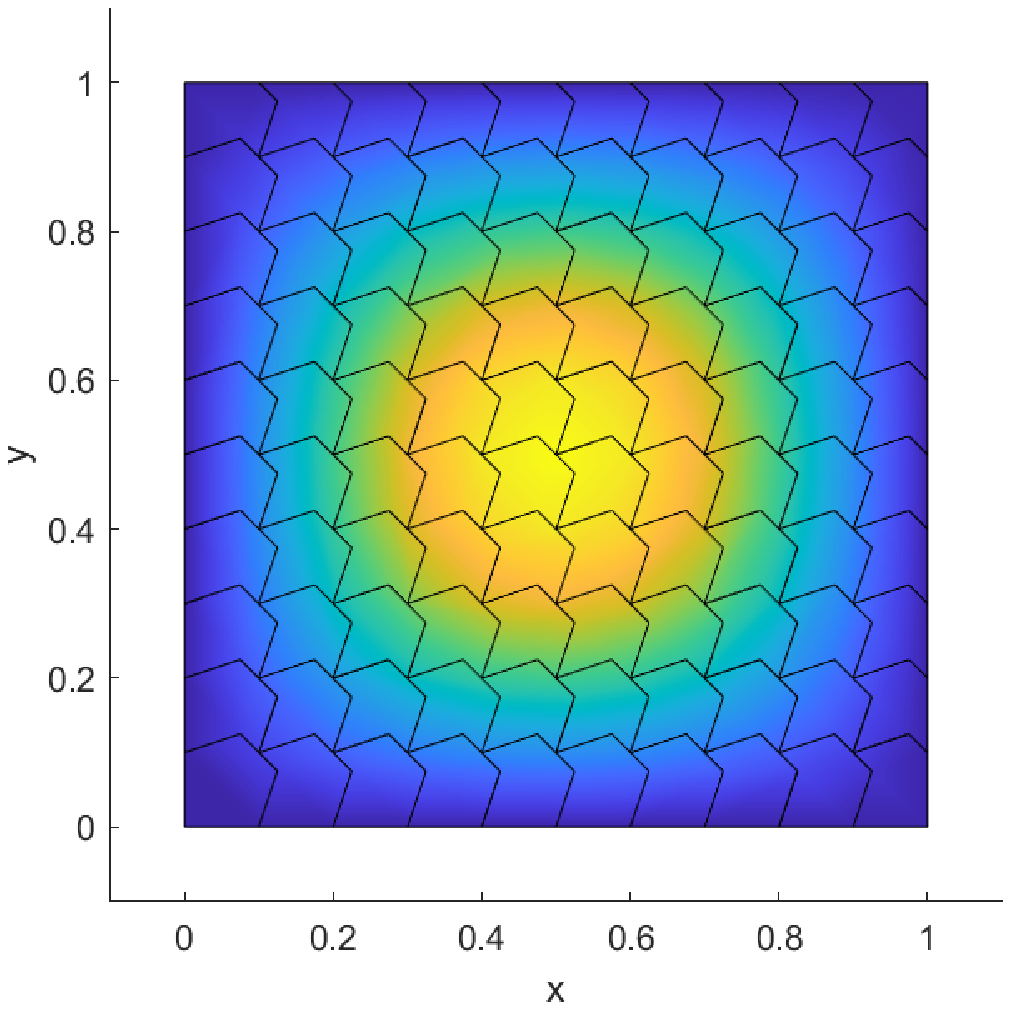}
	\end{subfigure}%
	\begin{subfigure}{.5\textwidth}
		\centering
		\includegraphics[width=1\linewidth]{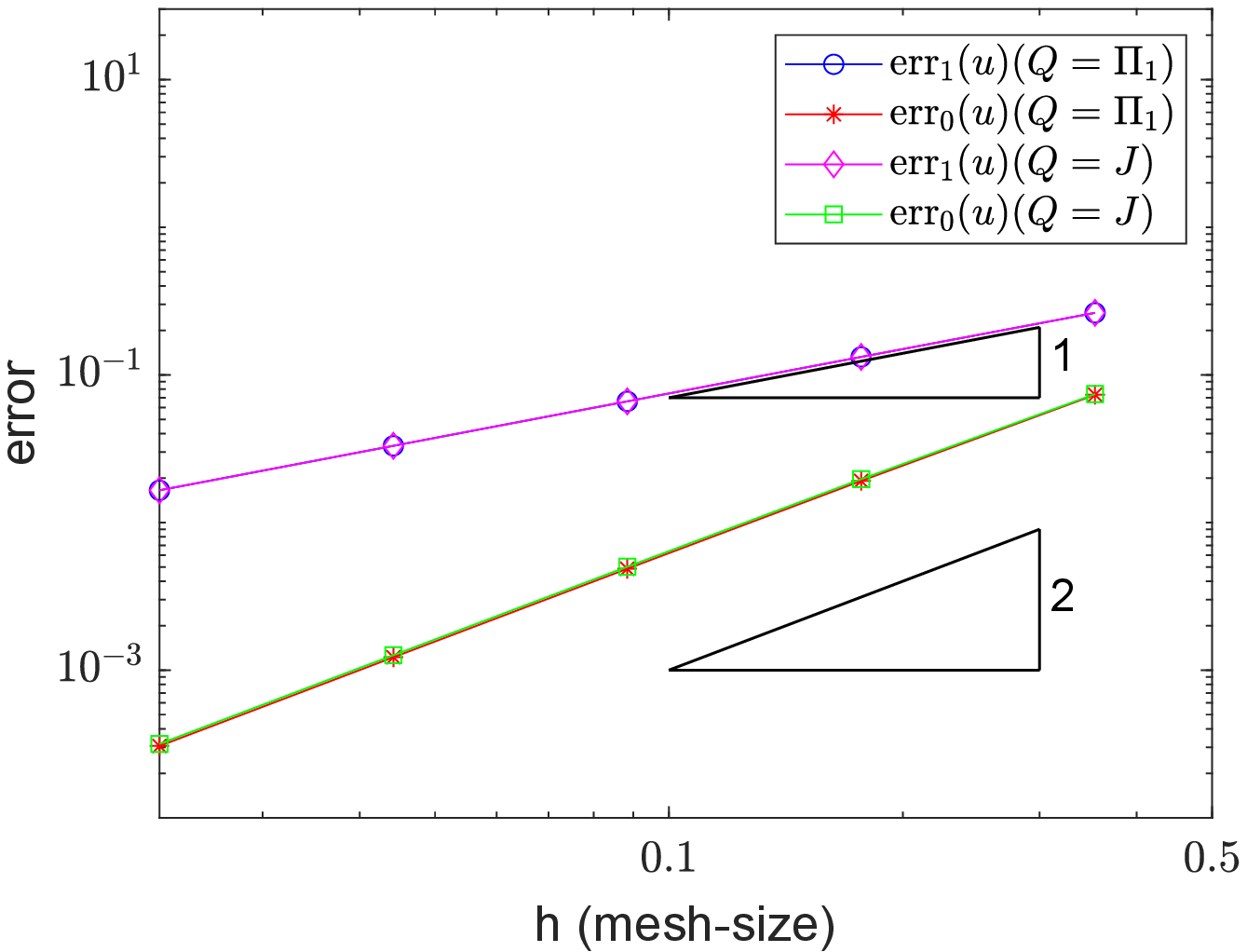}
	\end{subfigure}
		\caption{Polygonal mesh of 100 nonconvex elements (left) and  convergence plot for two choices of $Q$ (right).}
	\label{fig:PJ}
\end{figure}
\begin{table}[H]
	\begin{center}
		\begin{tabular}{|c|c|c|c|c|}
			\hline
			&\multicolumn{2}{c|}{$Q=\Pi_1$}&\multicolumn{2}{c|}{$Q=J$}\\
			\hline
			$h$&$\text{err}_1(u)$&$\text{err}_0(u)$&$\text{err}_1(u)$&$\text{err}_0(u)$\\
			\hline
			\hline	
   0.35355&0.26312&0.073286&0.26328&0.074126\\
0.17678&0.13226&0.019136&0.13231&0.019657\\
0.08838&0.06615&0.004860&0.06616&0.005015\\
0.04419&0.03306&0.001222&0.03307&0.001262\\
0.02209&0.01653&0.000306&0.01653&0.000316\\
\hline
   \end{tabular}
   \end{center}
   \caption{Energy and $L^2$ errors for two choices of $Q$.}
   \label{table:ex1_PJ}
   \end{table}

\subsubsection{Point load}
This subsection considers the general second-order problem \eqref{cp} with a point source $f=\delta_c$ supported at $c$ and the discrete problem \eqref{dp} with $Q=J$.  Theorem~\ref{thm:J}.a simplifies the discrete right-hand side to
\begin{align}f_h(v_h)=f(Jv_h)=\delta_c(Jv_h)=Jv_h(c).\label{J_pt}\end{align}
Since $u\in H^1(\Omega)$ and the dual solution $\Phi\in H^{\widetilde{s}}(\Omega)$ for $\widetilde{s}\geq 2$ from \eqref{adjoint} on a square domain $\Omega=(0,1)^2$, we expect from Theorem~\ref{thm:err} and also numerically observe (see Table~\ref{table:ex1}-\ref{table:ex2})  convergence rate of the error in the $L^2$ norm as  $\min\{\widetilde{s}-1,1\}=1$.

\medskip
\begin{example}\label{ex:5.1}
{\rm  We consider the Poisson  problem ($\ba=1, \bb=0$ and $\gamma=0$ in \eqref{cp}) with $f=\delta_c$ for $c=(0.1,0.1)$ and   a sequence of Voronoi meshes with $5, 25, 100, 400, 1600$, and  $6400$ number of polygonal subdomains. Since $u\in H^1(\Omega)$, it is enough to take the polynomial degree $k=1$. Note that for $k=1$ and $v_h\in V_h^k$, the definition of $\pid_1$ from \eqref{pid} implies $\Pi_0 \nabla v_h =\nabla \pid_1 v_h$ and hence the discrete formulations \eqref{eq:fweak-VE} and \eqref{dp} coincide in this particular case.  Even though we do not obtain an order of   convergence in the $H^1$ norm, Table~\ref{table:ex1} indicates that the $H^1$ error  decreases after a few refinements.  }
\end{example}
 \begin{example}\label{ex:5.2} 
 {\rm This  is a general second-order indefinite (non-coercive) problem with the coefficients $\ba,\bb,\gamma$ from Subsection~\ref{sec:acad_ex}
 and with the point load $f=\delta_c$ for $c=(0.431260,0.438584)$. Figure~\ref{fig:pt_force} displays an initial square distorted mesh and we refine the mesh at each level by connecting the mid-points of the edges to the centroid. This leads to a sequence of quadrilateral meshes and the point $c$ belongs to the set of vertices $\mathcal{V}$ in every refinement. In this case, the companion $J$  need not be computed explicitly and the identity \eqref{J_pt} reduces to $f_h(v_h)= Jv_h(c)=v_h(c)$ from Theorem~\ref{thm:J}.a.   In Examples~\ref{ex:5.1}-\ref{ex:5.2}, we treat the discrete solution $u_h^6$ at the $6^{\text{th}}$ refinement level as the exact solution.}
 \end{example}
 \begin{table}[H]
	\begin{center}
		\begin{tabular}{|c|c|c|c|c|}
			\hline
			$h$&$\text{err}_1(u_h)$&conv. rate&$\text{err}_0(u_h)$&conv. rate\\
			\hline
			\hline	
0.33657&0.95732&0.29940&0.65875&1.4923\\
0.17849&0.79174&0.10103&0.25565&1.0226\\
0.09377&0.74190&0.09829&0.13238&0.6322\\
0.04798&0.69461&0.35248&0.08666&1.4682\\
0.02434&0.54681&-&0.03199&-\\
0.01221&0&-&0&-\\
\hline
   \end{tabular}
   \end{center}
   \caption{Energy and $L^2$ errors with respective convergence rates in Example~\ref{ex:5.1}.}
   \label{table:ex1}
   \end{table}
\begin{figure}[H]
	\centering
	\begin{subfigure}{.5\textwidth}
		\centering
		\includegraphics[width=1\linewidth]{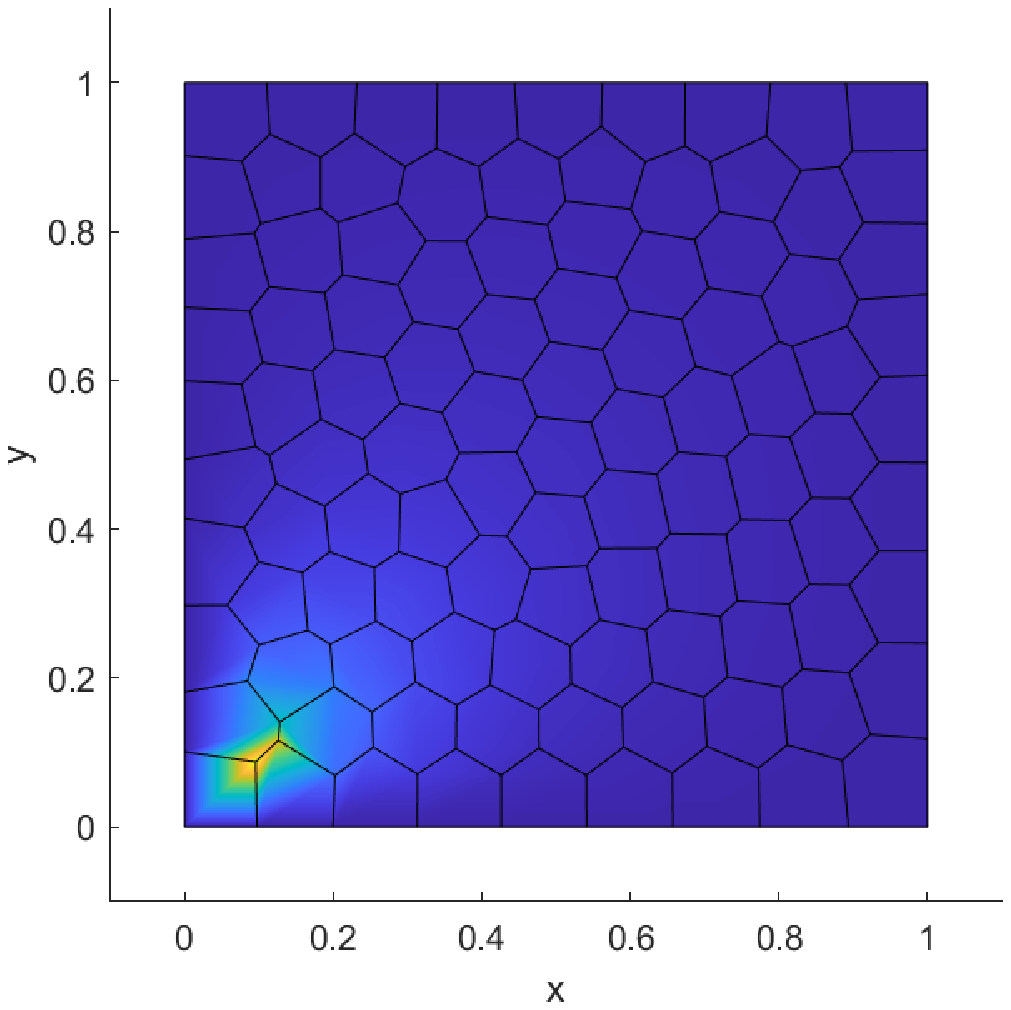}
	\end{subfigure}%
	\begin{subfigure}{.5\textwidth}
		\centering
		\includegraphics[width=1\linewidth]{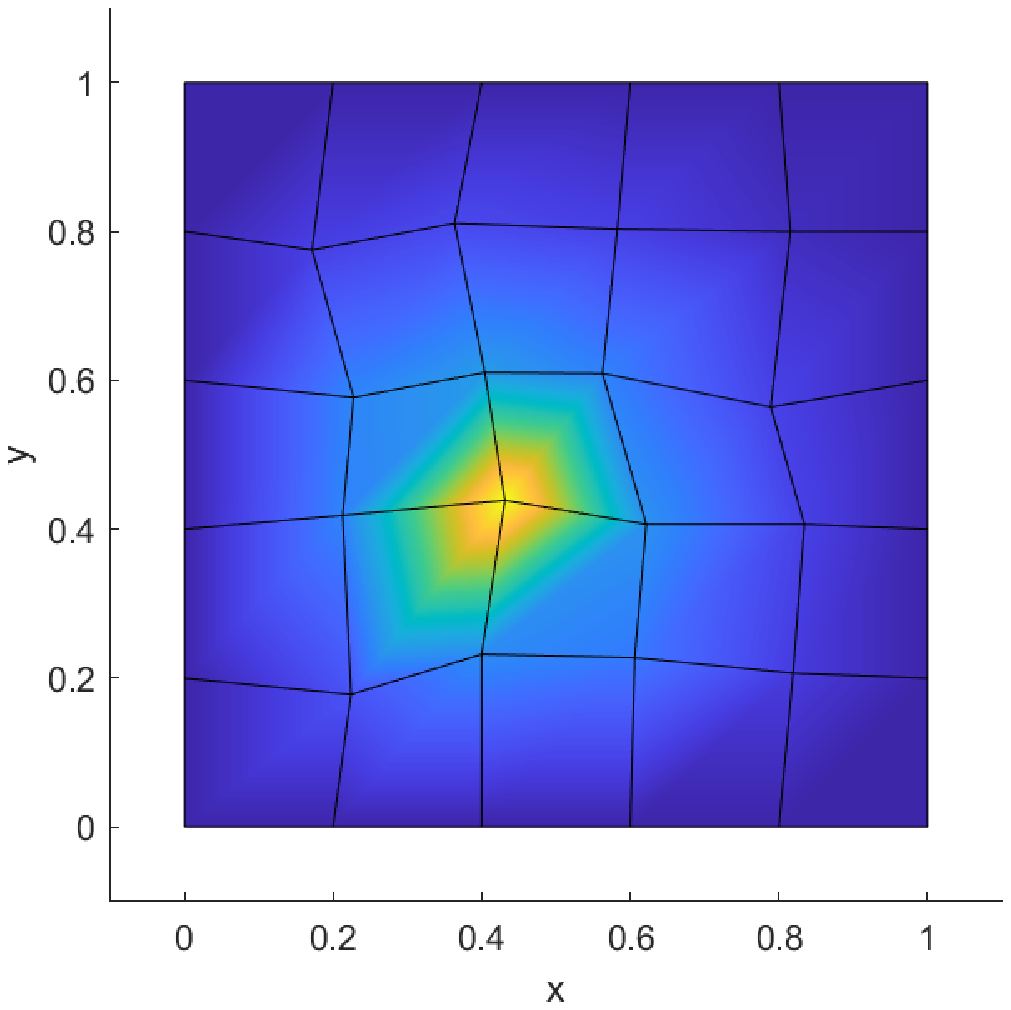}
	\end{subfigure}
		\caption{Voronoi mesh of 100 elements in Example~\ref{ex:5.1} (left) and  square distorted mesh of 25 elements in Example~\ref{ex:5.2} (right).}
	\label{fig:pt_force}
\end{figure}
\begin{table}[H]
	\begin{center}
		\begin{tabular}{|c|c|c|c|c|}
			\hline
			$h$&$\text{err}_1(u_h)$&conv. rate&$\text{err}_0(u_h)$&conv. rate\\
			\hline
			\hline	
0.28284&0.66028&0.11542&0.142200&0.93567\\
0.14142&0.60951&-0.04152&0.074344&0.99145\\
0.07097&0.62721&0.62531&0.037531&1.06480\\
0.03548&0.40661&0.22496&0.017942&1.08240\\
0.01772&0.34784&-&0.008465&-\\
0.00886&0&-&0&-\\
\hline
   \end{tabular}
   \end{center}
   \caption{Energy and $L^2$ errors with respective convergence rates in Example~\ref{ex:5.2}.}
   \label{table:ex2}
   \end{table}

\subsection{Inverse Problem}
The  algorithm below highlights the two main parts and subsequent major steps in each part in the VEM implementation of the discrete inverse problem. For the sake of simplicity, assume $p=1$ in the discrete virtual element space $F_{\tau,p}^\ell$ for the inverse problem (so that $V_h^k=F_{\tau,p}^{\ell}$). { The choice of $\mathcal{B}(\cdot,\cdot)$ in this paper is $a( \cdot,  \cdot)$, and the discrete bilinear form $\mathcal{B}_{\tau}(\cdot,\cdot)$ in assumption \ref{A4} is chosen as $a_h(\cdot, \cdot)$ from \eqref{2.18}. 

\medskip
\hrule
\vspace{0.2cm}
{\large \textbf{Algorithm}}
\vspace{0.2cm}
\hrule
\vspace{0.2cm}
{\small
\noindent \textbf{Part I - Forward problem}
\begin{enumerate}
\item Compute the projection matrix of $\pid_k$.
\item Solve the forward problem \eqref{eq:fweak-VE} for each basis function $\varphi_j$ of $V_h^k$ as a source function ($f=\varphi_j$) and denote the solution vector by $U_j=[U_j^r]_{r=1,\dots,\text{Ndof}}$.
\item Write $\cK_h\varphi_j = \sum_{r=1}^{\text{Ndof}}U_j^r\varphi_r$.
\end{enumerate}
\textbf{Part II - Inverse problem}
\begin{enumerate}
\item Compute the matrix $H_h\cK_h\varphi_j=\sum_{r=1}^{\text{Ndof}}U_j^rH(Q\varphi_r)$ to solve the discrete inverse problem \eqref{eq:weak-inv-h}.
\item  Compute $(H_h\cK_h\varphi_j)^T(H_h\cK_h\varphi_r)=((SU)^T(SU))_{jr}$ for $S_{ir}=\mathbb{h}_i(Q\varphi_r)$ with the matrix $S=[S_{ir}]_{i=1,\dots,N;r=1,\dots,\text{Ndof}}.$
\item Compute $\alpha$ for noise $n$ in measurement by solving $$\alpha = \underset{\alpha>0}{\arg\!\min}  \Big(\frac{ \alpha\,\|f_{\text{true}}\|}{\max(\text{eig}(L))} + \frac{\| n\|}{2 \sqrt{\alpha}} \Big) \text{ at finest mesh with $L_{ij}:=\cB_{\tau}(\eta_i^{\tau},\eta_j^{\tau})_{i,j=1,\dots,N},
$ }$$
 where $\eta_i^{\tau} \in F_{\tau,p}^{\ell}$ solves the discrete problem $\cB_{\tau}(\eta_i^{\tau},g_{\tau})=(\xi_i^h, g_{\tau})_{L^2(\Omega)}$ and $\xi_i^h \in V_h^k$ solves the discrete forward problem \eqref{eq:fweak-VE} with load function $\mathbb{h}_i$ for each $i=1, \dots,N$.

\item Depending on the choice of $\cB_\tau$, compute the projection matrices involved in $\cB_\tau$ and evaluate the matrix $B= [B_{jr}]_{j,r=1,\dots,\text{Ndof}}$ for 
 the term $\cB_\tau(\varphi_j,\varphi_r)$.
\item Compute the discrete right-hand side $m^TSU$.
\item Solve the linear system $A\mathcal{F}=R$ for $A= (SU)^T(SU)+\alpha B$ and $R=m^TSU$.
\end{enumerate}
\hrule
}
\subsubsection{Measurement functionals in $L^2(\Omega)$}

\medskip \noindent 
Recall that $m$ is the given measurement of the true forward solution $u$ and $m_h$ is the computed measurement of the discrete solution $u_h$; $f$ is the true inverse solution, $f_h$ is the solution to the discrete problem \eqref{eq:weak-inv-h} with mesh size $h$. The approximation errors $\text{err}(m)$ of measurement $m$, and errors $\text{err}_0(f)$ and $\text{err}_1(f)$ of the solution of the inverse problem  in $L^2$ and $H^1$ norm, and $\text{err}_d(f_h)$   are defined  by
\begin{equation} \label{eq:err_compute}
\begin{split}
\text{err}(m)=\frac{\|m - m_h\|}{\|m\|}, 
\quad  \text{err}_d(f)=\frac{\|f - \Pi_1 f_h \|_{d,\pw}}{\|f\|_{d,\pw}}, \\
\text{and} \quad
\text{err}_d(f_h)=\frac{\|\Pi_1f_h^{\text{final}} - \Pi_1f_{h}\|_{d,\pw}}{\| \Pi_1 f_h^{\text{final}}\|_{d,\pw}}, \qquad
\end{split}
\end{equation}
where $f_h^{\text{final}}$ is the solution of inverse problem \eqref{eq:weak-inv-h} at the finest mesh. 

Note that $f_r$ is not computable (even when $f$ is known) and is assumed as $f_h^{\text{final}}$ in order to verify the theoretical results. Moreover, the error $err_d(f)$ converges to relative error of regularised solution ($\text{err}_d(f_r) = \|f-f_r\|_{d,\Omega}/\|f\|_{d,\Omega}$) as the mesh-size decreases.

\begin{figure}[H]
	\begin{subfigure}{.5\textwidth}
		\centering
		\includegraphics[width=0.8\linewidth]{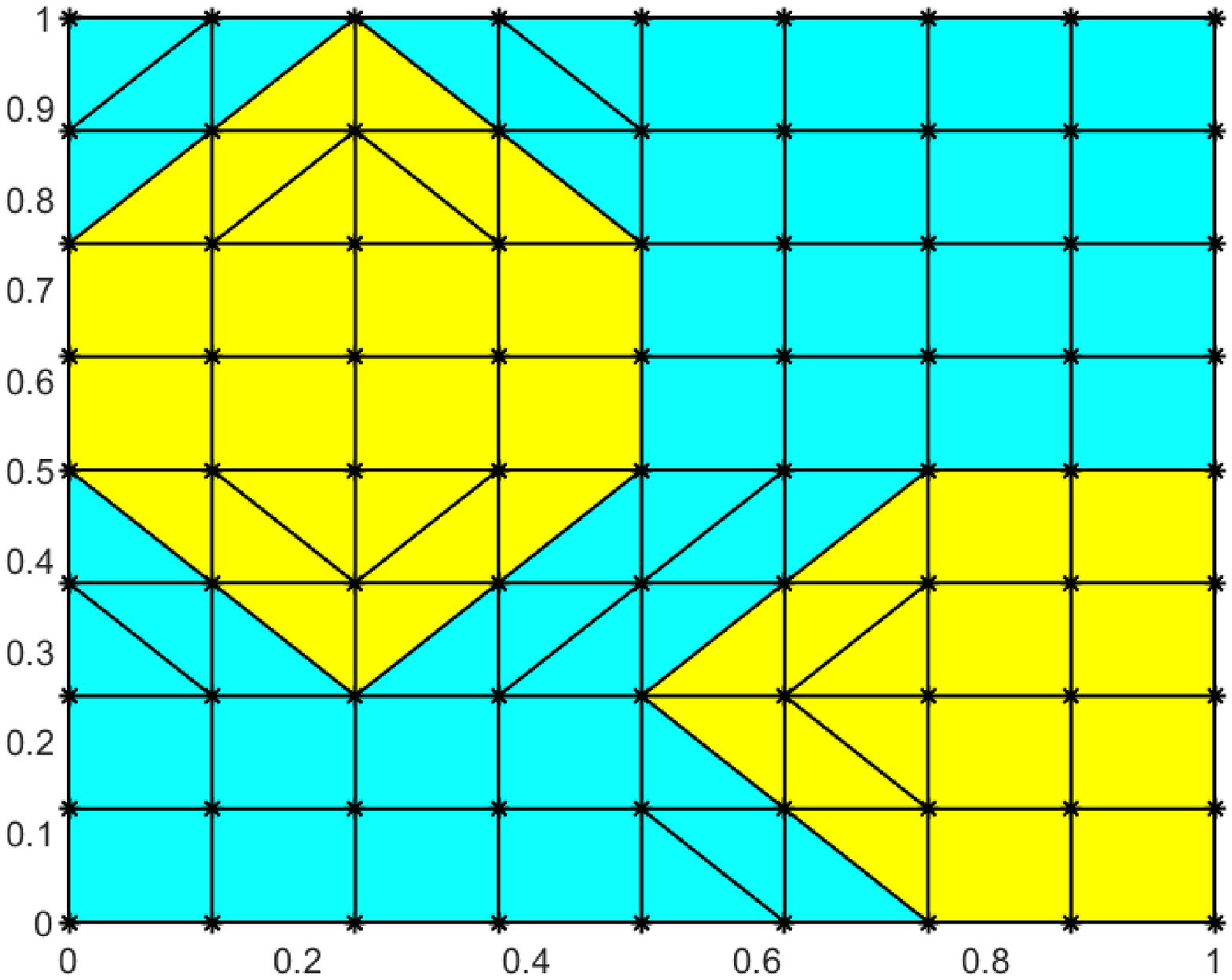}
	\end{subfigure}
	\begin{subfigure}{.5\textwidth}
		\centering
		\includegraphics[width=0.8\linewidth]{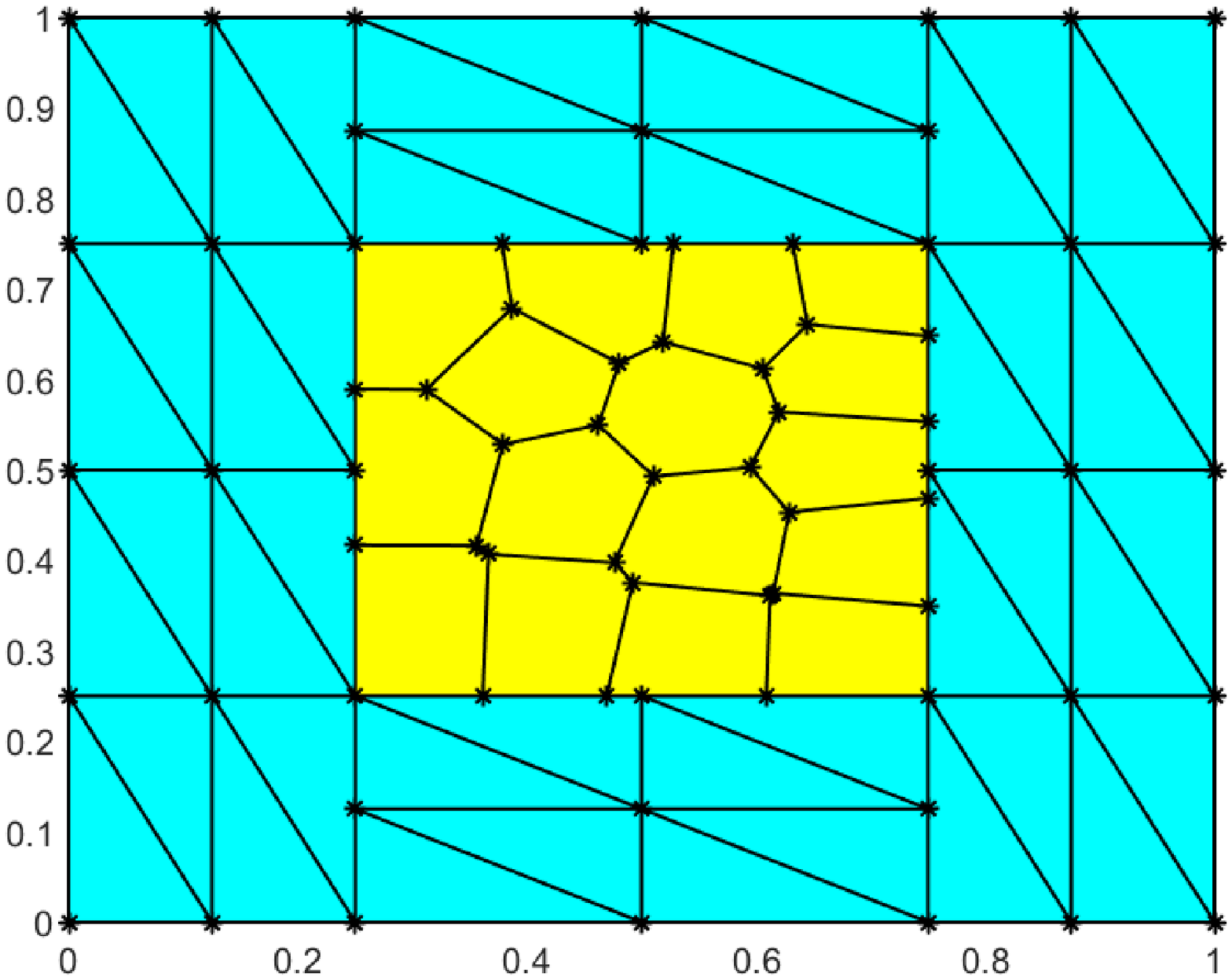}
	\end{subfigure}
	\caption{Meshing of domains in (a) Example~\ref{ex:5.4} (left) and (b) Example~\ref{ex:5.5} (right).}
	\label{fig:meshes}
\end{figure}



\begin{example}\label{ex:5.4}
 {\rm This example considers the  true solution for the forward problem as 
$$u(x,y)=\sin(\pi x) \sin(\pi y)$$ on the domain $\Omega:=[0,1]^2$,
and the true solution $f$ for the inverse problem is computed from the Poisson equation. The measurements of displacement $u$ are given in two subdomains $\omega_1$ and $\omega_2$
of $\Omega$ as shown in Figure \ref{fig:m-domains}.  
The measurement functionals $\mathbb{h}_i$ are defined as the  average of the solution $u$ on subdomains $\omega_i$ for $i=1,2$. That is,   the  measurement input $m$ is
\begin{equation} \label{eq:m_comp}
m(i)=\mathbb{h}_i(u) := \frac{1}{|\omega_i|} \int_{\omega_i} u \,\mbox{d}x \quad\text{for}\; i=1,2.
\end{equation}
In this example, the red refinement (new elements formed by joining midpoints of each old element) of the initial mesh partition is considered, and the mesh partition with $h=0.14$ is displayed in Figure~\ref{fig:meshes}(a). We determine the regularization parameter $\alpha$ as  $1.013e-05$ referring to \cite{nair2021conforming} on the finest mesh with choice of noise as $\| n\|/\|m\| = 2 \%$. Table \ref{tab:Inv_ex1} shows convergence rate $2$ for the $L^2$ error in measurement, and error with respect to a true solution is almost  constant after a few iterations. The iterative errors in $L^2$ and $H^1$ norms for the discrete solution $f_h$ converge with optimal rates (see Figure~\ref{fig:errors}(a)).
  \begin{table}[H]
	\begin{center}
		\begin{tabular}{|c|c|c|c|c|}
			\hline
   ${h}$&$\text{err}_0(m)$&$\text{err}_1(f)$&$\text{err}_1(f_h)$&$\text{err}_0(f_h)$\\
			\hline
			\hline	
         0.2800  &   0.118532   &   0.416440  & 0.439647  & 0.178973  \\
		0.1400 & 0.030081 & 0.209853 & 0.188650  & 0.043131  \\
		0.0700 & 0.007595  & 0.153089 & 0.090389 & 0.010682  \\
		0.0350 & 0.001906  & 0.137285 & 0.045747  & 0.002557  \\
		0.0175 & 0.000477  & 0.133165 & 0.021661 & 0.000533  \\
		0.0088 & 0.000119  & 0.132123 & -  & - \\
\hline
   \end{tabular}
   \end{center}
   \caption{Measurement error $m-m_h$, and source approximation error $f-f_h$ in the energy and $L^2$ norms.}
   \label{tab:Inv_ex1}
   \end{table}
   
   \begin{figure}[ht!]
	\centering
	\begin{subfigure}{.5\textwidth}
		\centering
		\includegraphics[width=1\linewidth]{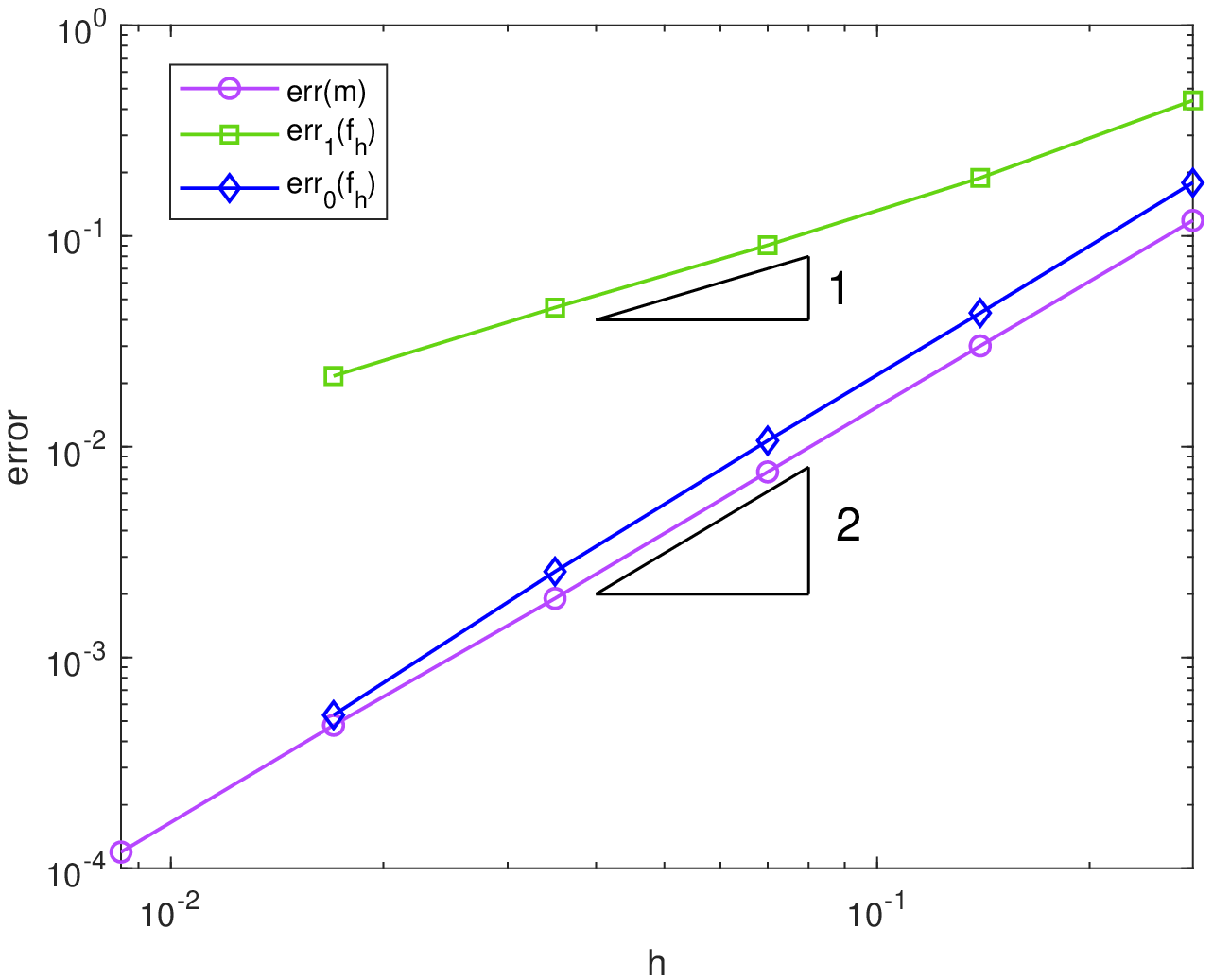}
	\end{subfigure}%
	\begin{subfigure}{.5\textwidth}
		\centering
		\includegraphics[width=1\linewidth]{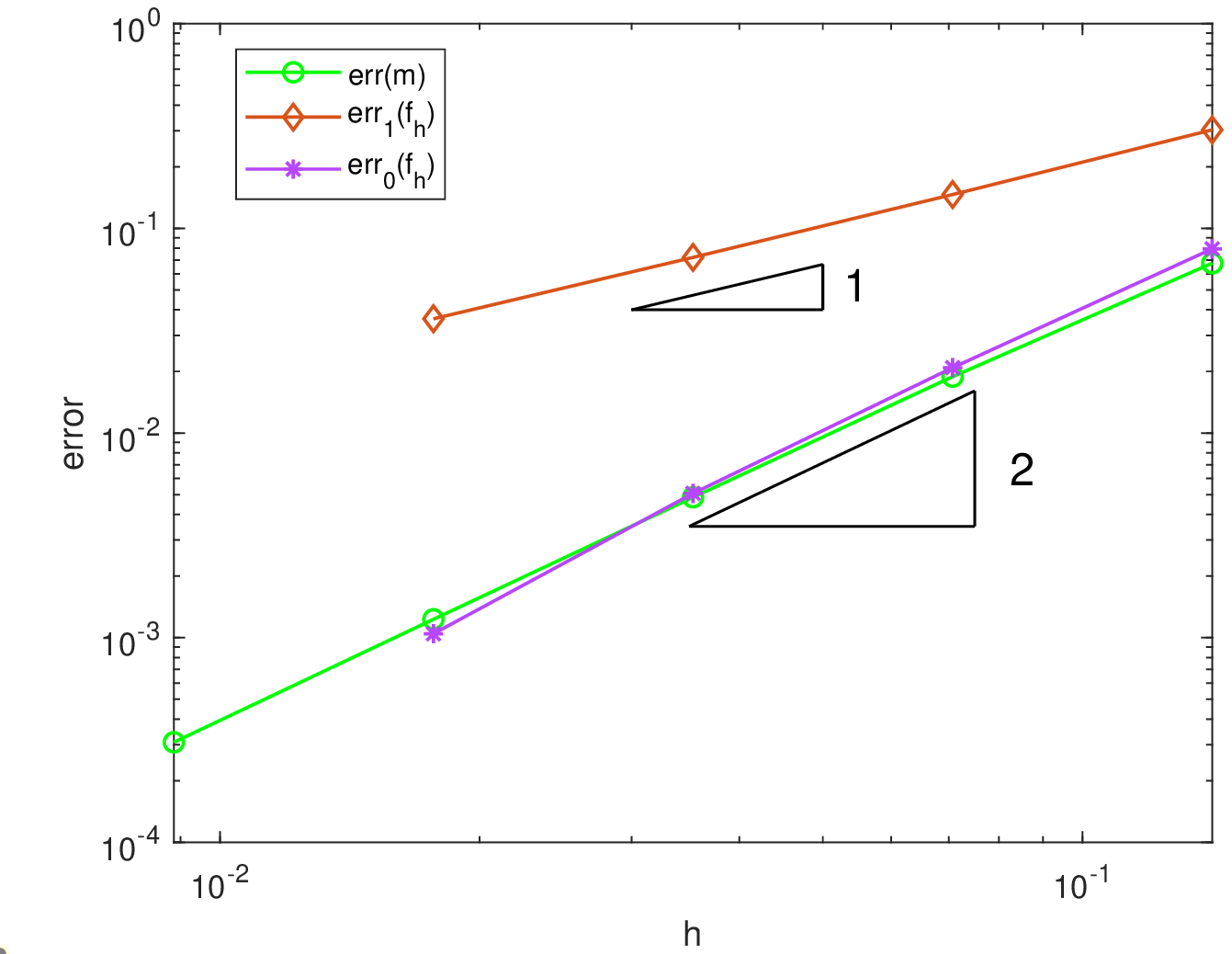}
	\end{subfigure}
		\caption{Error convergence plots for measurement $m$, and inverse solution $f_h$ in energy and $L^2$ norm in (a) Example~\ref{ex:5.4} and (b) Example~\ref{ex:5.5}.}
	\label{fig:errors}
\end{figure}}


\end{example}

\begin{example}\label{ex:5.5} 
{\rm 
The measurements of displacement $u$ are given in a subdomain $\omega:=[0.25,0.75]^2$ of the domain $\Omega$ (see Figure \ref{fig:m-domains}), and $m, m_h$ are computed from equation \eqref{eq:m_comp} {with computed regularization parameter $\alpha=1.847e-07$ at the finest mesh}. We have considered here the same true solutions as in Example~\ref{ex:5.2}, however, the triangulation {(combination of polygonal mesh for the measurement domain $\omega$ and red refinement for the rest)} of domain $\Omega$ includes the polygons. {The domain partition with mesh-size $h=0.1414$ is shown in Figure~\ref{fig:meshes}(b), the convergence results} are displayed in Table~\ref{tab:Inv_ex2} and the optimal rate of convergence in Figure~\ref{fig:errors}(b). 
 \begin{table}[H]
	\begin{center}
		\begin{tabular}{|c|c|c|c|c|}
			\hline
   ${h}$&$\text{err}_0(m)$ &$\text{err}_1(f)$ &$\text{err}_1(f_h)$&$\text{err}_0(f_h)$\\
			\hline
   \hline	
         0.1414 &  0.067571 & 3.41648 & 0.303115  & 0.079437  \\
		0.0707 & 0.018856 & 3.07953 &  0.146344  &  0.020854 \\
		0.0354 &  0.004852 & 2.98657 &  0.072144   & 0.005070  \\
		0.0177 &  0.001230 & 2.96281 &  0.036166  & 0.001045  \\
		0.0089 &  0.000308 & 2.95679 & - & -   \\
\hline
   \end{tabular}
   \end{center}
   \caption{Measurement error $m-m_h$, and source approximation error $f-f_h$ in the energy and $L^2$ norms.}
   \label{tab:Inv_ex2}
   \end{table}}


\end{example}




\subsubsection{Point Measurement}
This subsection deals with  rough measurements, in particular, the point loads and  computes the measurement error $m-m_h$ and the source approximation error $f-f_h$ in the energy and $L^2$ norm.   
\begin{example}\label{ex:5.6}
{\rm 
We consider uniform decompositions of the domain $\Omega= (0,1)^2$ into squares. Assume that the measurements of an exact solution $u$ are known at a few points (say $V_1, V_2, \dots, V_N$) in domain $\Omega$, that is, let $m=(u(V_i))_{i=1,\dots, N}$. The aim is to recover the approximate source function  with this information. Suppose that the measurement points are $V_1=(0.5,0.5), \,V_2=(0.75, 0.25), \,V_3=(0.25, 0.75),\, V_4=(0.25, 0.25),\, V_5=(0.75, 0.75),\, V_6=(0.125,0.375) $, and $V_7=(0.375, 0.375)$ as shown in Figure~\ref{fig:m-ex4}, and the exact solution $f$ is same as in Example~\ref{ex:5.4}. The square mesh are chosen such that the measurement points belong to the set of vertices, and they remain vertices in the next uniform refinements.
\begin{figure}[H]
		\centering
		\includegraphics[width=0.5\linewidth]{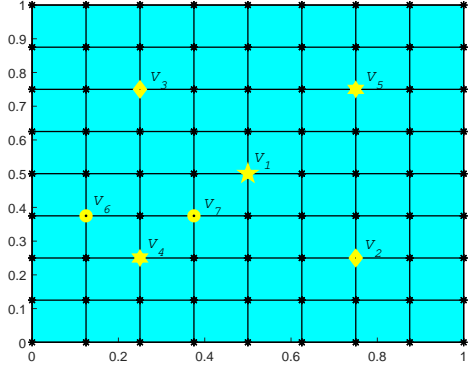}                
	\caption{The point measurements at $V_1,\dots,V_7$ in 
 a square mesh with $h=0.17678$. }
	\label{fig:m-ex4}
\end{figure}
   The numerical experiments demonstrate that the error decreases as the number of measurement points ``$N$" increases. See Table~\ref{tab:Inv_ex3}  for the values of approximate solution $f_h$ at the point $(0.5,0.5)$ with varying mesh-size $h$ and the number of measurements $N$, where the true value is  $f(0,5,0.5)=19.739209$. This is an interesting observation because we expect to recover a better approximation  $f_h$ with more information of $f$. The regularity indices in Assumptions~\ref{A1}-\ref{A3} are $s=2, r=1$ and $q=2$, and so theoretically we expect a linear order of  convergence  for $m-m_h$ from Theorem~\ref{thm:meas} and for $f-f_h$ in both  $H^1$ norm  $(p=1)$ and $L^2$ norm $(p=0)$   from Theorem~\ref{thm:dis-error}. Numerically we observe the expected rate for $f-f_h$ in the energy norm, but a better (quadratic) convergence rate  for $m-m_h$, and consequently for $f-f_h$ in the $L^2$ norm (see Figures~\ref{fig:errors_ex3}-\ref{fig:errors_ex4}).
\begin{table}[H]
	\begin{center}
		\begin{tabular}{|c|c|c|c|c|}
			\hline
   $h$ & $N=1$ & $N=3$ & $N=5$ & $N=7$
   \\	\hline
   \hline
   0.70711 & 28.526129 & 28.521174 & 28.511269 & 28.511269 \\
   0.35355 & 22.997621 & 22.188333 & 21.515906 & 21.501371 \\
   0.17678 & 22.065475  & 21.256892 & 20.547227 & 20.437930 \\
   0.08839 & 21.802501 &  20.975939 & 20.286060 & 20.170386 \\
   0.04419 & 21.722483 & 20.897970 & 20.213959 & 20.098021 \\
   0.02210 & 21.702987 & 20.876174 & 20.194335 & 20.078657 \\
\hline
   \end{tabular}
   \end{center}
   \caption{Approximate solution $f_h$ at $(0.5,0.5)$ for different mesh-sizes $h$ and number of measurement points $N$.}
   \label{tab:Inv_ex3}
   \end{table}
\begin{figure}[H]
	\centering
 \begin{subfigure}{.48\textwidth}
 \centering
 \includegraphics[width=1\linewidth]{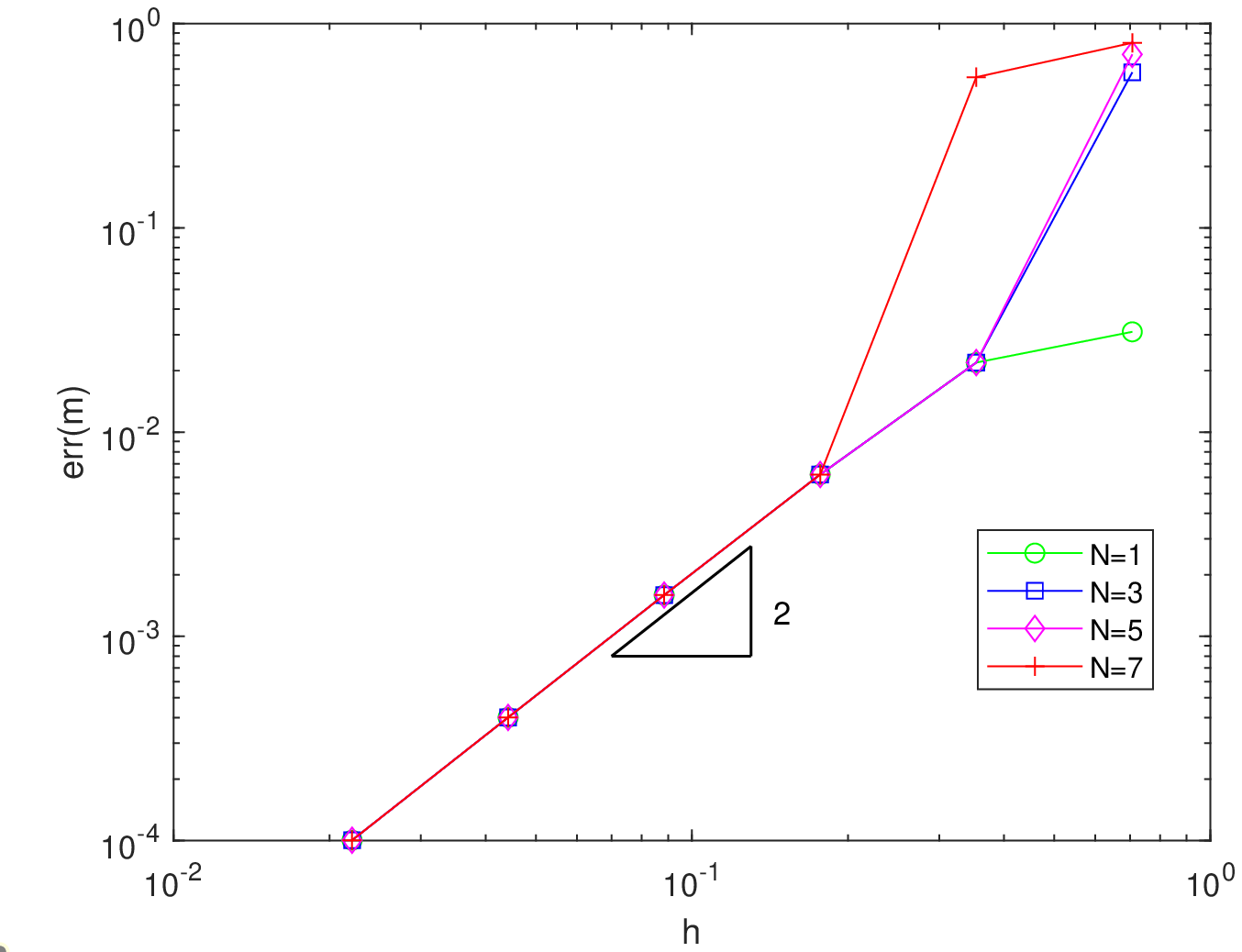} 
 \end{subfigure}
 \begin{subfigure}{.48\textwidth}
		\centering
		\includegraphics[width=1\linewidth]{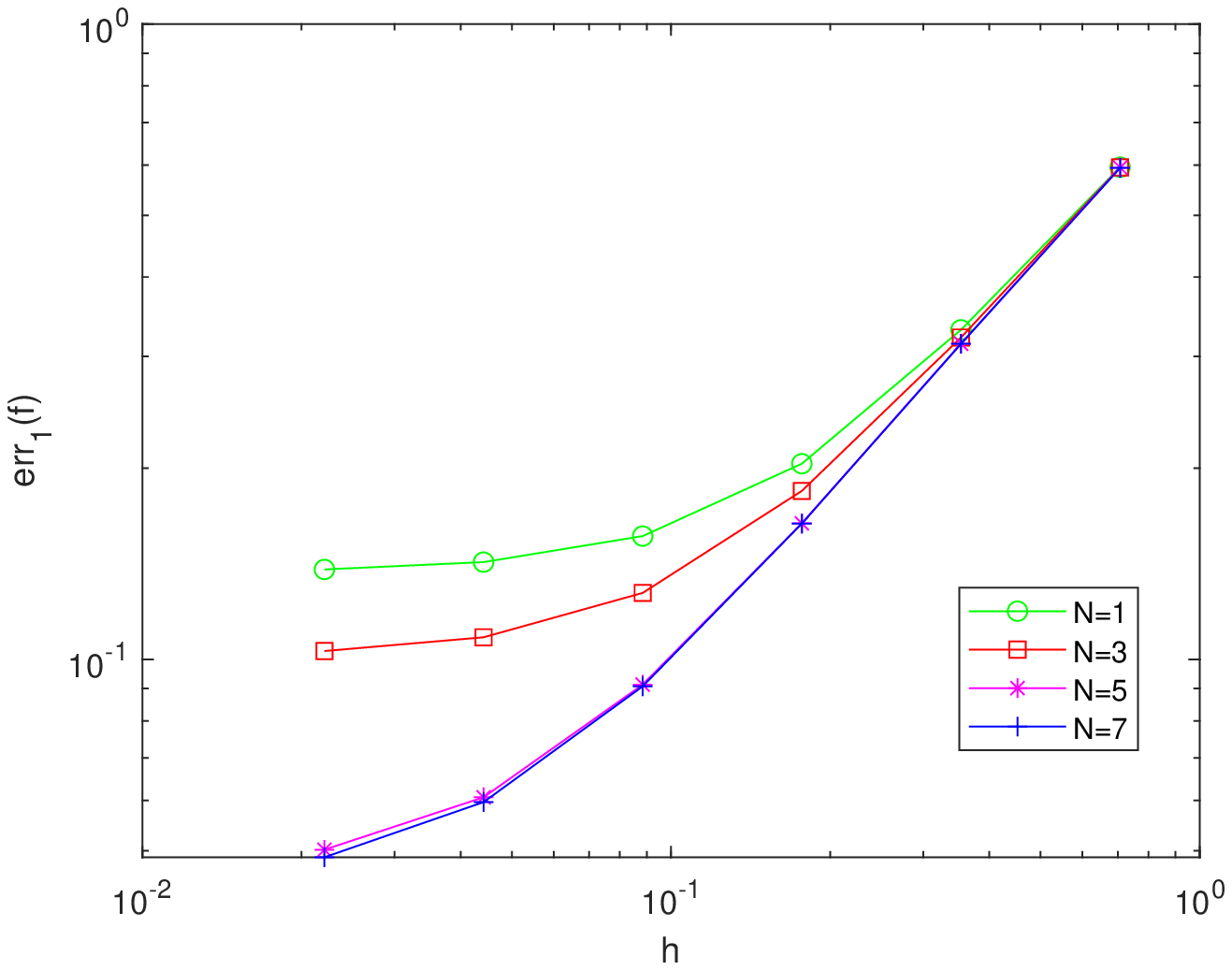}
	\end{subfigure}
 \caption{Error convergence plots of (a) relative $l^2$ error in $m$ (left) and (b) relative $H^1$ error in $f$ (right).}
 \label{fig:errors_ex3}
 \end{figure}
 \begin{figure}[H]
 \centering
	\begin{subfigure}{.48\textwidth}
		\centering
		\includegraphics[width=1\linewidth]{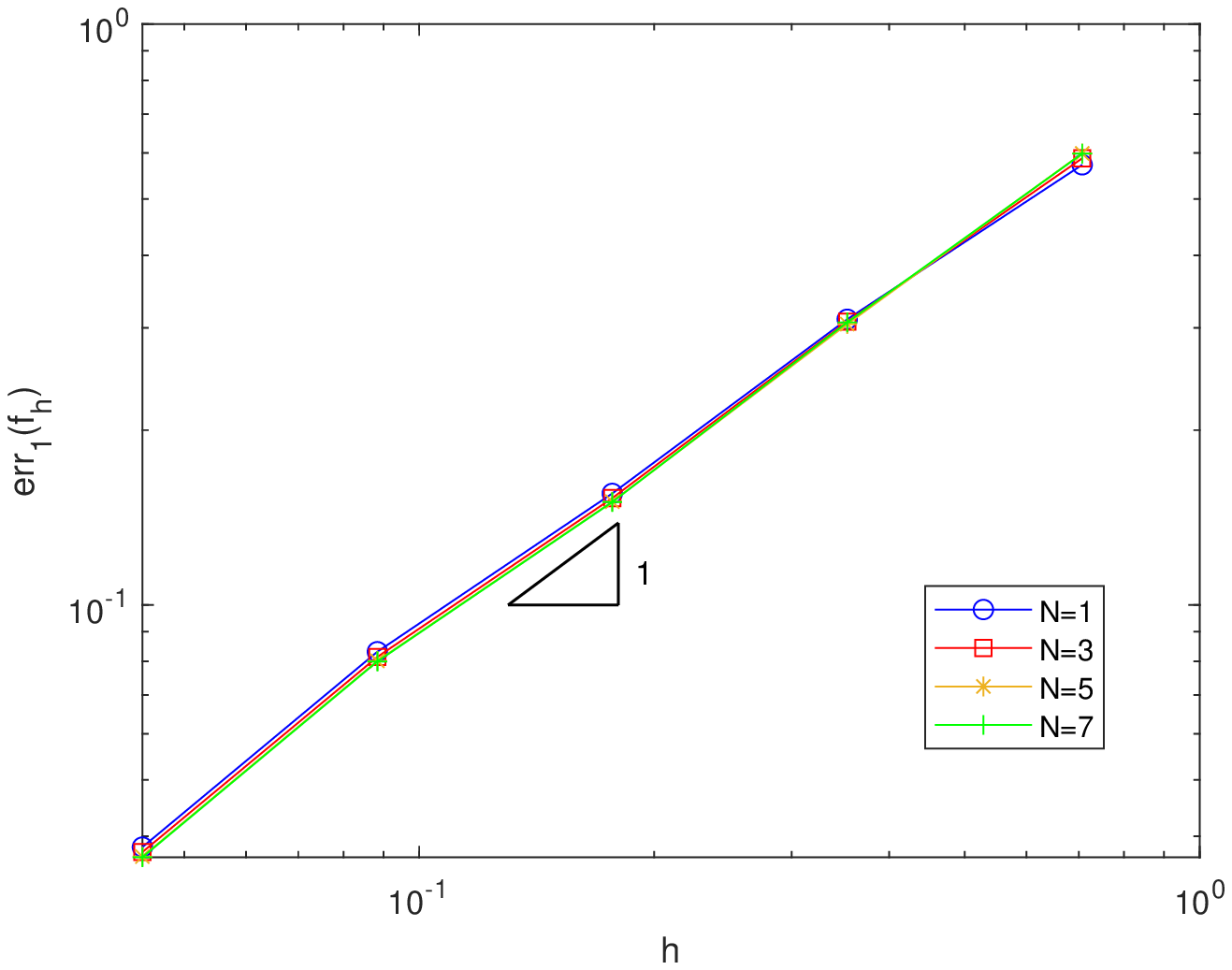}
	\end{subfigure}%
	\begin{subfigure}{.48\textwidth}
		\centering
		\includegraphics[width=1\linewidth]{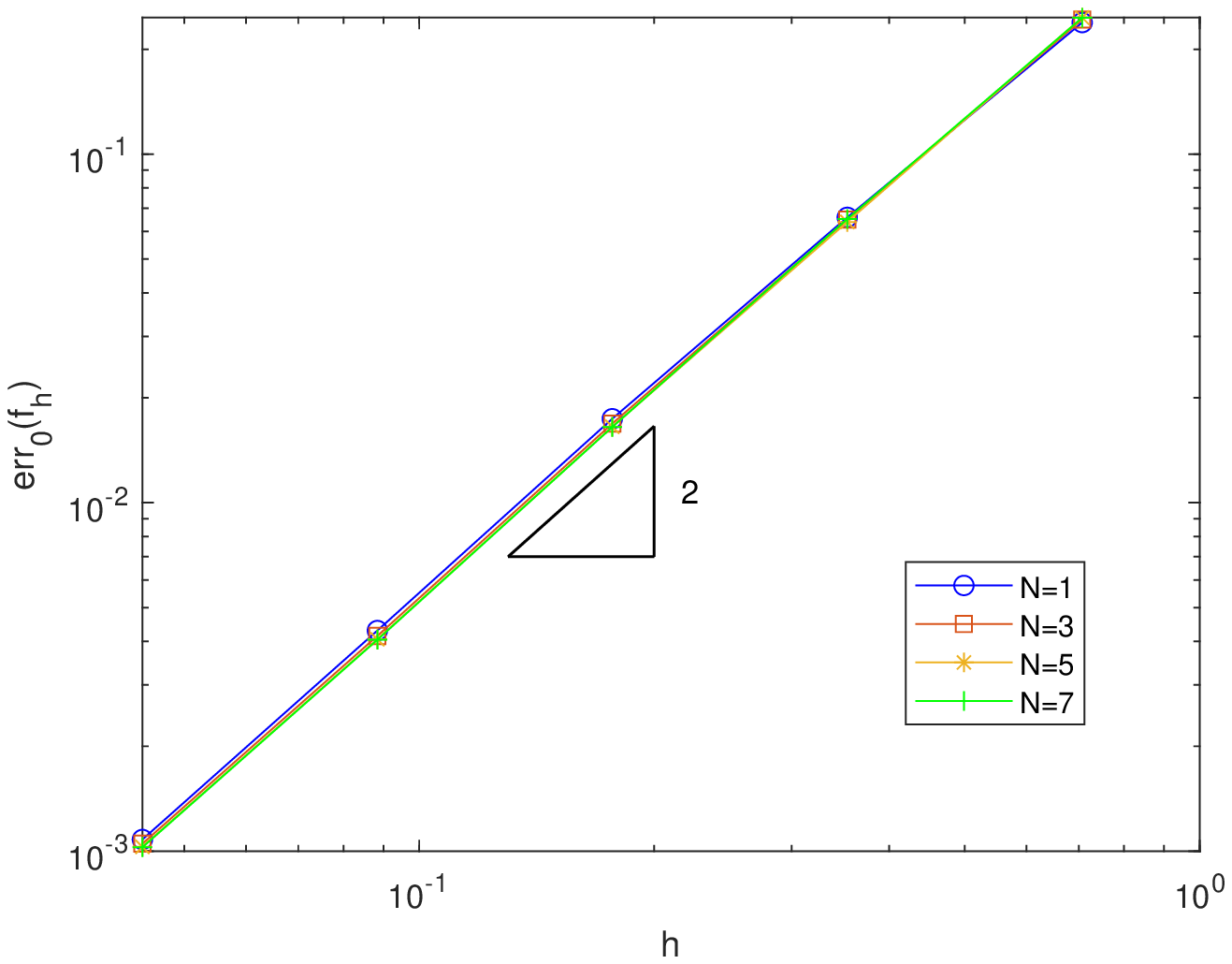} 
	\end{subfigure} 	
		\caption{Error convergence plots of (a) relative $H^1$ error in $f_h$ (left) and (b) relative $L^2$ error in $f_h$ (right).}
	\label{fig:errors_ex4}
\end{figure}
}
\end{example}


\subsection*{Conclusions}
This paper introduces the notion of  companion operators or smoothers for the  conforming VEM.  A smoother is used  to handle rough data in the discrete forward and inverse source problems. The techniques developed for the Poisson inverse source problems are different from the FEM analysis in \cite{MR3245124,nair2021conforming}.
\par The inverse source problem corresponding to  general second-order problems is challenging and the ideas in Section~4 have to be modified appropriately. For instance, Theorem~\ref{thm:meas} utilizes the symmetry of the bilinear form $a(\cdot,\cdot)$ and hence the auxiliary problem needs to be modified. Also the ideas developed in this article are fairly general and extension to higher-order problems and general boundary conditions with new companion operators is a future work. 
\subsection*{Acknowledgements}
Neela Nataraj and Nitesh Verma gratefully acknowledge the funding from the SERB POWER Fellowship SPF/2020/000019.

\bibliographystyle{amsplain}
\bibliography{references}
\end{document}